\documentclass{amsart}
\usepackage{amssymb}

\usepackage{amscd}
\usepackage{thmdefs}

\theoremstyle{definition}
\theoremstyle{remark}
\numberwithin{equation}{section}

\begin{document}
\title[A Theta operator]{A Theta operator on Picard modular forms modulo an inert prime}
\author{Ehud de Shalit}
\address{Hebrew University, Jerusalem, Israel}
\email{deshalit@math.huji.ac.il.}
\author{Eyal Z. Goren}
\address{McGill University, Montreal, Canada}
\email{eyal.goren@mcgill.ca}
\date{April 14, 2015}
\maketitle

\begin{equation*}
\text{\emph{To the memory of Robert Coleman}}
\end{equation*}

\bigskip

H.P.F. Swinnerton-Dyer [SwD] and J.-P. Serre [Se] introduced a certain
differential operator $\theta $ on (elliptic) modular forms over $\Bbb{\bar{F%
}}_{p}.$ In terms of the $q$-expansion 
\begin{equation}
f=\sum_{n=0}^{\infty }a_{n}q^{n}
\end{equation}
($a_{n}\in \Bbb{\bar{F}}_{p}$) of such a form, $\theta $ is given by $qd/dq.$
It lifts, by the same formula, to the space of $p$-adic modular forms. This
suggests a relation with the Tate twist of the $\mod p$ Galois
representation attached to $f,$ if the latter is a Hecke eigenform.

Over $\Bbb{C},$ this operator has been considered already by Ramanujan,
where it fails to preserve modularity ``by a multiple of $E_{2}".$ Maass
modified it so that modularity is preserved, sacrificing holomorphicity.
Shimura studied Maass' differential operators on more general symmetric
domains, as well as their iterations. They have become known as
Maass-Shimura operators, and play an important role in the theory of
automorphic forms [Sh3, chapter III].

At the same time, Serre's $p$-adic operator has been studied in relation to $%
\mod p$ Galois representations, congruences between modular forms, $p$%
-adic families of modular forms and $p$-adic $L$-functions. As an example we
cite Coleman's celebrated classicality theorem, asserting that
``overconvergent modular forms of small slope are classical'' [Col]. A key
step in Coleman's original proof of that theorem was the observation that,
although the $p$-adic theta operator did not preserve the space of
overconvergent modular forms, for any $k\ge 0,\,\theta ^{k+1}$ mapped
overconvergent forms of weight $-k$ to overconvergent forms of weight $k+2.$%
\bigskip

Underlying the $p$-adic theory is Katz' geometric approach to the theta
operator, via the Gauss-Manin connection on the de Rham cohomology of the
universal elliptic curve [Ka1] [Ka2]. Broadly speaking, Katz' starting point
is the unit-root splitting of the Hodge filtration in this cohomology over
the ordinary locus. It is supposed to replace the Hodge decomposition over $%
\Bbb{C},$ which can be used to make a \emph{geometric} theory of the $%
C^{\infty }$ operators of Maass-Shimura, thereby explaining their arithmetic
significance. This approach has been adapted successfully to other Shimura
varieties of PEL type, as long as they admit a non-empty ordinary locus in
their characteristic $p$ fiber. For unitary Shimura varieties, this has been
done by Eischen [Ei1] [Ei2], if $p$ \emph{splits} in the quadratic imaginary
field (and the signature is $(n,n)$). B\"{o}cherer and Nagaoka [B-N] defined
theta operators on Siegel modular forms by studying their $q$-expansions.

The assumption that the ordinary locus is non-empty may nevertheless fail.
This is the case, for example, for Picard modular surfaces (associated with
the group $U(2,1)$) modulo a prime $p$ which is inert in the underlying
quadratic imaginary field. In this case the abelian varieties parametrized
by the open dense $\mu $-ordinary stratum [Mo] are not ordinary. More
generally, this happens for Shimura varieties associated with $U(n,m)$ if $%
n\neq m$, and $p$ is inert ([Hi1], Lemma 8.10). Another complication present
in these examples is the fact that modular forms on $U(n,m)$ admit
Fourier-Jacobi (FJ) expansions at the cusps, which are $q$-expansions with 
\emph{theta functions} as coefficients.\bigskip

One of the main goals of this paper is to define the theta operator for
Picard modular surfaces at a good inert prime, and study its properties. To
explain how we overcome the need to consider the unit-root splitting of the
cohomology of the universal abelian variety, let us re-examine the case of
the modular curve $X$ of full level $N\ge 3$ over $\Bbb{Z}_{p}$ ($(p,N)=1$).
We follow an approach of Gross [Gr], see also [An-Go], who extended it to
Hilbert modular varieties. Let $\kappa $ be a fixed algebraic closure of $%
\Bbb{F}_{p},$ and consider the geometric characteristic $p$ fiber $X_{\kappa
}.$ Let $\mathcal{A}$ be the universal elliptic curve over $Y=X\backslash C$
(the complement of the cusps) and let $\mathcal{L}=\omega _{\mathcal{A}/Y}$
be its cotangent bundle at the origin. Then $\mathcal{A}$ extends to a
semi-abelian variety over $X,$ and so does $\mathcal{L}=\omega _{\mathcal{A}%
/X}.$ By definition, a weight $k,$ level $N$ modular form over $%
\kappa ,$ is a global section of $\mathcal{L}^{k}$ over $X_{\kappa },$ i.e. 
\begin{equation}
M_{k}(N;\kappa )=H^{0}(X_{\kappa },\mathcal{L}^{k}).
\end{equation}

Let $X_{\kappa }^{ord}$ be the ordinary locus in $X_{\kappa }.$ Let $\tau
:I\rightarrow X_{\kappa }^{ord}$ be the Igusa curve of level $p$,
classifying (besides the elliptic curve $A$ and level structure classified
by $X_{\kappa }$) embeddings of finite flat group schemes $\iota :\mu
_{p}\hookrightarrow A[p]$. Let 
\begin{equation}
h\in H^{0}(X_{\kappa },\mathcal{L}^{p-1})
\end{equation}
be the Hasse invariant. As the universal $\iota :\mu _{p}\hookrightarrow 
\mathcal{A}[p]$ over $I$ induces an isomorphism 
\begin{equation}
\tau ^{*}\mathcal{L}=\omega _{\mathcal{A}/I}=\omega _{\mathcal{A}[p]/I}%
\overset{\iota ^{*}}{\simeq }\omega _{\mu _{p}/I}=\mathcal{O}_{I},
\end{equation}
the line bundle $\tau ^{*}\mathcal{L}$ is trivialized over $I$ by a
canonical section $a.$ In fact, $a^{p-1}=\tau ^{*}h.$

Now, given a $\kappa $-valued modular form $f\in H^{0}(X_{\kappa },\mathcal{L%
}^{k}),$ we consider its pull-back $\tau ^{*}f$ to $I,$ divide by $a^{k}$ to
get a function on $I,$ and take its differential 
\begin{equation}
\eta _{f}=d(\tau ^{*}f/a^{k})\in \Omega _{I}^{1}.
\end{equation}
The Gauss-Manin connection induces the Kodaira-Spencer isomorphism 
\begin{equation}
KS:\mathcal{L}^{2}\otimes \mathcal{O}(C)^{\vee }\simeq \Omega _{X}^{1}.
\end{equation}
As $\tau $ is \'{e}tale, $\Omega _{I}^{1}=\tau ^{*}\Omega _{X_{\kappa
}^{ord}}^{1}$ and we may pull $KS$ back to a similar isomorphism over $I.$
We can therefore look at 
\begin{equation}
a^{k}\cdot KS^{-1}(\eta _{f}).
\end{equation}
This is a section of $\tau ^{*}(\mathcal{L}^{k+2}\otimes \mathcal{O}%
(C)^{\vee })$ over $I.$ Since we divided and multiplied by the same power of 
$a,$ it descends to $X_{\kappa }^{ord}.$ A calculation shows that it has at
most \emph{simple} poles at the supersingular points $X_{\kappa }^{ss},$ so 
\begin{equation}
\theta (f)=h\cdot a^{k}\cdot KS^{-1}(\eta _{f})
\end{equation}
extends to a global section of $\mathcal{L}^{k+p+1}\otimes \mathcal{O}%
(C)^{\vee }$ over $X_{\kappa },$ i.e. to a cusp-form of weight $k+p+1$ and
level $N$ over $\kappa .$ Note that $\theta (f)$ and $a^{k}\cdot
KS^{-1}(\eta _{f})$ have the same $q$-expansions, since the $q$-expansion of 
$h$ is 1. It can be checked that $\theta $ coincides with the operator
denoted by $A\theta $ in [Ka2].\medskip

The absence of the unit-root splitting from the above-mentioned construction
can be ``explained'' by the use we made of the Igusa curve, which lies over
the ordinary stratum. In the case of Picard modular surfaces at an inert
prime $p,$ it is nevertheless possible to construct an ``Igusa surface''
lying over the $\mu $-ordinary part, even though the ordinary stratum (in
the usual sense) is empty. Our construction of the theta operator is based
on the same procedure, but there are now two automorphic vector bundles to
consider, a line bundle $\mathcal{L}$ and a plane bundle $\mathcal{P}.$ The
Verschiebung homomorphism allows us to project the analogue of $KS^{-1}(\eta
_{f})$ (which is a section of $\mathcal{P}\otimes \mathcal{L}$) to an
appropriate one-dimensional piece.

The resulting operator $\Theta $ enjoys all the desired properties. It has
the right effect on Fourier-Jacobi expansions, extends holomorphically
across the $1$-dimensional supersingular locus, and compares well with the
theta operators on embedded modular curves. The theory of ``theta cycles''
[Joc] even presents a surprise (see \ref{theta cycles}).\bigskip

$
\begin{tabular}{l}
\textbf{Table of Contents} \\ 
1. Background \\ 
2. Picard modular schemes modulo an inert prime \\ 
3. Modular forms modulo $p$ and the theta operator \\ 
4. Further results on $\Theta $ \\ 
5. The Igusa tower and $p$-adic modular forms
\end{tabular}
$

\bigskip

Let us now review the contents of the paper in more detail. We denote by $%
\mathcal{K}$ a quadratic imaginary field and by $\bar{S}$ a compactified
integral model of the Picard modular surface of full level $N\ge 3,$
associated with $\mathcal{K}.$ The surface $\bar{S}$ is defined over $R_{0}=%
\mathcal{O}_{\mathcal{K}}[1/2D_{\mathcal{K}}N]$ and we may consider its
reduction modulo the prime $p,$ which is assumed to be relatively prime to $%
2N$ and inert in $\mathcal{K}.$ For simplicity, fix an algebraic closure $%
\kappa $ of $R_{0}/pR_{0}$ and consider the geometric fiber $\bar{S}_{\kappa
}=\bar{S}\times _{Spec(R_{0})}Spec(\kappa ).$ Let $\mathcal{A}$ be the
universal semi-abelian variety over $\bar{S}.$ It is relatively $3$%
-dimensional, has complex multiplication by $\mathcal{O}_{\mathcal{K}},$ and
the cotangent bundle at the origin, $\omega _{\mathcal{A}/\bar{S}}$, is of
type $(2,1).$ This means that if $\Sigma :\mathcal{O}_{\mathcal{K}%
}\hookrightarrow R_{0}$ is the canonical embedding and $\bar{\Sigma}$ its
complex conjugate, then 
\begin{equation}
\omega _{\mathcal{A}/\bar{S}}=\mathcal{P}\oplus \mathcal{L},
\end{equation}
where $\mathcal{P}=\omega _{\mathcal{A}/\bar{S}}(\Sigma )$ is a plane bundle
on which $\mathcal{O}_{\mathcal{K}}$ acts via $\Sigma ,$ and $\mathcal{L}%
=\omega _{\mathcal{A}/\bar{S}}(\bar{\Sigma})$ is a line bundle on which it
acts via $\bar{\Sigma}.$ Scalar modular forms of weight $k\ge 0$ defined
over an $R_{0}$-algebra $R$ are by definition elements of 
\begin{equation}
M_{k}(N;R):=H^{0}(\bar{S}_{R},\mathcal{L}^{k}).
\end{equation}

Our main interest is in $R=\kappa .$ In this case there are homomorphisms of
vector bundles $V_{\mathcal{P}}:\mathcal{P}\rightarrow \mathcal{L}^{(p)}$
and $V_{\mathcal{L}}:\mathcal{L}\rightarrow \mathcal{P}^{(p)}$ deduced from
the Verschiebung homomorphism. Here, for any vector bundle $\mathcal{V}$
over $\bar{S}_{\kappa },$ $\mathcal{V}^{(p)}$ stands for its base-change
under the absolute Frobenius morphism of degree $p,\,\Phi :\bar{S}_{\kappa
}\rightarrow \bar{S}_{\kappa }.$ The \emph{Hasse invariant} is the map 
\begin{equation}
h_{\bar{\Sigma}}=V_{\mathcal{P}}^{(p)}\circ V_{\mathcal{L}}:\mathcal{L}%
\rightarrow \mathcal{L}^{(p^{2})}.
\end{equation}
Since $\mathcal{L}$ is a line bundle, $\mathcal{L}^{(p)}\simeq \mathcal{L}%
^{p}$, so $h_{\bar{\Sigma}}\in H^{0}(\bar{S}_{\kappa },\mathcal{L}%
^{p^{2}-1}) $ is a modular form of weight $p^{2}-1$ over $\kappa .$ The
divisor of $h_{\bar{\Sigma}}$ is precisely the supersingular locus $%
S_{ss}\subset \bar{S}_{\kappa }.$ This is a reduced $1$-dimensional closed
subscheme whose geometric points $x$ are characterized by the fact that $%
\mathcal{A}_{x}$ is supersingular (the Newton polygon of its $p$-divisible
group has constant slope $1/2$). The structure of $S_{ss}$ has been
determined by Vollaard [V], following work of B\"{u}ltel and Wedhorn
[Bu-We]. Its irreducible components are curves whose normalizations are all
isomorphic to the Fermat curve of degree $p+1$. (If $N$ is large enough,
depending on $p,$ these components are even non-singular.) They intersect
transversally at finitely many points, which form the singular locus of $%
S_{ss}.$ This singular locus is also the superspecial locus $S_{ssp}$ in $%
\bar{S}_{\kappa },$ characterized by the fact that $x\in S_{ssp}$ if and
only if $\mathcal{A}_{x}$ is isomorphic to a product of three supersingular
elliptic curves. At $x\in S_{ssp}$ the maps $V_{\mathcal{P}}$ and $V_{%
\mathcal{L}}$ vanish, but over the \emph{general} supersingular locus $%
S_{gss}=S_{ss}\backslash S_{ssp}$ they are both of rank $1.$ The complement
of $S_{ss}$ in $\bar{S}_{\kappa }$ is the dense, open $\mu $-ordinary locus $%
\bar{S}_{\mu }.$ Over a $\mu $-ordinary point which does not belong to a
cuspidal component, the $p$-divisible group of $\mathcal{A}_{x}$ is a
product of a height 2 group of multiplicative type, a height 2 group of
local-local type, and a height 2 \'{e}tale group (all stable under $\mathcal{%
O}_{\mathcal{K}}$). See [dS-G] and Section \ref{p-div}.\bigskip

Section 1 is a rather thorough introduction to Picard modular surfaces and
modular forms, that will serve us also in future work. Occasionally (e.g.
when we compute the Gauss-Manin connection in the complex model), we could
not find a reference for the results in the form that was needed. We
preferred to work them out from scratch, rather than embark on a tedious
translation of notation. This section benefitted in several places from the
excellent exposition in Bella\"{i}che's thesis [Bel].\medskip

In Section 2 we review the geometry of $\bar{S}$ and the automorphic vector
bundles $\mathcal{P}$ and $\mathcal{L}$ modulo an inert prime $p$. Here we
follow [Bu-We] and [V], and the exposition in [dS-G]. We construct the Igusa
surface of level $p$. It is a finite \'{e}tale Galois cover 
\begin{equation}
\tau :\bar{I}g_{\mu }\rightarrow \bar{S}_{\mu }
\end{equation}
of the $\mu $-ordinary part in $\bar{S}_{\kappa },$ with Galois group $%
\Delta (p)=(\mathcal{O}_{\mathcal{K}}/p\mathcal{O}_{\mathcal{K}})^{\times }.$
We prove that it is relatively irreducible, and compactify it over the
supersingular locus to get a normal surface $\bar{I}g$, finite and flat over 
$\bar{S}_{\kappa },$ which is totally ramified over $S_{ss}.$ The Hasse
invariant has a tautological $p^{2}-1$ root $a$ over the whole of $\bar{I}g.$
Thus $a\in H^{0}(\bar{I}g,\tau ^{*}\mathcal{L})$ and $a^{p^{2}-1}=\tau
^{*}h_{\bar{\Sigma}}.\medskip $

In Section 3 we construct the theta operator. We pull back $f\in H^{0}(\bar{S%
}_{\kappa },\mathcal{L}^{k})$ to $\bar{I}g_{\mu },$ divide by the
non-vanishing section $a^{k}$ to get a function, and let 
\begin{equation}
\eta _{f}=d(\tau ^{*}f/a^{k})\in H^{0}(\bar{I}g_{\mu },\Omega ^{1}).
\end{equation}

The Kodaira-Spencer isomorphism over $S$ is an isomorphism of rank two
vector bundles 
\begin{equation}
KS:\mathcal{P}\otimes \mathcal{L}\simeq \Omega _{S}^{1}.
\end{equation}
When we try to extend it to $\bar{S}$ we find out that it has a pole along
the cuspidal divisor $C=\bar{S}\backslash S.$ Nevertheless, in the
characteristic $p$ fiber, the map 
\begin{equation}
(V_{\mathcal{P}}\otimes 1)\circ KS^{-1}:\Omega _{S_{\kappa }}^{1}\rightarrow 
\mathcal{L}^{(p)}\otimes \mathcal{L}=\mathcal{L}^{p+1}
\end{equation}
extends holomorphically across $C$, and even acquires a simple zero there.
We pull it back from $\bar{S}_{\mu }$ to $\bar{I}g_{\mu }$ under the
\'{e}tale map $\tau ,$ and define 
\begin{equation}
\Theta (f)=a^{k}\cdot (V_{\mathcal{P}}\otimes 1)\circ KS^{-1}(\eta _{f})\in
H^{0}(\bar{I}g_{\mu },\tau ^{*}\mathcal{L}^{k+p+1}).
\end{equation}
Thanks to the fact that we have multiplied by $a^{k},$ this section descends
to $\bar{S}_{\mu }$. A pleasant computation reveals that $\Theta (f)$ has no
poles along $S_{ss}.$ We end up with 
\begin{equation}
\Theta (f)\in H^{0}(\bar{S}_{\kappa },\mathcal{L}^{k+p+1}),
\end{equation}
a weight $k+p+1,$ level $N$ modular form over $\kappa .$

It is curious to note that in the case of modular curves, $a^{k}\cdot
KS^{-1}(\eta _{f})$ was of weight $k+2,$ but had poles at the supersingular
points, and only $\theta (f)=h\cdot a^{k}\cdot KS^{-1}(\eta _{f})$ extended
holomorphically to a weight $k+p+1$ modular form. Here, the projection $V_{%
\mathcal{P}}$ takes care of the shift by $p+1$ in the weight, and at the
same time reduces the order of the pole along $Ig_{ss}=\bar{I}g\backslash 
\bar{I}g_{\mu },$ so that $\Theta (f)$ becomes holomorphic over the whole
surface.

The ultimate justification for our construction comes when we compute the
effect of $\Theta $ on Fourier-Jacobi expansions, which is essentially a
``Tate twist''. The computation uses both $p$-adic and complex formalisms.
It may be possible to perform it entirely on the ``Mumford-Tate object''
(see Section 4.5 of [Lan] and [Ei1]), but we believe that our approach has
its own didactical merit.\medskip

In Section 4 we compare our theta operator with theta operators on embedded
modular curves. We also discuss theta cycles and filtrations on modular
forms $\mod p.\medskip $

Section 5 brings up $p$-adic modular forms in the style of Serre and Katz.
The study of overconvergent forms, intimately connected with the study of
the canonical subgroup and Coleman's classicality theorem, will be the
subject of another paper.

Many of the results of this paper, including the construction of the theta
operator, generalize to unitary Shimura varieties associated with $U(n-1,1)$
for general $n.$ Another direction in which the set-up could be generalized
is to replace $\mathcal{K}$ by an arbitrary CM field. This seems to require
substantial additional work, apart from a heavy load of notation, even if
the general lay-out would be the same. We refer the reader to [Hs] for a
detailed discussion of some of the topics treated here over general CM
fields (albeit for a split prime $p$). 

\section{Background}

\subsection{The unitary group and its symmetric space}

\subsubsection{Notation}

Let $\mathcal{K}$ be an imaginary quadratic field, contained in $\Bbb{C}.$
We denote by $\Sigma :\mathcal{K}\hookrightarrow \Bbb{C}$ the inclusion and
by $\bar{\Sigma}:\mathcal{K}\hookrightarrow \Bbb{C}$ its complex conjugate.
We use the following notation:

\begin{itemize}
\item  $d_{\mathcal{K}}$ - the square free integer such that $\mathcal{K}=%
\Bbb{Q}(\sqrt{d_{\mathcal{K}}}).$

\item  $D_{\mathcal{K}}$ - the discriminant of $\mathcal{K}$, equal to $d_{%
\mathcal{K}}$ if $d_{\mathcal{K}}\equiv 1\mod 4$ and $4d_{\mathcal{K}}$
if $d_{\mathcal{K}}\equiv 2,3\mod 4.$

\item  $\delta _{\mathcal{K}}=\sqrt{D_{\mathcal{K}}}$ - the square root with
positive imaginary part, a generator of the different of $\mathcal{K},$
sometimes simply denoted $\delta .$

\item  $\omega _{\mathcal{K}}=(1+\sqrt{d_{\mathcal{K}}})/2$ if $d_{\mathcal{K%
}}\equiv 1\mod 4,$ otherwise $\omega _{\mathcal{K}}=\sqrt{d_{\mathcal{K}%
}},$ so that $\mathcal{O}_{\mathcal{K}}=\Bbb{Z}+\Bbb{Z}\omega _{\mathcal{K}}.
$

\item  $\bar{a}$ - the complex conjugate of $a\in \mathcal{K}.$

\item  $\text{Im}_{\delta }(a)=(a-\bar{a})/\delta $, for $a\in \mathcal{K}.$
\end{itemize}

We fix an integer $N\ge 3$ (the ``tame level'') and let $R_{0}=\mathcal{O}_{%
\mathcal{K}}[1/(2d_{\mathcal{K}}N)].$ This is our \emph{base ring}. If $R$
is any $R_{0}$-algebra and $M$ is any $R$-module with $\mathcal{O}_{\mathcal{%
K}}$-action, then $M$ becomes an $\mathcal{O}_{\mathcal{K}}\otimes R$-module
and we have a canonical \emph{type decomposition} 
\begin{equation}
M=M(\Sigma )\oplus M(\bar{\Sigma})
\end{equation}
where $M(\Sigma )=e_{\Sigma }M$ and $M(\bar{\Sigma})=e_{\bar{\Sigma}}M,$ and
where the idempotents $e_{\Sigma }$ and $e_{\bar{\Sigma}}$ are defined by 
\begin{equation}
e_{\Sigma }=\frac{1\otimes 1}{2}+\frac{\delta \otimes \delta ^{-1}}{2}%
,\,\,\,\,e_{\bar{\Sigma}}=\frac{1\otimes 1}{2}-\frac{\delta \otimes \delta
^{-1}}{2}.
\end{equation}
Then $M(\Sigma )$ (resp. $M(\bar{\Sigma})$) is the part of $M$ on which $%
\mathcal{O}_{\mathcal{K}}$ acts via $\Sigma $ (resp. $\bar{\Sigma}$). The
same notation will be used for sheaves of modules on $R$-schemes, endowed
with an $\mathcal{O}_{\mathcal{K}}$ action. If $M$ is locally free, we say
that it has \emph{type }$(p,q)$ if $M(\Sigma )$ is of rank $p$ and $M(\bar{%
\Sigma})$ is of rank $q.$

We denote by 
\begin{equation}
\mathbf{T}=res_{\Bbb{Q}}^{\mathcal{K}}\Bbb{G}_{m}
\end{equation}
the non-split torus whose $\Bbb{Q}$-points are $\mathcal{K}^{\times },$ and
by $\rho $ the non-trivial automorphism of $\mathbf{T}$, which on $\Bbb{Q}$%
-points induces $\rho (a)=\bar{a}.$ The group $\Bbb{G}_{m}$ embeds in $%
\mathbf{T}$ and the homomorphism $a\mapsto a\cdot \rho (a)$ from $\mathbf{T}$
to itself factors through a homomorphism 
\begin{equation}
N:\mathbf{T}\rightarrow \Bbb{G}_{m},
\end{equation}
the \emph{norm homomorphism}. Its kernel $\ker (N)$ is denoted $\mathbf{T}%
^{1}.$

\subsubsection{The unitary group}

Let $V=\mathcal{K}^{3}$ and endow it with the hermitian pairing 
\begin{equation}
(u,v)=\,^{t}\bar{u}\left( 
\begin{array}{lll}
&  & \delta ^{-1} \\ 
& 1 &  \\ 
-\delta ^{-1} &  & 
\end{array}
\right) v.
\end{equation}
We identify $V_{\Bbb{R}}$ with $\Bbb{C}^{3}$ ($\mathcal{K}$ acting via the
natural inclusion $\Sigma $). It then becomes a hermitian space of signature 
$(2,1).$ Conversely, any $3$-dimensional hermitian space over $\mathcal{K}$
whose signature at the infinite place is $(2,1)$ is isomorphic to $V$ after
rescaling the hermitian form by a positive rational number.

Let 
\begin{equation}
\mathbf{G=GU}(V,(,))
\end{equation}
be the general unitary group of $V,$ regarded as an algebraic group over $%
\Bbb{Q}.$ For any $\Bbb{Q}$-algebra $A$ we have 
\begin{equation}
\mathbf{G}(A)=\left\{ (g,\mu )\in GL_{3}(A\otimes \mathcal{K})\otimes
A^{\times }|\,\,(gu,gv)=\mu \cdot (u,v)\,\,\forall u,v\in V_{A}\right\} .
\end{equation}

We write $G=\mathbf{G}(\Bbb{Q}),$ $G_{\infty }=\mathbf{G}(\Bbb{R})$ and $%
G_{p}=\mathbf{G}(\Bbb{Q}_{p}).$ A similar notational convention will apply
to any algebraic group over $\Bbb{Q}$ without further ado. If $p$ splits in $%
\mathcal{K}$, $\Bbb{Q}_{p}\otimes \mathcal{K}\simeq \Bbb{Q}_{p}^{2}$ and $%
G_{p}$ becomes isomorphic to $GL_{3}(\Bbb{Q}_{p})\times \Bbb{Q}_{p}^{\times
}.$ The isomorphism depends on the embedding of $\mathcal{K}$ in $\Bbb{Q}%
_{p},$ i.e. on the choice of a prime above $p$ in $\mathcal{K}.$ For a
non-split prime $p$ the group $G_{p},$ like $G_{\infty },$ is of
(semisimple) rank 1.

As $\mu $ is determined by $g$ we often abuse notation and write $g$ for the
pair $(g,\mu )$ and $\mu (g)$ for the \emph{similitude factor (multiplier)} $%
\mu .$ It is a character of algebraic groups over $\Bbb{Q},\,\mu :\mathbf{%
G\rightarrow }\Bbb{G}_{m}.$ Another character is $\det :\mathbf{G\rightarrow
T},$ defined by $\det (g,\mu )=\det (g).$ If we let 
\begin{equation}
\nu =\mu ^{-1}\cdot \det :\mathbf{G}\rightarrow \mathbf{T}
\end{equation}
then both $\mu $ and $\det $ are expressible in terms of $\nu ,$ namely $\mu
=\nu \cdot (\rho \circ \nu )$ and $\det =\nu ^{2}\cdot (\rho \circ \nu ).$

The groups 
\begin{equation}
\mathbf{U}=\ker \mu ,\,\,\,\mathbf{SU}=\ker \nu =\ker \mu \cap \ker (\det )
\end{equation}
are the unitary and the special unitary group, respectively.

We also introduce an alternating $\Bbb{Q}$-linear pairing $\left\langle
,\right\rangle :V\times V\rightarrow \Bbb{Q}$ (the \emph{polarization form})
defined by $\left\langle u,v\right\rangle =\text{Im}_{\delta }(u,v).$ We
then have the formulae 
\begin{equation}
\left\langle au,v\right\rangle =\left\langle u,\bar{a}v\right\rangle
,\,\,\,\,\,\,2(u,v)=\left\langle u,\delta v\right\rangle +\delta
\left\langle u,v\right\rangle .
\end{equation}

\subsubsection{The hermitian symmetric domain}

The group $G_{\infty }=\mathbf{G}(\Bbb{R})$ acts on $\Bbb{P}_{\Bbb{C}}^{2}=%
\Bbb{P}(V_{\Bbb{R}})$ by projective linear transformations and preserves the
open subdomain $\frak{X}$ of negative definite lines (in the metric $(,)$),
which is biholomorphic to the open unit ball in $\Bbb{C}^{2}$. Every
negative definite line is represented by a unique vector $^{t}(z,u,1),$ and
such a vector represents a negative definite line if and only if 
\begin{equation}
\lambda (z,u)\overset{\text{def}}{=}\vspace{0.01in}\mathrm{Im}_{\delta }(z)-u%
\bar{u}>0.
\end{equation}
One refers to the realization of $\frak{X}$ as the set of points $(z,u)\in 
\Bbb{C}^{2}$ satisfying this inequality as a \emph{Siegel domain of the
second kind}. It is convenient to think of the point $x_{0}=$ $(\delta /2,0)$
as the ``center'' of $\frak{X.}$

If we let $K_{\infty }$ be the stabilizer of $x_{0}$ in $G_{\infty }$, then $%
K_{\infty }$ is compact modulo center ($K_{\infty }\cap \mathbf{U}(\Bbb{R})$
is compact and isomorphic to $U(2)\times U(1)$). Since $G_{\infty }$ acts
transitively on $\frak{X},$ we may identify $\frak{X}$ with $G_{\infty
}/K_{\infty }.$

The usual upper half plane embeds in $\frak{X}$ as the set of points where $%
u=0.$

\subsubsection{The cusps of $\frak{X}$}

The boundary $\partial \frak{X}$ of $\frak{X}$ is the set of points $(z,u)$
where $\text{Im}_{\delta }(z)=u\bar{u},$ together with a unique point ``at
infinity'' $c_{\infty }$ represented by the line $^{t}(1:0:0).$ The lines
represented by $\partial \frak{X}$ are the isotropic lines in $V_{\Bbb{R}}.$
The set of \emph{cusps} $\mathcal{C}\frak{X}$ is the set of $\mathcal{K}$%
-rational isotropic lines. If $s\in \mathcal{K}$ and $r\in \Bbb{Q}$ we write 
\begin{equation}
c_{s}^{r}=(r+\delta s\bar{s}/2,s).
\end{equation}
Then $\mathcal{C}\frak{X}=\left\{ c_{s}^{r}|r\in \Bbb{Q},s\in \mathcal{K}%
\right\} \cup \{c_{\infty }\}.$ The group $G=\mathbf{G}(\Bbb{Q})$ acts
transitively on the cusps.

The stabilizer of a cusp is a Borel subgroup in $G_{\infty }.$ Since $G$
acts transitively on the cusps, we may assume that our cusp is $c_{\infty }.$
It is then easy to check that its stabilizer $P_{\infty }$ has the form $%
P_{\infty }=M_{\infty }N_{\infty },$ where 
\begin{equation}
M_{\infty }=\left\{ tm(\alpha ,\beta )=t\left( 
\begin{array}{lll}
\alpha &  &  \\ 
& \beta &  \\ 
&  & \bar{\alpha}^{-1}
\end{array}
\right) |\,t\in \Bbb{R}_{+}^{\times },\,\alpha \in \Bbb{C}^{\times },\beta
\in \Bbb{C}^{1}\right\} ,
\end{equation}
\begin{equation}
N_{\infty }=\left\{ n(u,r)=\left( 
\begin{array}{lll}
1 & \delta \bar{u} & r+\delta u\bar{u}/2 \\ 
& 1 & u \\ 
&  & 1
\end{array}
\right) |\,u\in \Bbb{C},r\in \Bbb{R}\right\} .
\end{equation}
The matrix $tm(\alpha ,\beta )$ belongs to $U_{\infty }$ if and only if $%
t=1, $ and to $SU_{\infty }$ if furthermore $\beta =\bar{\alpha}/\alpha .$
The group $N_{\infty }$ is contained in $SU_{\infty }.$ Since $N=N_{\infty
}\cap G$ still acts transitively on the set of finite\emph{\ }cusps $%
c_{s}^{r}$, we conclude that $G$ acts \emph{doubly transitively} on $%
\mathcal{C}\frak{X}. $

Of particular interest to us will be the \emph{geodesics} connecting an
interior point $(z,u)$ to a cusp $c\in \mathcal{C}\frak{X}.$ If $%
(z,u)=n(u,r)m(d,1)x_{0}$ (recall $x_{0}=\,^{t}(\delta /2:0:1)$) where $d$ is
real and positive (i.e. $r=\Re z$ and $d=\sqrt{\lambda (z,u)}$) then
the geodesic connecting $(z,u)$ to $c_{\infty }$ can be described by the
formula 
\begin{eqnarray}
\gamma _{u}^{r}(t) &=&n(u,r)m(t,1)x_{0}  \label{geodesic} \\
&=&(r+\delta (u\bar{u}+t^{2})/2,u)\text{ \thinspace \thinspace (}d\le
t<\infty \text{).}  \notag
\end{eqnarray}
The same geodesic extends in the opposite direction for $0<t\le d,$ and if $%
u $ and $r$ lie in $\mathcal{K},$ it ends there in the cusp $c_{u}^{r}.$ We
shall call $\gamma _{u}^{r}(t)$ the \emph{geodesic retraction} of $\frak{X}$
to the cusp $c_{\infty }.$ As $0<t<\infty $ these parallel geodesics exhaust 
$\frak{X}.$

\subsection{Picard modular surfaces over $\Bbb{C}$}

\subsubsection{Lattices and their arithmetic groups\label{lattices}}

Fix an $\mathcal{O}_{\mathcal{K}}$-invariant lattice $L\subset V$ which is 
\emph{self-dual} in the sense that 
\begin{equation}
L=\left\{ u\in V|\,\left\langle u,v\right\rangle \in \Bbb{Z\,\,\forall }v\in
L\right\} .
\end{equation}
Equivalently, $L$ is its own $\mathcal{O}_{\mathcal{K}}$-dual with respect
to the hermitian pairing $(,).$ We assume also that the Steinitz class%
\footnote{%
The Steinitz class of a finite projective $\mathcal{O}_{\mathcal{K}}$-module
is the class of its top exterior power as an invertible module.} of $L$ as
an $\mathcal{O}_{\mathcal{K}}$-module is $[\mathcal{O}_{\mathcal{K}}],$ or,
what amounts to the same, that $L$ is a free $\mathcal{O}_{\mathcal{K}}$%
-module. When we introduce the Shimura variety later on, we shall relax this
last assumption, but the resulting scheme will be disconnected (over $\Bbb{C}
$).

Fix an integer $N\ge 1$ and let 
\begin{equation}
\Gamma =\Gamma _{L}(N)=\left\{ g\in G|\,gL=L\,\text{and }g(u)\equiv u\
mod NL\,\,\forall u\in L\right\} ,
\end{equation}
a discrete subgroup of $G_{\infty }.$ It is easy to see that if $N\ge 3$
then $\Gamma $ is torsion free, acts freely and faithfully on $\frak{X},$
and is contained in $SU_{\infty }.$ From now on we assume that this is the
case.

If $g\in G$ and $\mu (g)=1$ (i.e. $g\in U$) the lattice $gL$ is another
lattice of the same sort and the discrete group corresponding to it is $%
g\Gamma g^{-1}.$ Since $U$ acts transitively on the cusps, this reduces the
study of $\Gamma \backslash \frak{X}$ near a cusp to the study of a
neighborhood of the standard cusp $c_{\infty }$ (at the price of changing $L$
and $\Gamma $).

It is important to know the classification of lattices $L$ as above
(self-dual and $\mathcal{O}_{\mathcal{K}}$-free). Let $e_{1},e_{2},e_{3}$ be
the standard basis of $\mathcal{K}^{3}.$ Let 
\begin{equation}
L_{0}=Span_{\mathcal{O}_{\mathcal{K}}}\{\delta e_{1},e_{2},e_{3}\}
\end{equation}
and 
\begin{equation}
L_{1}=Span_{\mathcal{O}_{\mathcal{K}}}\{\frac{\delta }{2}e_{1}+e_{3},e_{2},%
\frac{\delta }{2}e_{1}-e_{3}\}.
\end{equation}
These two lattices are self-dual and, of course, $\mathcal{O}_{\mathcal{K}}$%
-free. The following theorem is based on the local-global principle and a
classification of lattices over $\Bbb{Q}_{p}$ by Shimura [Sh1].

\begin{lemma}
\label{classification of lattices}([La1], p.25). For any lattice $L$ as
above there exists a $g\in U$ such that $gL=L_{0}$ or $gL=L_{1}.$ If $D_{%
\mathcal{K}}$ is odd, $L_{0}$ and $L_{1}$ are equivalent. If $D_{\mathcal{K}}
$ is even, they are inequivalent.
\end{lemma}

Indeed, if $D_{\mathcal{K}}$ is even, $L_{0}\otimes \Bbb{Q}_{p}$ and $%
L_{1}\otimes \Bbb{Q}_{p}$ are $U_{p}$-equivalent for every $p\neq 2$, but
not for $p=2.$

\subsubsection{Picard modular surfaces and the Baily-Borel compactification}

We denote by $X_{\Gamma }$ the complex surface $\Gamma \backslash \frak{X.}$
Since the action of $\Gamma $ is free, $X_{\Gamma }$ is smooth. We describe
a topological compactification of $X_{\Gamma }.$ A \emph{standard
neighborhood} of the cusp $c_{\infty }$ in $\frak{X}$ is an open set of the
form 
\begin{equation}
\Omega _{R}=\left\{ (z,u)|\lambda (z,u)>R\right\} .
\end{equation}
The set $\mathcal{C}_{\Gamma }=\Gamma \backslash \mathcal{C}\frak{X}$ is
finite, and we write $c_{\Gamma }=\Gamma c.$ We let $X_{\Gamma }^{*}$ be the
disjoint union of $X_{\Gamma }$ and $\mathcal{C}_{\Gamma }.$ Let $\Gamma
_{cusp}$ be the stabilizer of $c_{\infty }$ in $\Gamma .$ We topologize $%
X_{\Gamma }^{*}$ by taking $\Gamma _{cusp}\backslash \Omega _{R}\cup
\{c_{\infty ,\Gamma }\}$ as a basis of neighborhoods at $c_{\infty ,\Gamma
}. $ If $c=g(c_{\infty })$ where $g\in U,$ we take $g(g^{-1}\Gamma
_{cusp}g\backslash \Omega _{R})\cup \{c_{\Gamma }\}$ instead. The following
theorem is well-known.

\begin{theorem}
(Satake, Baily-Borel) $X_{\Gamma}^{*}$ is projective and the singularities
at the cusps are normal. In other words, there exists a normal complex
projective surface $S_{\Gamma}^{*}$ and a homeomorphism 
$\iota :S_{\Gamma}^{*}(\Bbb{C})\simeq X_{\Gamma}^{*},$ which on $S_{\Gamma 
}(\Bbb{C})=\iota ^{-1}(X_{\Gamma })$ is an isomorphism of complex
manifolds. $S_{\Gamma}^{*}$ is uniquely determined up to isomorphism.
\end{theorem}

\subsubsection{The universal abelian variety over $X_{\Gamma }$}

With $x\in \frak{X}$ and with our choice of $L$ we shall now associate a
PEL-structure $\underline{A}_{x}=(A_{x},\lambda _{x},\iota _{x},\alpha _{x})$
where

\begin{enumerate}
\item  $A_{x}$ is a $3$-dimensional complex abelian variety,

\item  $\lambda _{x}$ is a principal polarization on $A_{x}$ (i.e. an
isomorphism $A_{x}\simeq A_{x}^{t}$ with its dual abelian variety induced by
an ample line bundle),

\item  $\iota _{x}:\mathcal{O}_{\mathcal{K}}\hookrightarrow End(A_{x})$ is
an embedding of CM type $(2,1)$ (i.e. the action of $\iota (a)$ on the
tangent space of $A_{x}$ at the origin induces the representation $2\Sigma +%
\bar{\Sigma}$) such that the Rosati involution induced by $\lambda _{x}$
preserves $\iota (\mathcal{O}_{\mathcal{K}})$ and is given by $\iota
(a)\mapsto \iota (\bar{a}),$

\item  $\alpha _{x}:N^{-1}L/L\simeq A_{x}[N]$ is a full level $N$ structure,
compatible with the $\mathcal{O}_{\mathcal{K}}$-action and the polarization.
The latter condition means that if we denote by $\left\langle ,\right\rangle
_{\lambda }$ the Weil ``$e_{N}$-pairing'' on $A_{x}[N]$ induced by $\lambda
_{x},$ then for $l,l^{\prime }\in N^{-1}L$%
\begin{equation}
\left\langle \alpha _{x}(l),\alpha _{x}(l^{\prime })\right\rangle _{\lambda
}=e^{2\pi iN\left\langle l,l^{\prime }\right\rangle }.
\end{equation}
\end{enumerate}

Let $W_{x}$ be the negative definite complex line in $V_{\Bbb{R}}=\Bbb{C}%
^{3} $ defined by $x,$ and $W_{x}^{\perp }$ its orthogonal complement, a
positive definite plane. Let $J_{x}$ be the complex structure which is
multiplication by $i$ on $W_{x}^{\perp }$ and by $-i$ on $W_{x}.$ Let $%
A_{x}=(V_{\Bbb{R}},J_{x})/L.$ Then the polarization form $\langle ,\rangle $
is a Riemann form on $L,$ which determines a principal polarization on $%
A_{x} $ as usual. The action of $\mathcal{O}_{\mathcal{K}}$ is derived from
the underlying $\mathcal{K}$ structure of $V.$ As we have reversed the
complex structure on $W_{x},$ the CM type is now (2,1). Finally the level $N$
structure $\alpha _{x}$ is the identity map.

If $\gamma \in \Gamma $ then $\gamma $ induces an isomorphism between $%
\underline{A}_{x}$ and $\underline{A}_{\gamma (x)}.$ Conversely, if $%
\underline{A}_{x}$ and $\underline{A}_{x^{\prime }}$ are isomorphic
structures, it is easy to see that $x^{\prime }$ and $x$ must belong to the
same $\Gamma $-orbit. It follows that points of $X_{\Gamma }$ are in a
bijection with PEL structures of the above type for which the triple 
\begin{equation}
(H_{1}(A_{x},\Bbb{Z}),\iota _{x},\left\langle ,\right\rangle _{\lambda _{x}})
\end{equation}
is isomorphic to $(L,\iota ,\left\langle ,\right\rangle )$ (here $\iota $
refers to the $\mathcal{O}_{\mathcal{K}}$ action on $L$), with the further
condition that $\alpha _{x}$ is compatible with the isomorphism between $L$
and $H_{1}(A_{x},\Bbb{Z})$ in the sense that we have a commutative diagram 
\begin{equation}
\begin{array}{lllllllll}
0 & \rightarrow & L & \rightarrow & N^{-1}L & \rightarrow & N^{-1}L/L & 
\rightarrow & 0 \\ 
&  & \downarrow &  & \downarrow &  & \downarrow \alpha _{x} &  &  \\ 
0 & \rightarrow & H_{1}(A_{x},\Bbb{Z}) & \rightarrow & N^{-1}H_{1}(A_{x},%
\Bbb{Z}) & \rightarrow & A_{x}[N] & \rightarrow & 0
\end{array}
.
\end{equation}

\subsubsection{A ``moving lattice'' model for the universal abelian variety%
\label{moving lattice}}

We want to assemble the individual $A_{x}$ into an \emph{abelian variety} $A$
\emph{over} $\frak{X}$. In other words, we want to construct a 5-dimensional
complex manifold $A,$ together with a holomorphic map $A\rightarrow \frak{X}$
whose fiber over $x$ is identified with $A_{x}.$ For that, as well as for
the computation of the Gauss-Manin connection below, it is convenient to
introduce another model, in which the complex structure on $\Bbb{C}^{3}$ is
fixed, but the lattice varies.

For simplicity we assume from now on that $L=L_{0}$ is spanned over $%
\mathcal{O}_{\mathcal{K}}$ by $\delta e_{1},e_{2}$ and $e_{3}.$ The case of $%
L_{1}$ can be handled similarly.

Let $\Bbb{C}^{3}$ be given the usual complex structure, and let $a\in 
\mathcal{O}_{\mathcal{K}}$ act on it via the matrix 
\begin{equation}
\iota ^{\prime }(a)=\left( 
\begin{array}{lll}
a &  &  \\ 
& a &  \\ 
&  & \bar{a}
\end{array}
\right) .
\end{equation}
Given $x=(z,u)\in \frak{X}$ consider the lattice 
\begin{equation}
L_{x}^{\prime }=Span_{\iota ^{\prime }(\mathcal{O}_{\mathcal{K}})}\left\{
\left( 
\begin{array}{l}
0 \\ 
1 \\ 
1
\end{array}
\right) ,\left( 
\begin{array}{l}
-1 \\ 
0 \\ 
-u
\end{array}
\right) ,\left( 
\begin{array}{l}
u \\ 
-z/\delta \\ 
z/\delta
\end{array}
\right) \right\} \subset \Bbb{C}^{3}.  \label{spanning vectors}
\end{equation}
The map $T_{x}:\Bbb{C}^{3}\rightarrow \Bbb{C}^{3}$ which sends $\zeta
=\,^{t}(\zeta _{1},\zeta _{2},\zeta _{3})$ to 
\begin{equation}
T(\zeta )=\lambda (z,u)^{-1}\left\{ -\zeta _{1}\left( 
\begin{array}{l}
\bar{u}z \\ 
(z-\bar{z})/\delta \\ 
\bar{u}
\end{array}
\right) -\zeta _{2}\left( 
\begin{array}{l}
\bar{z}+\delta u\bar{u} \\ 
u \\ 
1
\end{array}
\right) +\bar{\zeta}_{3}\left( 
\begin{array}{l}
z \\ 
u \\ 
1
\end{array}
\right) \right\}
\end{equation}
is a complex linear isomorphism between $\Bbb{C}^{3}$ and $(V_{\Bbb{R}%
},J_{x}).$ In fact, it sends $\Bbb{C}e_{1}+\Bbb{C}e_{2}$ linearly to $%
W_{x}^{\perp }$ and $\Bbb{C}e_{3}$ conjugate-linearly to $W_{x}.$ It
intertwines the $\iota ^{\prime }$ action of $\mathcal{O}_{\mathcal{K}}$ on $%
\Bbb{C}^{3}$ with its $\iota $-action on $(V_{\Bbb{R}},J_{x}).$ It
furthermore sends $L_{x}^{\prime }$ to $L.$ In fact, an easy computation
shows that it sends the three generating vectors of $L_{x}^{\prime }$ to $%
\delta e_{1},e_{2}$ and $e_{3},$ respectively. We conclude that $T_{x}$
induces an isomorphism 
\begin{equation}
T_{x}:A_{x}^{\prime }=\Bbb{C}^{3}/L_{x}^{\prime }\simeq A_{x}.
\label{moving lattice model}
\end{equation}

Consider the differential forms $d\zeta _{1},d\zeta _{2}$ and $d\zeta _{3}.$
As their periods along any $l\in L_{x}^{\prime }$ vary holomorphically in $z$
and $u$, the five coordinates $\zeta _{1},\zeta _{2},\zeta _{3},z,u$ form a
local system of coordinates on the family $A^{\prime }\rightarrow \frak{X}.$
Identifying $A^{\prime }$ with $A$ allows us to put the desired complex
structure on the family $A.$ Alternatively, we may define $A^{\prime }$ as
the quotient of $\Bbb{C}^{3}\times \frak{X}$ by $\zeta \mapsto \zeta +l(z,u)$
where $l(z,u)$ varies over the holomorphic lattice-sections.

The model $A^{\prime }$ has another advantage, that will become clear when
we examine the degeneration of the universal abelian variety at the cusp $%
c_{\infty }.$ It suffices to note at this point that the first two of the
three generating vectors of $L_{x}^{\prime }$ depend only on $u.$

\subsection{The Picard moduli scheme}

\subsubsection{The moduli problem\label{Moduli}}

Fix $N\ge 3$ and $L=L_{0}\subset V$ as before. Let $R$ be an $R_{0}$%
-algebra. Let $\mathcal{M}(R)$ be the collection of (isomorphism classes of)
PEL structures $(A,\lambda ,\iota ,\alpha )$ where

\begin{enumerate}
\item  $A/R$ is an abelian scheme of relative dimension 3

\item  $\lambda :A\simeq A^{t}$ is a principal polarization

\item  $\iota :$ $\mathcal{O}_{\mathcal{K}}\rightarrow End(A/R)$ is a
homomorphism such that (1) $\iota $ makes $Lie(A/R)$ a locally free $R$%
-module of type $(2,1),$ (2) the Rosati involution induced on $\iota (%
\mathcal{O}_{\mathcal{K}})$ by $\lambda $ is $\iota (a)\mapsto \iota (\bar{a}%
).$

\item  $\alpha :N^{-1}L/L\simeq A[N]$ is an isomorphism of $\mathcal{O}_{%
\mathcal{K}}$-group schemes over $R$ which is compatible with the
polarization in the sense that there exists an isomorphism $\nu _{N}:\Bbb{Z}%
/N\Bbb{Z}\simeq \mu _{N}$ of group schemes over $R$ such that 
\begin{equation}
\left\langle \alpha (\frac{l}{N}),\alpha (\frac{l^{\prime }}{N}%
)\right\rangle _{\lambda }=\nu _{N}(\left\langle l,l^{\prime }\right\rangle 
\mod N).  \label{nu_N}
\end{equation}
In addition we require that for every multiple $N^{\prime }$ of $N,$ locally 
\'{e}tale over $Spec(R),$ there exists a similar level $N^{\prime }$%
-structure $\alpha ^{\prime },$ restricting to $\alpha $ on $N^{-1}L/L.$ One
says that $\alpha $ is \emph{locally \'{e}tale symplectic liftable }([Lan],
1.3.6.2).
\end{enumerate}

In view of Lemma \ref{classification of lattices}, the last condition of
symplectic liftability is void if $D_{\mathcal{K}}$ is odd, while if $D_{%
\mathcal{K}}$ is even it is equivalent to the following condition ([Bel],
I.3.1):

\begin{itemize}
\item  For any geometric point $\eta :R\rightarrow k$ ($k$ algebraically
closed field, necessarily of characteristic different from 2), the $\mathcal{%
O}_{\mathcal{K}}\otimes \Bbb{Z}_{2}$ polarized module $(T_{2}A_{\eta
},\left\langle ,\right\rangle _{\lambda })$ is isomorphic to $(L\otimes \Bbb{%
Z}_{2},\left\langle ,\right\rangle )$ under a suitable identification of $%
\lim_{\leftarrow }\mu _{2^{n}}(k)$ with $\Bbb{Z}_{2}.$
\end{itemize}

The choice of $L_{0}$ was arbitrary. If we took $L_{1}$ as our basic lattice
we would get a similar moduli problem.

A level $N$ structure $\alpha $ can exist only if the group schemes $\Bbb{Z}%
/N\Bbb{Z}$ and $\mu _{N}$ become isomorphic over $R,$ but the isomorphism $%
\nu _{N}$ is then determined by $\alpha .$

$\mathcal{M}$ becomes a functor on the category of $R_{0}$-algebras (and
more generally, on the category of $R_{0}$-schemes) in the obvious way. The
following theorem is of fundamental importance ([Lan], I.4.1.11).

\begin{theorem}
The functor $R\mapsto \mathcal{M}(R)$ is represented by a smooth
quasi-projective scheme $S$ over $Spec(R_{0}),$ of relative dimension 2.
\end{theorem}

We call $S$ the \emph{(open) Picard modular surface of level} $N.$ It comes
equipped with a universal structure $(\mathcal{A},\lambda ,\iota ,\alpha )$
of the above type over $S.$ We call $\mathcal{A}$ the \emph{universal
abelian scheme} over $S$. For every $R_{0}$-algebra $R$ and PEL structure in 
$\mathcal{M}(R),$ there exists a unique $R$-point of $S$ such that the given
PEL structure is obtained from the universal one by base-change.

We refer to [Lan], 1.4.3 for the relation between the given formulation of
the moduli problem and other formulations due, e.g. to Kottwitz.

\subsubsection{The Shimura variety $Sh_{K}$}

We briefly recall the interpretation of the Picard modular surface as a
canonical model of a Shimura variety. The symmetric domain $\frak{X}$ can be
interpreted as a $G_{\infty }$-conjugacy class of homomorphisms 
\begin{equation}
h:\Bbb{S}=res_{\Bbb{R}}^{\Bbb{C}}\Bbb{G}_{m}\rightarrow \mathbf{G}_{\Bbb{R}}
\end{equation}
turning $(\mathbf{G},\frak{X})$ into a Shimura datum in the sense of Deligne
[De]. In fact $h_{x}(i)=J_{x}.$ The reflex field associated to this datum
turns out to be $\mathcal{K}.$ Let $K_{\infty }$ be the stabilizer of $x_{0}$
in $G_{\infty }$ and $K_{f}^{0}\subset \mathbf{G}(\Bbb{A}_{f})$ the subgroup
stabilizing $\widehat{L}=L\otimes \widehat{\Bbb{Z}}$. Let $K_{f}$ be the
subgroup of $K_{f}^{0}$ inducing the identity on $L/NL.$ Let $K=K_{\infty
}K_{f}\subset \mathbf{G}(\Bbb{A}).$ Then the Shimura variety $Sh_{K}$ is a
complex quasi-projective variety whose complex points are isomorphic, as a
complex manifold, to the double coset space 
\begin{eqnarray}
Sh_{K}(\Bbb{C}) &\simeq &\mathbf{G}(\Bbb{Q})\backslash \mathbf{G}(\Bbb{A})/K
\\
&=&\mathbf{G}(\Bbb{Q})\backslash (\frak{X}\times \mathbf{G}(\Bbb{A}%
_{f})/K_{f}).  \notag
\end{eqnarray}
The theory of Shimura varieties provides a \emph{canonical model} for $%
Sh_{K} $ over $\mathcal{K}.$ The following important theorem complements the
one on the representability of the functor $\mathcal{M}.$

\begin{theorem}
The canonical model of $Sh_{K}$ is the generic fiber $S_{\mathcal{K}}$ of $S.
$
\end{theorem}

Let us explain only how to associate to a point of $Sh_{K}(\Bbb{C})$ a point
in $S(\Bbb{C}).$ For that we have to associate an element of $\mathcal{M}(%
\Bbb{C})$ to $g\in \mathbf{G}(\Bbb{A}),$ and show that the structures
associated to $g$ and to $\gamma gk$ ($\gamma \in G,k\in K$) are isomorphic.
Let $x=x_{g}=g_{\infty }(x_{0})\in \frak{X}.$ Let $L_{g}=g_{f}(\widehat{L}%
)\cap V$ (the intersection taking place in $V_{\Bbb{A}}=\widehat{L}\otimes 
\Bbb{Q}$) and 
\begin{equation}
A_{g}=(V_{\Bbb{R}},J_{x})/L_{g}.
\end{equation}
Note that $J_{x}$ depends only on $g_{\infty }K_{\infty }$ and $L_{g}$ only
on $g_{f}K_{f}^{0},$ so $A_{g}$ depends only on $gK^{0}.$

Let $\tilde{\mu}(g)$ be the unique positive rational number such that for
every prime $p,$%
\begin{equation}
ord_{p}\tilde{\mu}(g)=ord_{p}\mu (g_{p}).
\end{equation}
Such a rational number exists since $\mu (g_{p})$ is a $p$-adic unit for
almost all $p$ and $\Bbb{Q}$ has class number 1. We claim that 
\begin{equation*}
\left\langle ,\right\rangle _{g}=\tilde{\mu}(g)^{-1}\left\langle
,\right\rangle :L_{g}\times L_{g}\rightarrow \Bbb{Q}
\end{equation*}
induces a principal polarization $\lambda _{g}$ on $A_{g}.$ That this is a
(rational) Riemann form follows from the fact that $(u,v)_{J_{x}}=\left%
\langle u,J_{x}v\right\rangle +i\left\langle u,v\right\rangle $ is hermitian
positive definite. That $\left\langle ,\right\rangle _{g}$ is indeed $\Bbb{Z}
$-valued and $L_{g}$ is self-dual follows from the choice of $\tilde{\mu}(g)$
since locally at $p$ the dual of $g_{p}L_{p}$ under $\left\langle
,\right\rangle :V_{p}\times V_{p}\rightarrow \Bbb{Q}_{p}$ is $\mu
(g_{p})^{-1}g_{p}L_{p}.$ We conclude that there exists a unique polarization 
$\lambda _{g}:A_{g}\rightarrow A_{g}^{t}$ such that 
\begin{equation}
\left\langle u,v\right\rangle _{\lambda _{g}}=\exp (2\pi il\left\langle
u,v\right\rangle _{g})
\end{equation}
for every $u,v\in A_{g}[l]=l^{-1}L_{g}/L_{g}$ and every $l\ge 1.$ This
polarization is principal.

Since $g_{f}$ commutes with the $\mathcal{K}$-structure on $V_{\Bbb{A}},$ $%
L_{g}$ is still an $\mathcal{O}_{\mathcal{K}}$-lattice, hence $\iota _{g}$
is defined.

Finally $\alpha _{g}$ is derived from 
\begin{equation}
N^{-1}L/L=N^{-1}\widehat{L}/\widehat{L}\overset{g_{f}}{\rightarrow }N^{-1}%
\widehat{L}_{g}/\widehat{L}_{g}=N^{-1}L_{g}/L_{g}=A_{g}[N].
\end{equation}
We note that $\alpha _{g}$ depends only on $gK$ because $K_{f}\subset
K_{f}^{0}$ is the principal level-$N$ subgroup, and that it lifts to level $%
N^{\prime }$ structure for any multiple $N^{\prime }$ of $N,$ by the same
formula. The isomorphism $\nu _{N,g}$ between $\Bbb{Z}/N\Bbb{Z}$ and $\mu
_{N}(\Bbb{C})$ that makes (\ref{nu_N}) work is self-evident (see (\ref
{nu_N_again})). Let $\underline{A}_{g}\in \mathcal{M}(\Bbb{C})$ be the
structure just constructed.

Let now $\gamma \in \mathbf{G}(\Bbb{Q}).$ Then the action of $\gamma $ on $V$
induces an isomorphism between the tuples $\underline{A}_{g}$ and $%
\underline{A}_{\gamma g}.$ Indeed, $\gamma :V_{\Bbb{R}}\rightarrow V_{\Bbb{R}%
}$ intertwines the complex structures $x_{g}$ and $x_{\gamma g},$ and
carries $L_{g}$ to $L_{\gamma g},$ so induces an isomorphism of the abelian
varieties, which clearly commutes with the PEL structures.

This shows that $\underline{A}_{g}$ depends solely on the double coset of $g$
in $\mathbf{G(}\Bbb{Q})\backslash \mathbf{G}(\Bbb{A})/K.$ One is left now
with two tasks which we leave out: (i) Proving that if $\underline{A}%
_{g}\simeq \underline{A}_{g^{\prime }}$ then $g$ and $g^{\prime }$ belong to
the same double coset, and that every $\underline{A}\in \mathcal{M}(\Bbb{C})$
is obtained in this way, (ii) Identifying the canonical model of $Sh_{K}$
over $\mathcal{K}$ with $S_{\mathcal{K}}.$

\subsubsection{The connected components of $Sh_{K}\label%
{connected components}$}

Recall that $\mathbf{G}^{\prime }=\mathbf{SU}=\ker (\nu :\mathbf{G}%
\rightarrow \mathbf{T}).$ Since $\mathbf{G}^{\prime }$ is simple and simply
connected, strong approximation holds and 
\begin{equation}
\mathbf{G}^{\prime }(\Bbb{A})=\mathbf{G}^{\prime }(\Bbb{Q})G_{\infty
}^{\prime }K_{f}^{\prime }.
\end{equation}
Here $K^{\prime }=K\cap \mathbf{G}^{\prime }(\Bbb{A})$, $K_{f}^{\prime
}=K\cap \mathbf{G}^{\prime }(\Bbb{A}_{f}).$ From the connectedness of $%
G_{\infty }^{\prime }$ we deduce that 
\begin{equation}
\mathbf{G}^{\prime }(\Bbb{Q})\backslash \mathbf{G}^{\prime }(\Bbb{A}%
)/K^{\prime }
\end{equation}
is connected.

As $N\ge 3,$ $\nu (K)\cap \mathcal{K}^{\times }=\{1\}.$ Here $\mathcal{K}%
^{\times }=\nu (\mathbf{G}(\Bbb{Q})),$ and it follows that 
\begin{equation}
\mathbf{G}^{\prime }(\Bbb{Q})\backslash \mathbf{G}^{\prime }(\Bbb{A}%
)/K^{\prime }\hookrightarrow \mathbf{G}(\Bbb{Q})\backslash \mathbf{G}(\Bbb{A}%
)/K
\end{equation}
is injective. We now claim (see also Theorem 2.4 and 2.5 of [De]) that 
\begin{equation}
\nu :\pi _{0}(\mathbf{G}(\Bbb{Q})\backslash \mathbf{G}(\Bbb{A})/K)\simeq \pi
_{0}(\mathbf{T}(\Bbb{Q})\backslash \mathbf{T}(\Bbb{A})/\nu (K))
\end{equation}
is a bijection. For $\nu $ is surjective ([De] (0.2)) and continuous (on
double coset spaces) so clearly induces a surjective map between the sets of
connected components. On the other hand if $[g_{1}]$ and $[g_{2}]$ (double
cosets of $g_{i}\in \mathbf{G}(\Bbb{A})$) are mapped by $\nu $ to the same
connected component in $\mathbf{T}(\Bbb{Q})\backslash \mathbf{T}(\Bbb{A}%
)/\nu (K),$ then since $G_{\infty }$ is mapped \emph{onto} the connected
component of the identity in $\mathbf{T}(\Bbb{Q})\backslash \mathbf{T}(\Bbb{A%
})/\nu (K),$ modifying $g_{1}$ by an element of $G_{\infty }$ we may assume
that 
\begin{equation}
\nu ([g_{1}])=\nu ([g_{2}])\in \mathbf{T}(\Bbb{Q})\backslash \mathbf{T}(\Bbb{%
A})/\nu (K),
\end{equation}
without changing the connected component in which $[g_{1}]$ lies. Once this
has been established, for appropriate representatives $g_{i}$ of the double
cosets, $g_{1}^{-1}g_{2}\in \mathbf{G}^{\prime }(\Bbb{A}),$ so by the
connectedness of $\mathbf{G}^{\prime }(\Bbb{Q})\backslash \mathbf{G}^{\prime
}(\Bbb{A})/K^{\prime }$, $[g_{1}]$ and $[g_{2}]$ lie in the same connected
component of $\mathbf{G}(\Bbb{Q})\backslash \mathbf{G}(\Bbb{A})/K.$

The group $\pi _{0}(\mathbf{T}(\Bbb{Q})\backslash \mathbf{T}(\Bbb{A})/\nu
(K))$ is the group 
\begin{equation}
\mathcal{K}^{\times }\backslash \mathcal{K}_{\Bbb{A}}^{\times }/\Bbb{C}%
^{\times }\nu (K_{f})=\mathcal{K}^{\times }\backslash \mathcal{K}%
_{f}^{\times }/\nu (K_{f}).
\end{equation}
It sits in a short exact sequence 
\begin{equation}
0\rightarrow \mu _{\mathcal{K}}\backslash U_{\mathcal{K}}/\nu
(K_{f})\rightarrow \mathcal{K}^{\times }\backslash \mathcal{K}_{f}^{\times
}/\nu (K_{f})\overset{cl}{\rightarrow }Cl_{\mathcal{K}}\rightarrow 0,
\label{class}
\end{equation}
where $U_{\mathcal{K}}$ is the product of local units at all the finite
primes and $Cl_{\mathcal{K}}$ is the class group.

\subsubsection{The $cl$ and $\nu _{N}$ invariants of a connected component}

The norm $N:\mathcal{K}^{\times }\rightarrow \Bbb{Q}^{\times }$ satisfies $%
N\circ \nu =\nu \nu ^{\rho }=\mu ,$ hence induces a map 
\begin{equation}
\mathcal{K}^{\times }\backslash \mathcal{K}_{f}^{\times }/\nu
(K_{f})\rightarrow \Bbb{Q}_{+}^{\times }\backslash \Bbb{Q}_{f}^{\times }/\mu
(K_{f}).
\end{equation}

Using the lattice $L$ as an integral structure in $V,$ we see that $\mathbf{G%
}$ comes from a group scheme $\mathbf{G}_{\Bbb{Z}}$ over $\Bbb{Z},$ whose
points in any ring $A$ are 
\begin{equation}
\mathbf{G}_{\Bbb{Z}}(A)=\left\{ (g,\mu )\in GL_{\mathcal{O}_{\mathcal{K}%
}\otimes A}(L_{A})\times A^{\times }|\,\left\langle gu,gv\right\rangle =\mu
\left\langle u,v\right\rangle \,\,\right\} .
\end{equation}
We likewise get that $\mu $ is a homomorphism from $\mathbf{G}_{\Bbb{Z}}$ to 
$\Bbb{G}_{m}.$ The diagram 
\begin{equation}
\begin{array}{lll}
\mathbf{G}_{\Bbb{Z}}(\Bbb{Z}_{p}) & \overset{\mu }{\rightarrow } & \Bbb{Z}%
_{p}^{\times } \\ 
\downarrow &  & \downarrow \\ 
\mathbf{G}_{\Bbb{Z}}(\Bbb{Z}_{p}/N\Bbb{Z}_{p}) & \overset{\mu }{\rightarrow }
& (\Bbb{Z}_{p}/N\Bbb{Z}_{p})^{\times }
\end{array}
\end{equation}
commutes, $\mathbf{G}_{\Bbb{Z}}(\Bbb{Z}_{p})=K_{p}^{0}$ and the kernel of $%
\mathbf{G}_{\Bbb{Z}}(\Bbb{Z}_{p})\rightarrow \mathbf{G}_{\Bbb{Z}}(\Bbb{Z}%
_{p}/N\Bbb{Z}_{p})$ is $K_{p}.$ This shows that $\mu (K_{f})\subset \Bbb{%
\hat{Z}}^{\times }(N),$ the product of local units congruent to $1\mod %
N.$ But 
\begin{equation}
\Bbb{Q}_{+}^{\times }\backslash \Bbb{Q}_{f}^{\times }/\Bbb{\hat{Z}}^{\times
}(N)=(\Bbb{Z}/N\Bbb{Z})^{\times }.
\end{equation}

To conclude, we have shown the existence of two maps from the set of
connected components: 
\begin{equation}
cl:\pi _{0}(\mathbf{G}(\Bbb{Q})\backslash \mathbf{G}(\Bbb{A})/K)\rightarrow
Cl_{\mathcal{K}}
\end{equation}
\begin{equation}
\nu _{N}:\pi _{0}(\mathbf{G}(\Bbb{Q})\backslash \mathbf{G}(\Bbb{A}%
)/K)\rightarrow (\Bbb{Z}/N\Bbb{Z})^{\times }.
\end{equation}
These two maps are independent: together they map $\pi _{0}(\mathbf{G}(\Bbb{Q%
})\backslash \mathbf{G}(\Bbb{A})/K)$ \emph{onto} $Cl_{\mathcal{K}}\times (%
\Bbb{Z}/N\Bbb{Z})^{\times }.$ On the other hand, they have a non-trivial
common kernel, which grows with $N,$ as is evident from the interpretation
of $\mathcal{K}^{\times }\backslash \mathcal{K}_{f}^{\times }/\nu (K_{f})$
as the Galois group of a certain class field extension of $\mathcal{K}.$ The
map $cl$ gives the restriction to the Hilbert class field, while the map $%
\nu _{N}$ gives the restriction to the cyclotomic field $\Bbb{Q}(\mu _{N}).$
We have singled out $cl$ and $\nu _{N}$, because when $N\ge 3,$ they have an
interpretation in terms of the complex points of $Sh_{K}.$

\begin{proposition}
Let $[g]\in \mathbf{G}(\Bbb{Q})\backslash \mathbf{G}(\Bbb{A})/K=Sh_{K}(\Bbb{C%
}).$ Then

(i) $cl([g])$ is the \emph{Steinitz class} of the lattice $L_{g}=g_{f}(\hat{L%
})\cap V$ in $Cl_{\mathcal{K}}$.

(ii) $\nu _{N}([g])$ is (essentially) the $\nu _{N,g}$ that appears in the
definition of $\alpha _{g}$ (see \ref{Moduli}).
\end{proposition}

\begin{proof}
(i) $cl([g])$ is the class of the ideal $(\nu (g_{f}))$ associated to the
idele $\nu (g_{f})\in \mathcal{K}_{f}^{\times }$. This ideal is in the same
class as $(\det (g_{f})),$ because $\mu (g_{f})\in \Bbb{Q}_{f}^{\times }$,
so $(\mu (g_{f}))$ is principal. But the class of $(\det (g_{f}))$ is the
Steinitz class of $L_{g},$ since the Steinitz class of $L$ is trivial.

(ii) To find $\nu _{N}([g])$ we first project the idele $\mu (g_{f})$ to $%
\Bbb{\hat{Z}}^{\times }$ using $\Bbb{Q}_{f}^{\times }=\Bbb{Q}_{+}^{\times }%
\Bbb{\hat{Z}}^{\times }.$ But this is just $\tilde{\mu}(g_{f})^{-1}\mu
(g_{f}).$ We then take the result modulo $N,$ so 
\begin{equation}
\nu _{N}([g])=\tilde{\mu}(g_{f})^{-1}\mu (g_{f})\mod N.
\end{equation}
Now the definition of the tuple $(A_{g},\lambda _{g},\iota _{g},\alpha _{g})$
is such that if $u,v\in N^{-1}L/L$ then 
\begin{eqnarray}
\left\langle \alpha _{g}(u),\alpha _{g}(v)\right\rangle _{\lambda _{g}}
&=&\exp \left( 2\pi iN\left\langle g_{f}u,g_{f}v\right\rangle _{g}\right) 
\notag \\
&=&\exp \left( 2\pi i\tilde{\mu}(g_{f})^{-1}N\left\langle
g_{f}u,g_{f}v\right\rangle \right)  \notag \\
&=&\exp \left( 2\pi i\tilde{\mu}(g_{f})^{-1}\mu (g_{f})N\left\langle
u,v\right\rangle \right)  \notag \\
&=&\exp \left( 2\pi i\nu _{N}([g])N\left\langle u,v\right\rangle \right)
\label{nu_N_again}
\end{eqnarray}
Part (ii) follows if we identify $\nu _{N,g}\in Isom_{\Bbb{C}}(N^{-1}\Bbb{Z}/%
\Bbb{Z},\mu _{N})$ with $\nu _{N}([g])\in (\Bbb{Z}/N\Bbb{Z})^{\times }$
using $\exp (2\pi i(\cdot )).$
\end{proof}

\subsubsection{The complex uniformization\label{uniformization}}

Recall that $\frak{X}=G_{\infty }/K_{\infty }$ and that it was equipped with
a base point $x_{0}$ (coresponding to $(z,u)=(\delta _{\mathcal{K}}/2,0)$ in
the Siegel domain of the second kind). Let $1=g_{1},\dots ,g_{m}\in \mathbf{G%
}(\Bbb{A}_{f})$ $(m=\#(\mathcal{K}^{\times }\backslash \mathcal{K}%
_{f}^{\times }/\nu (K_{f})))$ be representatives of the connected components
of $\mathbf{G}(\Bbb{Q})\backslash \mathbf{G}(\Bbb{A})/K,$ and define
congruence groups 
\begin{equation}
\Gamma _{j}=\mathbf{G}(\Bbb{Q})\cap g_{j}K_{f}g_{j}^{-1}.
\end{equation}
We write $[x,g_{j}]$ for $\mathbf{G}(\Bbb{Q})(x,g_{j}K_{f})\in \mathbf{G}(%
\Bbb{Q})\backslash (\frak{X}\times \mathbf{G}(\Bbb{A}_{f})/K_{f})=\mathbf{G}(%
\Bbb{Q})\backslash \mathbf{G}(\Bbb{A})/K.$ Then $[x^{\prime
},g_{j}]=[x,g_{j}]$ if and only if $x^{\prime }=\gamma x$ for $\gamma \in
\Gamma _{j}.$ The map 
\begin{equation}
\coprod_{j=1}^{m}X_{\Gamma _{j}}=\coprod_{j=1}^{m}\Gamma _{j}\backslash 
\frak{X}\simeq Sh_{K}(\Bbb{C})
\end{equation}
sending $\Gamma _{j}x$ to $[x,g_{j}]$ is an isomorphism.

Note that $\Gamma _{1}=\Gamma $ is the principal level-$N$ congruence
subgroup in $\mathbf{G}_{\Bbb{Z}}(\Bbb{Z}),$ the stabilizer of $L$.
Similarly, $\Gamma _{j}$ is the principal level-$N$ congruence subgroup in
the stabilizer of $L_{g_{j}},$ and is thus a group of the type considered in 
\ref{lattices}, except that we have dropped the assumption on the Steinitz
class of $L_{g_{j}}.$ As $N\ge 3,$ $\det (\gamma )=1$ and $\mu (\gamma )=1$
for all $\gamma \in \Gamma _{j},$ for every $j$. Indeed, on the one hand
these are in $\mathcal{K}^{\times }$ and $\Bbb{Q}_{+}^{\times }$
respectively. On the other hand, they are local units which are congruent to 
$1\mod N$ everywhere. It follows that $\Gamma _{j}$ are subgroups of $%
\mathbf{G}^{\prime }(\Bbb{Q})=\mathbf{SU}(\Bbb{Q}).$

We get a similar decomposition to connected components (as an algebraic
surface) 
\begin{equation}
S_{\Bbb{C}}=\coprod_{j=1}^{m}S_{\Gamma _{j}}
\end{equation}
and we write $S_{\Bbb{C}}^{*}=\coprod_{j=1}^{m}S_{\Gamma _{j}}^{*}$ for the
Baily-Borel compactification.

\subsection{Smooth compactifications}

\subsubsection{The smooth compactification of $X_{\Gamma }\label%
{complex smooth compactification}$}

We begin by working in the complex analytic category and follow the
exposition of [Cog]. The Baily-Borel compactification $X_{\Gamma }^{*}$ is
singular at the cusps, and does not admit a modular interpretation. For
general unitary Shimura varieties, the theory of toroidal compactifications
provides smooth compactifications that depend, in general, on extra data. It
is a unique feature of Picard modular surfaces, stemming from the finiteness
of $\mathcal{O}_{\mathcal{K}}^{\times }$, that this smooth compactification
is canonical. As all cusps are equivalent (if we vary the lattice $L$ or $%
\Gamma $), it is enough, as usual, to study the smooth compactification at $%
c_{\infty }.$ In [Cog] this is described for an arbitrary $L$ (not even $%
\mathcal{O}_{\mathcal{K}}$-free), but for simplicity we write it down only
for $L=L_{0}.$

As $N\ge 3,$ elements of $\Gamma $ stabilizing $c_{\infty }$ lie in $%
N_{\infty }.\footnote{%
No confusion should arise from the use of the letter $N$ to denote both the
level and the unipotent radical of $P$.}$ The computations, which we omit,
are somewhat simpler if $N$ is \emph{even}, an assumption made for the rest
of this section. Let 
\begin{equation}
\Gamma _{cusp}=\Gamma \cap N_{\infty }.
\end{equation}

\begin{lemma}
Let $N\ge 3$ be even. The matrix $n(s,r)\in \Gamma _{cusp}$ if and only if:
(i) ($d_{\mathcal{K}}\equiv 1\mod 4$) $s\in N\mathcal{O}_{\mathcal{K}},$
$r\in ND_{\mathcal{K}}\Bbb{Z}$, (ii) ($d_{\mathcal{K}}\equiv 2,3\mod 4$%
) $s\in N\mathcal{O}_{\mathcal{K}}$ and $r\in 2^{-1}ND_{\mathcal{K}}\Bbb{Z}$
.
\end{lemma}

Let $M=N|D_{\mathcal{K}}|$ in case \emph{(i) }and $M=2^{-1}N|D_{\mathcal{K}%
}| $ in case \emph{(ii)}. This is the \emph{width of the cusp }$c_{\infty }.$
Let 
\begin{equation}
q=q(z)=e^{2\pi iz/M}.  \label{width}
\end{equation}
For $R>0,$ the domain $\Omega _{R}=\left\{ (z,u)\in \frak{X}|\,\lambda
(z,u)>R\right\} $ is invariant under $\Gamma _{cusp}$ and if $R$ is large
enough, two points of it are $\Gamma $-equivalent if and only if they are $%
\Gamma _{cusp}$-equivalent. A sufficiently small punctured neighborhood of $%
c_{\infty }$ in $X_{\Gamma }^{*}$ therefore looks like $\Gamma
_{cusp}\backslash \Omega _{R}.$ As 
\begin{equation}
n(s,r)(z,u)=(z+\delta \bar{s}(u+s/2)+r,u+s)  \label{unipotent action}
\end{equation}
we obtain the following description of $\Gamma _{cusp}\backslash \Omega
_{R}. $ Let $\Lambda =N\mathcal{O}_{\mathcal{K}}$ and $E=\Bbb{C}/\Lambda ,$
an elliptic curve with complex multiplication by $\mathcal{O}_{\mathcal{K}}.$
Let $\mathcal{T}$ be the quotient 
\begin{equation}
\mathcal{T}=(\Bbb{C}\times \Bbb{C})/\Lambda
\end{equation}
where the action of $s\in \Lambda $ is via 
\begin{equation}
\lbrack s]:(t,u)\mapsto (e^{2\pi i\delta \bar{s}(u+s/2)/M}t,u+s).
\end{equation}
It is a line bundle over $E$ via the second projection. We denote the class
of $(t,u)$ modulo the action of $\Lambda $ by $[t,u].$

\begin{proposition}
\label{disk bundle}Let $\mathcal{T}_{R}\subset \mathcal{T}$ be the disk
bundle consisting of all the points $[t,u]$ where 
\begin{equation}
|t|<e^{-\pi |\delta |(R+u\bar{u})/M}.
\end{equation}
(This condition is invariant under the action of $\Lambda .)$ Let $\mathcal{T%
}_{R}^{\prime }$ be the \emph{punctured disk bundle }obtained by removing
the zero section from $\mathcal{T}_{R}.$ Then the map $(z,u)\mapsto (q(z),u)$
induces an analytic isomorphism between $\Gamma _{cusp}\backslash \Omega
_{R} $ and $\mathcal{T}_{R}^{\prime }.$
\end{proposition}

\begin{proof}
This follows from the discussion so far and the fact that $\lambda (z,u)>R$
is equivalent to the above condition on $t=q(z)$ ([Cog], Prop. 2.1).
\end{proof}

To get a smooth compactification $\bar{X}_{\Gamma }$ of $X_{\Gamma }$ (as a
complex surface), we glue the disk bundle $\mathcal{T}_{R}$ to $X_{\Gamma }$
along $\mathcal{T}_{R}^{\prime }$. In other words, we complete $\mathcal{T}%
_{R}^{\prime }$ by adding the zero section, which is isomorphic to $E.$ The
same procedure should be carried out at any other cusp of $\mathcal{C}%
_{\Gamma }.$

Note that the geodesic (\ref{geodesic}) connecting $(z,u)\in \frak{X}$ to
the cusp $c_{\infty }$ projects in $\bar{X}_{\Gamma }$ to a geodesic which
meets $E$ transversally at the point $u\mod \Lambda .$ We caution that
this geodesic in $X_{\Gamma }$ depends on $(z,u)$ and $c_{\infty }$ and not
only on their images modulo $\Gamma .$

The line bundle $\mathcal{T}$ is the \emph{inverse} of an ample line bundle
on $E$. In fact, $\mathcal{T}^{\vee }$ is the $N$-th (resp. $2N$-th) power
of one of the four basic \emph{theta line bundles} if $d_{\mathcal{K}}\equiv
1\mod 4$ (resp. $d_{\mathcal{K}}\equiv 2,3\mod 4)$. A basic theta
function of the lattice $\Lambda $ satisfies, for $u\in \Bbb{C}$ and $s\in
\Lambda ,$%
\begin{equation}
\theta (u+s)=\alpha (s)e^{2\pi \bar{s}(u+s/2)/|\delta |N^{2}}\theta (u)
\end{equation}
where $\alpha :\Lambda \rightarrow \pm 1$ is a quasi-character (see [Mu1],
p.25). Recalling the relation between $M$ and $N$, and the assumption that $%
N $ was even, we easily get the relation between $\mathcal{T}$ and the theta
line bundles.

Recall that with any $x=(z,u)\in \frak{X}$ we associated a complex abelian
variety $A_{x},$ and another model $A_{x}^{\prime }$ of the same abelian
variety (\ref{moving lattice model}). This allowed us to define sections $%
d\zeta _{1},d\zeta _{2}$ and $d\zeta _{3}$ of $\omega _{\mathcal{A}/\frak{X}%
}.$ A simple matrix computation gives the following.

\begin{lemma}
\label{section invariance}The sections $d\zeta _{1}$ and $d\zeta _{3}$ are
invariant under $\Gamma _{cusp}.$ The section $d\zeta _{2}$ is invariant
modulo the sub-bundle generated by $d\zeta _{1}.$
\end{lemma}

Thus $d\zeta _{1},d\zeta _{3}$ and $d\zeta _{2}\mod \left\langle d\zeta
_{1}\right\rangle $ descend to well-defined sections in the neighborhood $%
\mathcal{T}_{R}\simeq \Gamma _{cusp}\backslash \Omega _{R}\cup E$ of $E$ in $%
\bar{X}_{\Gamma }.$

\subsubsection{The smooth compactification of $S$}

The arithmetic compactification $\bar{S}$ of the Picard surface $S$ (over $%
R_{0}$) is due to Larsen [La1],[La2] (see also [Bel] and [Lan]). We
summarize the results in the following theorem. We mention first that as $S_{%
\Bbb{C}}$ has a canonical model $S$ over $R_{0},$ its Baily-Borel
compactification $S_{\Bbb{C}}^{*}$ has a similar model $S^{*}$ over $R_{0}$,
and $S$ embeds in $S^{*}$ as an open dense subscheme.

\begin{theorem}
(i) There exists a projective scheme $\bar{S},$ smooth over $R_{0},$ of
relative dimension 2, together with an open dense immersion of $S$ in $\bar{S%
}$, and a proper morphism $p:\bar{S}\rightarrow S^{*}$, making the following
diagram commutative 
\begin{equation}
\begin{array}{lll}
S & \rightarrow & \bar{S} \\ 
\downarrow & \overset{p}{\swarrow } &  \\ 
S^{*} &  & 
\end{array}
.
\end{equation}

(ii) As a complex manifold, there is an isomorphism 
\begin{equation}
\bar{S}_{\Bbb{C}}\simeq \coprod_{j=1}^{m}\bar{X}_{\Gamma _{j}},
\end{equation}
extending the isomorphism of $S_{\Bbb{C}}$ with $\coprod_{j=1}^{m}X_{\Gamma
_{j}}.$

(iii) Let $C=p^{-1}(S^{*}\backslash S).$ Let $R_{N}$ be the integral closure
of $R_{0}$ in the ray class field $\mathcal{K}_{N}$ of conductor $N$ over $%
\mathcal{K}$. Then the connected components of $C_{R_{N}}$ are geometrically
irreducible, and are indexed by the cusps of $S_{R_{N}}^{*}$ over which they
sit. Furthermore, each component $E\subset C_{R_{N}}$ is an elliptic curve
with complex multiplication by $\mathcal{O}_{\mathcal{K}}.$
\end{theorem}

We call $C$ the \emph{cuspidal divisor}. If $c\in S_{\Bbb{C}}^{*}\backslash
S_{\Bbb{C}}$ is a cusp, we denote the complex elliptic curve $p^{-1}(c)$ by $%
E_{c}.$ Although $E_{c}$ is in principle definable over the Hilbert class
field $\mathcal{K}_{1},$ no canonical model of it over that field is
provided by $\bar{S}.$ On the other hand, $E_{c}$ does come with a canonical
model over $\mathcal{K}_{N}$, and even over $R_{N}.$

We refer to [La1] and [Bel] for a moduli-theoretic interpretation of $C$ as
a moduli space for semi-abelian schemes with a suitable action of $\mathcal{O%
}_{\mathcal{K}}$ and a ``level-$N$ structure''.

\subsubsection{Change of level\label{level change}}

Assume that $N\ge 3$ is even, and $N^{\prime }=QN.$ We then obtain a
covering map $X_{\Gamma (N^{\prime })}\rightarrow X_{\Gamma (N)}$ where by $%
\Gamma (N)$ we denote the group previously denoted by $\Gamma .$ Near any of
the cusps, the analytic model allows us to analyze this map locally. Let $%
E^{\prime }$ be an irreducible cuspidal component of $\bar{X}_{\Gamma
(N^{\prime })}$ mapping to the irreducible component $E$ of $\bar{X}_{\Gamma
(N)}.$ The following is a consequence of the discussion in the previous
sections.

\begin{proposition}
The map $E^{\prime }\rightarrow E$ is a multiplication-by-$Q$ isogeny, hence
\'{e}tale of degree $Q^{2}.$ When restricted to a neighborhood of $E^{\prime
},$ the covering $\bar{X}_{\Gamma (N^{\prime })}\rightarrow \bar{X}_{\Gamma
(N)}$ is of degree $Q^{3},$ and has ramification index $Q$ along $E,$ in the
normal direction to $E$.
\end{proposition}

\begin{corollary}
The pull-back to $E^{\prime }$ of the normal bundle $\mathcal{T}(N)$ of $E$
is the $Q$th power of the normal bundle $\mathcal{T}(N^{\prime })$ of $%
E^{\prime }.$
\end{corollary}

\subsection{The universal semi-abelian scheme $\mathcal{A}$}

\subsubsection{The universal semi-abelian scheme over $\bar{S}\label%
{universal semi abelian}$}

As Larsen and Bella\"{i}che explain, the universal abelian scheme $\pi :%
\mathcal{A}\rightarrow S$ extends canonically to a semi-abelian scheme $\pi :%
\mathcal{A}\rightarrow \bar{S}.$ The polarization $\lambda $ extends over
the boundary $C=\bar{S}\backslash S$ to a principal polarization $\lambda $
of the abelian part of $\mathcal{A}$. The action $\iota $ of $\mathcal{O}_{%
\mathcal{K}}$ extends to an action on the semi-abelian variety, which
necessarily induces separate actions on the toric part and on the abelian
part.

Let $E$ be a connected component of $C_{R_{N}},$ mapping (over $\Bbb{C}$ and
under the projection $p$) to the cusp $c\in S_{\Bbb{C}}^{*}.$ Then there
exist (1) a principally polarized elliptic curve $B$ defined over $R_{N},$
with complex multiplication by $\mathcal{O}_{\mathcal{K}}$ and CM type $%
\Sigma ,$ and (2) an ideal $\frak{a}$ of $\mathcal{O}_{\mathcal{K}}$, such
that every fiber $\mathcal{A}_{x}$ of $\mathcal{A}$ over $E$ is an $\mathcal{%
O}_{\mathcal{K}}$-group extension of $B$ by the $\mathcal{O}_{\mathcal{K}}$%
-torus $\frak{a}\otimes \Bbb{G}_{m}.$ Both $B$ (with its polarization) and
the ideal class $[\frak{a}]\in Cl_{\mathcal{K}}$ are uniquely determined by
the cusp $c.$ Only the extension class in the category of $\mathcal{O}_{%
\mathcal{K}}$-groups varies as we move along $E$. Note that since the Lie
algebra of the torus is of type $(1,1),$ the Lie algebra of such an
extension $\mathcal{A}_{x}$ is of type $(2,1),$ as is the case at an
interior point $x\in S.$ If we extend scalars to $\Bbb{C},$ the isomorphism
type of $B$ is given by another ideal class $[\frak{b}]$ (i.e. $B(\Bbb{C}%
)\simeq \Bbb{C}/\frak{b}$). In this case we say that the cusp $c$ is of type 
$(\frak{a},\frak{b}).$

The above discussion defines a homomorphism (of fppf sheaves over $%
Spec(R_{N})$) 
\begin{equation}
E\rightarrow Ext_{\mathcal{O}_{\mathcal{K}}}^{1}(B,\frak{a}\otimes \Bbb{G}%
_{m}).
\end{equation}
As we shall see soon, the $Ext$ group is represented by an elliptic curve
with CM by $\mathcal{O}_{\mathcal{K}},$ defined over $R_{N},$ and this map
is an isogeny.

\subsubsection{$\mathcal{O}_{\mathcal{K}}$-semi-abelian schemes of type $(%
\frak{a},\frak{b})$}

We digress to discuss the moduli space for semi-abelian schemes of the type
found above points of $E.$ Let $R$ be an $R_{0}$-algebra, $B$ an elliptic
curve over $R$ with complex multiplication by $\mathcal{O}_{\mathcal{K}}$
and CM type $\Sigma ,$ and $\frak{a}$ an ideal of $\mathcal{O}_{\mathcal{K}%
}. $ Consider a semi-abelian scheme $\mathcal{G}$ over $R,$ endowed with an $%
\mathcal{O}_{\mathcal{K}}$ action $\iota :\mathcal{O}_{\mathcal{K}%
}\rightarrow End(\mathcal{G}),$ and a short exact sequence 
\begin{equation}
0\rightarrow \frak{a}\otimes \Bbb{G}_{m}\rightarrow \mathcal{G}\rightarrow
B\rightarrow 0  \label{semi-abelian extension}
\end{equation}
of $\mathcal{O}_{\mathcal{K}}$-group schemes over $R.$ We call all this data
a \emph{semi-abelian scheme of type} $(\frak{a},B)$ (over $R$). The group
classifying such structures is $Ext_{\mathcal{O}_{\mathcal{K}}}^{1}(B,\frak{a%
}\otimes \Bbb{G}_{m}).$ Any $\chi \in \frak{a}^{*}=Hom(\frak{a},\Bbb{Z})$
defines, by push-out, an extension $\mathcal{G}_{\chi }$ of $B$ by $\Bbb{G}%
_{m},$ hence a point of $B^{t}=Ext^{1}(B,\Bbb{G}_{m}).$ We therefore get a
homomorphism from $Ext_{\mathcal{O}_{\mathcal{K}}}^{1}(B,\frak{a}\otimes 
\Bbb{G}_{m})$ to $Hom(\frak{a}^{*},B^{t})$. A simple check shows that its
image is in $Hom_{\mathcal{O}_{\mathcal{K}}}(\frak{a}^{*},B^{t})=\delta _{%
\mathcal{K}}\frak{a}\otimes _{\mathcal{O}_{\mathcal{K}}}B^{t}$, and that
this construction yields an isomorphism 
\begin{equation}
Ext_{\mathcal{O}_{\mathcal{K}}}^{1}(B,\frak{a}\otimes \Bbb{G}_{m})\simeq
\delta _{\mathcal{K}}\frak{a}\otimes _{\mathcal{O}_{\mathcal{K}}}B^{t}.
\end{equation}
Here we have used the canonical identification $\frak{a}^{*}=\delta _{%
\mathcal{K}}^{-1}\frak{a}^{-1}$ (via the trace pairing). Although $(\delta _{%
\mathcal{K}})$ is a principal ideal, so can be ignored, it is better to keep
track of its presence. We emphasize that the CM type of $B^{t},$ with the
natural action of $\mathcal{O}_{\mathcal{K}}$ derived from its action on $B,$
is $\bar{\Sigma}$ rather than $\Sigma .$

Thus over $\delta _{\mathcal{K}}\frak{a}\otimes _{\mathcal{O}_{\mathcal{K}%
}}B^{t}$ there is a universal semi-abelian scheme $\mathcal{G}(\frak{a},B)$
of type $(\frak{a},B),$ and any $\mathcal{G}$ as above, over any base $%
R^{\prime }/R,$ is obtained from $\mathcal{G}(\frak{a},B)$ by pull-back
(specialization) with respect to a unique map $Spec(R^{\prime })\rightarrow
\delta _{\mathcal{K}}\frak{a}\otimes _{\mathcal{O}_{\mathcal{K}}}B^{t}.$

When $R=\Bbb{C},$ $B\simeq \Bbb{C}/\frak{b}$ for a unique ideal class $[%
\frak{b}]$ (with $\mathcal{O}_{\mathcal{K}}$ acting via $\Sigma $). Then,
canonically, $B^{t}=\Bbb{C}/\delta _{\mathcal{K}}^{-1}\overline{\frak{b}}%
^{-1}$ (with $\mathcal{O}_{\mathcal{K}}$ acting via $\bar{\Sigma}$). The
pairing between the lattices, $\frak{b}\times \delta _{\mathcal{K}}^{-1}%
\overline{\frak{b}}^{-1}\rightarrow \Bbb{Z}$ is $(x,y)\mapsto Tr_{\mathcal{K}%
/\Bbb{Q}}(x\bar{y}).$ Since the $\mathcal{O}_{\mathcal{K}}$ action on $B^{t}$
is via $\bar{\Sigma},$%
\begin{equation}
Ext_{\mathcal{O}_{\mathcal{K}}}^{1}(\Bbb{C}/\frak{b},\frak{a}\otimes \Bbb{G}%
_{m})\simeq \delta _{\mathcal{K}}\frak{a}\otimes _{\mathcal{O}_{\mathcal{K}}}%
\Bbb{C}/\delta _{\mathcal{K}}^{-1}\overline{\frak{b}}^{-1}=\Bbb{C}/\overline{%
\frak{a}}\overline{\frak{b}}^{-1}.
\end{equation}
The universal semi-abelian variety $\mathcal{G}(\frak{a},B)$ will now be
denoted $\mathcal{G}(\frak{a},\frak{b}).$ In \ref{semi abelian model} below
we give a complex analytic model of this $\mathcal{G}(\frak{a},\frak{b}).$

\subsection{Degeneration of $\mathcal{A}$ along a geodesic connecting to a
cusp\label{degeneration}}

\subsubsection{The degeneration to a semi-abelian variety}

It is instructive to use the ``moving lattice model'' to compute the
degeneration of the universal abelian scheme along a geodesic, as we
approach a cusp. To simplify the computations, assume for the rest of this
section, as before, that $N\ge 3$ is \emph{even}, and that the cusp is the
standrad cusp at infinity $c=c_{\infty }.$ In this case we have shown that $%
E_{c}=\Bbb{C}/\Lambda ,$ where $\Lambda =N\mathcal{O}_{\mathcal{K}},$ and we
have given a neighborhood of $E_{c}$ in $\bar{X}_{\Gamma }$ the structure of
a disk bundle in a line bundle $\mathcal{T}.$ See Proposition \ref{disk
bundle}.

Consider the geodesic (\ref{geodesic}) connecting $(z,u)$ to $c_{\infty }.$
Consider the universal abelian scheme in the moving lattice model (\emph{cf }%
(\ref{moving lattice model})). Of the three vectors used to span $%
L_{x}^{\prime }$ over $\mathcal{O}_{\mathcal{K}}$ in (\ref{spanning vectors}%
) the first two do not depend on $z.$ As $u$ is fixed along the geodesic,
they are not changed. The third vector represents a cycle that vanishes at
the cusp (together with all its $\mathcal{O}_{\mathcal{K}}$-multiples). We
conclude that $A_{x}^{\prime }$ degenerates to 
\begin{equation}
\Bbb{C}^{3}/Span_{\iota ^{\prime }(\mathcal{O}_{\mathcal{K}})}\left\{ \left( 
\begin{array}{l}
0 \\ 
1 \\ 
1
\end{array}
\right) ,\left( 
\begin{array}{l}
1 \\ 
0 \\ 
u
\end{array}
\right) \right\} .
\end{equation}
Making the change of variables $(\zeta _{1}^{\prime },\zeta _{2}^{\prime
},\zeta _{3}^{\prime })=(\zeta _{1},\zeta _{2}+\bar{u}\zeta _{1},\zeta _{3})$
does not alter the $\mathcal{O}_{\mathcal{K}}$ action and gives the more
symmetric model 
\begin{equation}
\mathcal{G}_{u}=\Bbb{C}^{3}/Span_{\iota ^{\prime }(\mathcal{O}_{\mathcal{K}%
})}\left\{ \left( 
\begin{array}{l}
0 \\ 
1 \\ 
1
\end{array}
\right) ,\left( 
\begin{array}{l}
1 \\ 
\bar{u} \\ 
u
\end{array}
\right) \right\}  \label{G_u}
\end{equation}
(but note that $\zeta _{2}^{\prime },$ unlike $\zeta _{2},$ does not vary
holomorphically in the \emph{family} $\{\mathcal{G}_{u}\},$ only in each
fiber individually).

Let $e(x)=e^{2\pi ix}:\Bbb{C}\rightarrow \Bbb{C}^{\times }$ be the
exponential map, with kernel $\Bbb{Z}.$ For any ideal $\frak{a}$ of $%
\mathcal{O}_{\mathcal{K}}$ it induces a map 
\begin{equation}
e_{\frak{a}}:\frak{a}\otimes \Bbb{C}\rightarrow \frak{a}\otimes \Bbb{C}%
^{\times }
\end{equation}
with kernel $\frak{a}\otimes 1.$ As usual we identify $\frak{a}\otimes \Bbb{C%
}$ with $\Bbb{C}(\Sigma)\oplus \Bbb{C}(\bar{\Sigma}),$ sending $a\otimes \zeta
\mapsto (a\zeta ,\bar{a}\zeta ).$ We now note that if we use this
identification to identify $\Bbb{C}^{3}$ with $\Bbb{C} \oplus (\mathcal{O}_{%
\mathcal{K}}\otimes \Bbb{C})$ (an identification which is compatible with
the $\mathcal{O}_{\mathcal{K}}$ action) then the $\iota ^{\prime }(\mathcal{O%
}_{\mathcal{K}})$-span of the vector $^{t}(0,1,1)$ is just the kernel of $e_{%
\mathcal{O}_{\mathcal{K}}}.$ We conclude that 
\begin{equation}
\mathcal{G}_{u}\simeq \{\Bbb{C}\oplus (\mathcal{O}_{\mathcal{K}}\otimes \Bbb{%
C}^{\times })\}/L_{u}  \label{semi-abelian}
\end{equation}
where $L_{u}$ is the sub-$\mathcal{O}_{\mathcal{K}}$-module 
\begin{equation}
L_{u}=\left\{ (s,e_{\mathcal{O}_{\mathcal{K}}}(s\bar{u},\bar{s}u))|\,s\in 
\mathcal{O}_{\mathcal{K}}\right\} .
\end{equation}
This clearly gives $\mathcal{G}_{u}$ the structure of an $\mathcal{O}_{%
\mathcal{K}}$-semi-abelian variety of type $(\mathcal{O}_{\mathcal{K}},%
\mathcal{O}_{\mathcal{K}})$, i.e. an extension 
\begin{equation}
0\rightarrow \mathcal{O}_{\mathcal{K}}\otimes \Bbb{C}^{\times }\rightarrow 
\mathcal{G}_{u}\rightarrow \Bbb{C}/\mathcal{O}_{\mathcal{K}}\rightarrow 0.
\label{complex extension}
\end{equation}

\subsubsection{The analytic uniformization of the universal semi-abelian
variety of type $(\frak{a},\frak{b})\label{semi abelian model}$}

We now compare the description that we have found for the degeneration of $%
\mathcal{A}$ along the geodesic connecting $(z,u)$ to $c_{\infty }$ with the
analytic description of the universal semi-abelian variety of type $(\frak{a}%
,\frak{b}).$

\begin{proposition}
Let $\frak{a}$ and $\frak{b}$ be two ideals of $\mathcal{O}_{\mathcal{K}}.$
For $u\in \Bbb{C}$ consider 
\begin{equation}
\mathcal{G}_{u}\simeq \{\Bbb{C}\oplus (\frak{a}\otimes \Bbb{C}^{\times
})\}/L_{u}
\end{equation}
where 
\begin{equation}
L_{u}=\left\{ (s,e_{\frak{a}}(s\bar{u},\bar{s}u))|\,s\in \frak{b}\right\} .
\end{equation}
Then $\mathcal{G}_{u}$ is a semi-abelian variety of type $(\frak{a},\frak{b}%
),$ any complex semi-abelian variety of this type is a $\mathcal{G}_{u},$
and $\mathcal{G}_{u}\simeq \mathcal{G}_{v}$ if and only if $u-v\in \overline{%
\frak{a}}\overline{\frak{b}}^{-1}.$
\end{proposition}

\begin{proof}
That $\mathcal{G}_{u}$ is a semi-abelian variety of type $(\frak{a},\frak{b}%
) $ is obvious. That any abelian variety of this type is a $\mathcal{G}_{u}$
follows by passing to the universal cover $\Bbb{C}^{2}(\Sigma )\oplus \Bbb{C}%
(\bar{\Sigma})$, and noting that by a change of variables in the $\Sigma $-
and $\bar{\Sigma}$-isotypical parts, we may assume that the lattice by which
we divide is of the form 
\begin{equation}
\frak{a}\left( 
\begin{array}{l}
0 \\ 
1 \\ 
1
\end{array}
\right) \oplus \frak{b}\left( 
\begin{array}{l}
1 \\ 
\bar{u} \\ 
u
\end{array}
\right) .
\end{equation}
Finally, the map $u\mapsto [\mathcal{G}_{u}]$ is a homomorphism $\Bbb{%
C} \rightarrow Ext_{\mathcal{O}_{\mathcal{K}}}^{1}(\Bbb{C}/\frak{b},\frak{a}%
\otimes \Bbb{C}^{\times }),$ so we only have to prove that $\mathcal{G}_{u}$
is split if and only if $u\in \overline{\frak{a}}\overline{\frak{b}}^{-1}.$
But one can check easily that $\mathcal{G}_{u}$ is trivial if and only if $(s%
\bar{u},\bar{s}u)\in \ker e_{\frak{a}}=\frak{a}\otimes 1=\{(a,\bar{a})|a\in 
\frak{a}\}$ for every $s\in \frak{b},$ and this holds if and only if $u\in 
\overline{\frak{a}}\overline{\frak{b}}^{-1}.$
\end{proof}

\begin{corollary}
Let $N\ge 3$ be even. Let $c=c_{\infty }$ be the cusp at infinity. Then the
map 
\begin{equation}
E_{c}\rightarrow Ext_{\mathcal{O}_{\mathcal{K}}}^{1}(\Bbb{C}/\mathcal{O}_{%
\mathcal{K}},\mathcal{O}_{\mathcal{K}}\otimes \Bbb{C}^{\times })
\end{equation}
sending $u$ to the isomorphism class of the semi-abelian variety above $u%
\mod \Lambda $ is the isogeny of multiplication by $N$.
\end{corollary}

\begin{proof}
In view of the computations above, and the description of a neighborhood of $%
E_{c}$ in $\bar{X}_{\Gamma }$ given in Proposition \ref{disk bundle} this
map is identified with the canonical map 
\begin{equation}
\Bbb{C}/N\mathcal{O}_{\mathcal{K}}\rightarrow \Bbb{C}/\mathcal{O}_{\mathcal{K%
}}.
\end{equation}
\end{proof}

The extra data carried by $u\in E_{c},$ which is forgotten by the map of the
corollary, comes from the level $N$ structure. As mentioned before,
according to [La1] and [Bel] the cuspidal divisor $C$ has a modular
interpretation as the moduli space for semi-abelian schemes of the type
considered above, together with level-$N$ structure ($\mathcal{M}_{\infty
,N} $ structures in the language of [Bel]). A level-$N$ structure on a \emph{%
semi-abelian} variety $\mathcal{G}$ of type $(\frak{a},\frak{b})$ consists
of (i) a level-$N$ structures $\alpha :N^{-1}\mathcal{O}_{\mathcal{K}}/%
\mathcal{O}_{\mathcal{K}}\simeq \frak{a}\otimes \mu _{N}$ on the toric part
(ii) a level-$N$ structure $\beta :N^{-1}\mathcal{O}_{\mathcal{K}}/\mathcal{O%
}_{\mathcal{K}}\simeq N^{-1}\frak{b}/\frak{b}=B[N]$ on the abelian part
(iii) an $\mathcal{O}_{\mathcal{K}}$-splitting $\gamma $ of the map $%
\mathcal{G}[N]\rightarrow B[N].$

Over $c=c_{\infty },$ when $\frak{a}=\frak{b}=\mathcal{O}_{\mathcal{K}},$
there are obvious natural choices for $\alpha $ and $\beta $ (independent of 
$u$) but the splittings $\gamma $ in (iii) form a torsor under $\mathcal{O}_{%
\mathcal{K}}/N\mathcal{O}_{\mathcal{K}}.$ If we consider the splitting 
\begin{equation}
\gamma _{u}:N^{-1}\mathcal{O}_{\mathcal{K}}/\mathcal{O}_{\mathcal{K}}\ni
s\mapsto (s,e_{\mathcal{O}_{\mathcal{K}}}(s\bar{u},\bar{s}u))\mod L_{u}
\end{equation}
then the tuples $(\mathcal{G}_{u},\alpha ,\beta ,\gamma _{u})$ and $(%
\mathcal{G}_{v},\alpha ,\beta ,\gamma _{v})$ are isomorphic if and only if $%
u\equiv v\mod N\mathcal{O}_{\mathcal{K}},$ i.e. if and only if $u$ and $%
v$ represent the same point of $E_{c}.$

\subsection{The basic automorphic vector bundles}

\subsubsection{Definition and first properties}

Recall that we we have denoted by $\pi :\mathcal{A}\rightarrow \bar{S}$ the
universal semi-abelian variety over $\bar{S}$ (over the base ring $R_{0}$).
Let $\omega _{\mathcal{A}}$ be the relative cotangent space at the origin of 
$\mathcal{A}$. If $e:\bar{S}\rightarrow \mathcal{A}$ is the zero section, 
\begin{equation}
\omega _{\mathcal{A}}=e^{*}(\Omega _{\mathcal{A}/\bar{S}}^{1}).
\end{equation}
This is a rank 3 vector bundle over $\bar{S}$ and the action of $\mathcal{O}%
_{\mathcal{K}}$ allows to decompose it according to types. We let 
\begin{equation}
\mathcal{P}=\omega _{\mathcal{A}}(\Sigma ),\,\,\,\mathcal{L}=\omega _{%
\mathcal{A}}(\bar{\Sigma}).
\end{equation}
Then $\mathcal{P}$ is a plane bundle, and $\mathcal{L}$ a line bundle.

Over $S$ (but not over the cuspidal divisor $C=\bar{S}\backslash S$) we have
the usual identification $\omega _{\mathcal{A}}=\pi _{*}\Omega _{\mathcal{A}%
/S}^{1}.$ The relative de Rham cohomology of $\mathcal{A}/S$ is a rank 6
vector bundle sitting in an exact sequence (the Hodge filtration) 
\begin{equation}
0\rightarrow \omega _{\mathcal{A}}\rightarrow H_{dR}^{1}(\mathcal{A}%
/S)\rightarrow R^{1}\pi _{*}\mathcal{O}_{\mathcal{A}}\rightarrow 0.
\end{equation}
Since, for any abelian scheme, $R^{1}\pi _{*}\mathcal{O}_{\mathcal{A}%
}=\omega _{\mathcal{A}^{t}}^{\vee }$ (canonical isomorphism, see [Mu1]), and 
$\lambda :\mathcal{A}\rightarrow \mathcal{A}^{t}$ is an isomorphism which
reverses CM types, we obtain an exact sequence 
\begin{equation}
0\rightarrow \omega _{\mathcal{A}}\rightarrow H_{dR}^{1}(\mathcal{A}%
/S)\rightarrow \omega _{\mathcal{A}}^{\vee }(\rho )\rightarrow 0.
\label{Hodge}
\end{equation}
The notation $\mathcal{M}(\rho )$ means that $\mathcal{M}$ is a vector
bundle with an $\mathcal{O}_{\mathcal{K}}$ action and in $\mathcal{M}(\rho )$
the vector bundle structure is that of $\mathcal{M}$ but the $\mathcal{O}_{%
\mathcal{K}}$ action is conjugated. Decomposing according to types, we have
two short exact sequences 
\begin{eqnarray}
0 &\rightarrow &\mathcal{P}\rightarrow H_{dR}^{1}(\mathcal{A}/S)(\Sigma
)\rightarrow \mathcal{L}^{\vee }(\rho )\rightarrow 0 \\
0 &\rightarrow &\mathcal{L}\rightarrow H_{dR}^{1}(\mathcal{A}/S)(\bar{\Sigma}%
)\rightarrow \mathcal{P}^{\vee }(\rho )\rightarrow 0.  \notag
\end{eqnarray}

The pairing $\left\langle ,\right\rangle _{\lambda }$ on $H_{dR}^{1}(%
\mathcal{A}/S)$ induced by the polarization is $\mathcal{O}_{S}$-linear,
alternating, perfect, and satisfies $\left\langle \iota (a)x,y\right\rangle
_{\lambda }=\left\langle x,\iota (\bar{a})y\right\rangle _{\lambda }.$ It
follows that $H_{dR}^{1}(\mathcal{A}/S)(\Sigma )$ and $H_{dR}^{1}(\mathcal{A}%
/S)(\bar{\Sigma})$ are maximal isotropic subspaces, and are set in duality.
As $\omega _{\mathcal{A}}$ is also isotropic, this pairing induces pairings 
\begin{equation}
\mathcal{P}\times \mathcal{P}^{\vee }(\rho )\rightarrow \mathcal{O}_{S},\,\,%
\mathcal{L}\times \mathcal{L}^{\vee }(\rho )\rightarrow \mathcal{O}_{S}.
\end{equation}
These two pairings are the tautological pairings between a vector bundle and
its dual.

Another consequence of this discussion that we wish to record is the
canonical isomorphism over $S$%
\begin{equation}
\det \mathcal{P}=\mathcal{L}(\rho )\otimes \det \left( H_{dR}^{1}(\mathcal{A}%
/S)(\Sigma )\right) .  \label{detP}
\end{equation}

\subsubsection{The factors of automorphy corresponding to $\mathcal{L}$ and $%
\mathcal{P}\label{automorphy factors}$}

The formulae below can be deduced also from the matrix calculations in the
first few pages of [Sh2]. Let $\Gamma =\Gamma _{j}$ be one of the groups
used in the complex uniformization of $S_{\Bbb{C}},$ \emph{cf }Section \ref
{uniformization}. Via the analytic isomorphism $X_{\Gamma }\simeq S_{\Gamma }$ 
with the $j$th connected component, the vector bundles $\mathcal{P%
}$ and $\mathcal{L}$ are pulled back to $X_{\Gamma }$ and then to the
symmetric space $\frak{X},$ where they can be trivialized, hence described
by means of factors of automorphy. Let us denote by $\mathcal{P}_{an}$ and $%
\mathcal{L}_{an}$ the two vector bundles on $X_{\Gamma },$ in the complex
analytic category, or their pull-backs to $\frak{X}.$

To trivialize $\mathcal{L}_{an}$ we must choose a nowhere vanishing global
section over $\frak{X}$. As usual, we describe it only on the connected
component containing the standard cusp, corresponding to $j=1$ (where $%
L=L_{g_{1}}=L_{0}$). Recalling the ``moving lattice model'' (\ref{moving
lattice model}) and the coordinates $\zeta _{1},\zeta _{2},\zeta _{3}$
introduced there, we note that $d\zeta _{3}$ is a generator of $\mathcal{L}%
_{an}=\omega _{\mathcal{A}}(\bar{\Sigma}).$ For reasons that will become
clear later (\emph{cf }Section \ref{Fields of rationality}) we use $2\pi
i\cdot d\zeta _{3}$ to trivialize $\mathcal{L}_{an}$ over $\frak{X}$.
Suppose 
\begin{equation}
\gamma =\left( 
\begin{array}{lll}
a_{1} & b_{1} & c_{1} \\ 
a_{2} & b_{2} & c_{2} \\ 
a_{3} & b_{3} & c_{3}
\end{array}
\right) \in \Gamma \subset SU_{\infty }.
\end{equation}
If $\gamma (z,u)=(z^{\prime },u^{\prime })$ then 
\begin{equation}
z^{\prime }=\frac{a_{1}z+b_{1}u+c_{1}}{a_{3}z+b_{3}u+c_{3}},\,\,\,u^{\prime
}=\frac{a_{2}z+b_{2}u+c_{2}}{a_{3}z+b_{3}u+c_{3}}
\end{equation}
and 
\begin{equation}
\gamma \left( 
\begin{array}{l}
z \\ 
u \\ 
1
\end{array}
\right) =j(\gamma ;z,u)\left( 
\begin{array}{l}
z^{\prime } \\ 
u^{\prime } \\ 
1
\end{array}
\right) ,\,\,\,\,j(\gamma ;z,u)=a_{3}z+b_{3}u+c_{3}.
\end{equation}

\begin{lemma}
The following relation holds for every $\gamma \in U_{\infty }$%
\begin{equation}
\lambda (z,u)=\lambda (\gamma (z,u))\cdot |j(\gamma ;z,u)|^{2}.
\end{equation}
\end{lemma}

\begin{proof}
Let $v=v(z,u)=\,^{t}(z,u,1).$ Then 
\begin{equation}
\lambda (z,u)=-(v,v).
\end{equation}
As $v(\gamma (z,u))=j(\gamma ;z,u)^{-1}\cdot \gamma (v(z,u))$ the lemma
follows from $(\gamma v,\gamma v)=(v,v).$
\end{proof}

Let $\mathcal{V}=Lie(\mathcal{A}/\frak{X})=\omega _{\mathcal{A}/\frak{X}%
}^{\vee }$ and $\mathcal{W}=\mathcal{V}(\bar{\Sigma})=\mathcal{L}_{an}^{\vee
}$ (a line bundle). At a point $x=(z,u)\in \frak{X}$ the fiber $\mathcal{V}%
_{x}$ is identified canonically with $(V_{\Bbb{R}},J_{x})$ and then $%
\mathcal{W}_{x}=W_{x}=\Bbb{C} \cdot \, ^{t}(z,u,1).$

\begin{proposition}
For $x=(z,u)\in \frak{X}$ let 
\begin{equation}
v_{3}(z,u)=\lambda (z,u)^{-1}\left( 
\begin{array}{l}
z \\ 
u \\ 
1
\end{array}
\right) \in \mathcal{W}_{x}.
\end{equation}
Then (i) $v_{3}(z,u)$ is a nowhere vanishing holomorphic section of $%
\mathcal{W},$ (ii) $\left\langle d\zeta _{3},v_{3}\right\rangle \equiv 1,$
(iii) the automorphy factor corresponding to $d\zeta _{3}$ is the function $%
j(\gamma ;z,u).$
\end{proposition}

\begin{proof}
Since, by construction, $d\zeta _{3}$ is a nowhere vanishing holomorphic
section of $\mathcal{L}$ (over $\frak{X}$), (i) follows from (ii). To prove
(ii) we transfer $v_{3}(z,u)$ to the moving lattice model and get $%
^{t}(0,0,1),$ which is the dual vector to $d\zeta _{3}.$ To prove (iii) we
compute in $V_{\Bbb{R}}$ (with the original complex structure!) 
\begin{equation}
\frac{\gamma _{*}v_{3}(z,u)}{v_{3}(\gamma (z,u))}=\frac{\lambda (\gamma
(z,u))}{\lambda (z,u)}j(\gamma ;z,u)=\overline{j(\gamma ;z,u)}^{-1},
\end{equation}
and recall that since $W_{\gamma (z,u)}$ is precisely the line where the
complex structure in $(V_{\Bbb{R}},J_{\gamma (z,u)})$ has been reversed, in $%
(V_{\Bbb{R}},J_{\gamma (z,u)})$ we have 
\begin{equation}
\frac{\gamma _{*}v_{3}(z,u)}{v_{3}(\gamma (z,u))}=j(\gamma ;z,u)^{-1}.
\end{equation}
Dualizing, we get ($x=(z,u)$) 
\begin{equation}
\frac{(\gamma ^{-1})^{*}d\zeta _{3}|_{x}}{d\zeta _{3}|_{\gamma (x)}}%
=j(\gamma ,x).
\end{equation}
This concludes the proof.
\end{proof}

Consider next the plane bundle $\mathcal{P}_{an}.$ As we will only be
interested in scalar-valued modular forms, we do not compute its
matrix-valued factor of automorphy (but see [Sh2]). It is important to know,
however, that the line bundle $\det \mathcal{P}_{an}$ gives nothing new.

\begin{proposition}
There is an isomorphism of analytic line bundles over $X_{\Gamma },$%
\begin{equation}
\det \mathcal{P}_{an}\simeq \mathcal{L}_{an}.
\end{equation}
Moreover, $d\zeta _{1}\wedge d\zeta _{2}$ is a nowhere vanishing holomorphic
section of $\det \mathcal{P}_{an}$ over $\frak{X}$, and the factor of
automorphy corresponding to it is $j(\gamma ;z,u).$
\end{proposition}

\begin{proof}
Since a holomorphic line bundle on $X_{\Gamma }=\Gamma \backslash \frak{X}$
is determined, up to an isomorphism, by its factor of automorphy, and $%
j(\gamma ;z,u)$ is the factor of automorphy of $\mathcal{L}_{an}$
corresponding to $d\zeta _{3},$ it is enough to prove the second statement.
Let $\mathcal{U}=\mathcal{V}(\Sigma )$ be the plane bundle dual to $\mathcal{%
P}_{an}.$ Let 
\begin{equation}
v_{1}(z,u)=-\lambda (z,u)^{-1}\left( 
\begin{array}{l}
\bar{u}z \\ 
(z-\bar{z})/\delta \\ 
\bar{u}
\end{array}
\right)
\end{equation}
and 
\begin{equation}
v_{2}(z,u)=-\lambda (z,u)^{-1}\left( 
\begin{array}{l}
\bar{z}+\delta u\bar{u} \\ 
u \\ 
1
\end{array}
\right)
\end{equation}
(considered as vectors in $(V_{\Bbb{R}},J_{x})=\mathcal{V}_{x}$). As we have
seen in (\ref{moving lattice model}), these two vector fields are sections
of $\mathcal{U}$ and at each point $x\in \frak{X}$ form a basis dual to $%
d\zeta _{1}$ and $d\zeta _{2}.$ It follows that they are holomorphic
sections, and that $v_{1}\wedge v_{2}$ is the basis dual to $d\zeta
_{1}\wedge d\zeta _{2}.$ We must show that the factor of automorphy
corresponding to $v_{1}\wedge v_{2}$ is $j(\gamma ;z,u)^{-1},$ i.e. that 
\begin{equation}
\frac{\gamma _{*}(v_{1}\wedge v_{2}(z,u))}{v_{1}\wedge v_{2}(\gamma (z,u))}%
=j(\gamma ;z,u)^{-1}.
\end{equation}
Working in $V_{\Bbb{R}}=\Bbb{C}^{3}$ (with the original complex structure) 
\begin{eqnarray}
\frac{\gamma _{*}(v_{1}\wedge v_{2}(z,u))}{v_{1}\wedge v_{2}(\gamma (z,u))}%
\cdot \frac{1}{\overline{j(\gamma ;z,u)}} &=&  \notag \\
\frac{\gamma _{*}(v_{1}\wedge v_{2}(z,u))}{v_{1}\wedge v_{2}(\gamma (z,u))}%
\cdot \frac{\gamma _{*}v_{3}(z,u)}{v_{3}(\gamma (z,u))} &=&\frac{\gamma
_{*}(v_{1}\wedge v_{2}\wedge v_{3}(z,u))}{v_{1}\wedge v_{2}\wedge
v_{3}(\gamma (z,u))}.
\end{eqnarray}
But 
\begin{equation}
v_{1}\wedge v_{2}\wedge v_{3}(z,u)=\delta \lambda (z,u)^{-1}e_{1}\wedge
e_{2}\wedge e_{3},
\end{equation}
because 
\begin{equation}
\det \left( 
\begin{array}{lll}
\bar{u}z & \bar{z}+\delta u\bar{u} & z \\ 
(z-\bar{z})/\delta & u & u \\ 
\bar{u} & 1 & 1
\end{array}
\right) =\delta \lambda (z,u)^{2}.
\end{equation}
As $\det (\gamma )=1,$ this gives 
\begin{equation}
\frac{\gamma _{*}(v_{1}\wedge v_{2}(z,u))}{v_{1}\wedge v_{2}(\gamma (z,u))}%
\cdot \frac{1}{\overline{j(\gamma ;z,u)}}=\frac{\lambda (\gamma (z,u))}{%
\lambda (z,u)}=\frac{1}{j(\gamma ;z,u)\overline{j(\gamma ;z,u)}},
\end{equation}
and the proof is complete.
\end{proof}

\subsubsection{The relation $\det \mathcal{P}\simeq \mathcal{L}$ over $\bar{S%
}_{\mathcal{K}}$}

The isomorphism between $\det \mathcal{P}$ and $\mathcal{L}$ is in fact
algebraic, and even extends to the generic fiber $\bar{S}_{\mathcal{K}}$ of
the smooth compactification.

\begin{proposition}
\label{detP = L}One has $\det \mathcal{P}\simeq \mathcal{L}$ over $\bar{S}_{%
\mathcal{K}}.$
\end{proposition}

\begin{proof}
Since $Pic(\bar{S}_{\mathcal{K}})\subset Pic(\bar{S}_{\Bbb{C}})$ it is
enough to prove the proposition over $\Bbb{C}.$ By GAGA, it is enough to
establish the triviality of $\det \mathcal{P}\otimes \mathcal{L}^{-1}$ in
the analytic category. For each connected component $X_{\Gamma }$ of $S_{%
\Bbb{C}}$, the section $(d\zeta _{1}\wedge d\zeta _{2})\otimes d\zeta
_{3}^{-1}$ descends from $\frak{X}$ to $X_{\Gamma },$ because $d\zeta
_{1}\wedge d\zeta _{2}$ and $d\zeta _{3}$ have the same factor of automorphy 
$j(\gamma ,x)$ ($\gamma \in \Gamma ,$ $x\in \frak{X}$). This section is
nowhere vanishing on $X_{\Gamma },$ and extends to a nowhere vanishing
section on $\bar{X}_{\Gamma },$ trivializing $\det \mathcal{P}\otimes 
\mathcal{L}^{-1}$. In fact, if $c$ is the standard cusp, $d\zeta _{1}\wedge
d\zeta _{2}$ and $d\zeta _{3}$ are already well-defined and nowhere
vanishing sections of $\det \mathcal{P}$ and $\mathcal{L}$ in the
neighborhood 
\begin{equation}
\overline{\Gamma _{cusp}\backslash \Omega _{R}}=(\Gamma _{cusp}\backslash
\Omega _{R})\cup E_{c}
\end{equation}
of $E_{c}$ (see \ref{complex smooth compactification}). This is a
consequence of the fact that $j(\gamma ,x)=1$ for $\gamma \in \Gamma
_{cusp}. $

An alternative proof is to use Theorem 4.8 of [Ha]. In our case it gives a 
\emph{functor} $\mathcal{V}\mapsto [\mathcal{V]}$ from the category of $%
\mathbf{G}(\Bbb{C})$-equivariant vector bundles on the compact dual $\Bbb{P}%
_{\Bbb{C}}^{2}$ of $Sh_{K}$ to the category of vector bundles with $\mathbf{G%
}(\Bbb{A}_{f})$-action on the inverse system of Shimura varieties $Sh_{K}.$
Here $\Bbb{P}_{\Bbb{C}}^{2}=\mathbf{G}(\Bbb{C})/\mathbf{H}(\Bbb{C}),$ where $%
\mathbf{H}(\Bbb{C})$ is the parabolic group stabilizing the line $\Bbb{%
C\cdot }$ $^{t}(\delta /2,0,1)$ in $\mathbf{G}(\Bbb{C})=GL_{3}(\Bbb{C}%
)\times \Bbb{C}^{\times },$ and the irreducible $\mathcal{V}$ are associated
with highest weight representations of the Levi factor $\mathbf{L}(\Bbb{C})$
of $\mathbf{H}(\Bbb{C}).$ It is straightforward to check that $\det \mathcal{%
P}$ and $\mathcal{L}$ are associated with the same character of $\mathbf{L}(%
\Bbb{C}),$ up to a twist by a character of $\mathbf{G}(\Bbb{C})$, which
affects the $\mathbf{G}(\Bbb{A}_{f})$-action (hence the normalization of
Hecke operators), but not the structure of the line bundles themselves. The
functoriality of Harris' construction implies that $\det \mathcal{P}$ and $%
\mathcal{L}$ are isomorphic also algebraically.
\end{proof}

We de not know if $\det \mathcal{P}$ and $\mathcal{L}$ are isomorphic as
algebraic line bundles over $S.$ This would be equivalent, by (\ref{detP}),
to the statement that for every PEL structure $(A,\lambda ,\iota ,\alpha
)\in \mathcal{M}(R),$ for any $R_{0}$-algebra $R,$ $\det
(H_{dR}^{1}(A/R)(\Sigma ))$ is the trivial line bundle on $Spec(R)$. To our
regret, we have not been able to establish this, although a similar
statement in the ``Siegel case'', namely that for any principally polarized
abelian scheme $(A,\lambda )$ over $R,$ $\det H_{dR}^{1}(A/R)$ is trivial,
follows at once from the Hodge filtration (\ref{Hodge}). Our result,
however, suffices to guarantee the following corollary, which is all that we
will be using in the sequel.

\begin{corollary}
For any characteristic $p$ geometric point $Spec(k)\rightarrow Spec(R_{0}),$
we have $\det \mathcal{P}\simeq \mathcal{L}$ on $\bar{S}_{k}$. A similar
statement holds for morphisms $SpecW(k)\rightarrow Spec(R_{0}).$
\end{corollary}

\begin{proof}
Since $\bar{S}$ is a regular scheme, $\det \mathcal{P}\otimes \mathcal{L}%
^{-1}\simeq \mathcal{O}(D)$ for a Weil divisor $D$ supported on vertical
fibers over $R_{0}.$ Since any connected component $Z$ of $\bar{S}_{k}$ is
irreducible, we can modify $D$ so that $D$ and $Z$ are disjoint, showing
that $\det \mathcal{P}\otimes \mathcal{L}^{-1}|_{Z}$ is trivial. The second
claim is proved similarly.
\end{proof}

\subsubsection{Modular forms}

Let $R$ be an $R_{0}$-algebra. A \emph{modular form} of weight $k\ge 0$ and
level $N\ge 3$ defined over $R$ is an element of the finite $R$-module 
\begin{equation}
M_{k}(N,R)=H^{0}(\bar{S}_{R},\mathcal{L}^{k}).
\end{equation}
We usually omit the subscript $R,$ remembering that $\bar{S}$ is now to be
considered over $R.$ The well-known Koecher principle says that $H^{0}(\bar{S%
},\mathcal{L}^{k})=H^{0}(S,\mathcal{L}^{k})$. See [Bel], Section 2.2, for an
arithmetic proof valid integrally over any $R_{0}$-algebra $R$. A \emph{cusp}
\emph{form }is an element of the space 
\begin{equation}
M_{k}^{0}(N,R)=H^{0}(\bar{S},\mathcal{L}^{k}\otimes \mathcal{O}(C)^{\vee }).
\end{equation}
As we shall see below (\emph{cf }Corollary \ref{KS}), if $k\ge 3,$ there is
an isomorphism $\mathcal{L}^{k}\otimes \mathcal{O}(C)^{\vee }\simeq \Omega _{%
\bar{S}}^{2}\otimes \mathcal{L}^{k-3}.$ In particular, cusp forms of weight
3 are ``the same'' as holomorphic 2-forms on $\bar{S}$.

An alternative definition (\`{a} la Katz) of a modular form of weight $k$
and level $N$ defined over $R$, is as a ``rule'' $f$ which assigns to every $%
R$-scheme $T,$ and every $\underline{A}=(A,\lambda ,\iota ,\alpha )\in 
\mathcal{M}(T),$ together with a nowhere vanishing section $\omega \in
H^{0}(T,\omega _{A/T}(\bar{\Sigma})),$ an element $f(\underline{A},\omega
)\in H^{0}(T,\mathcal{O}_{T})$ satisfying

\begin{itemize}
\item  $f(\underline{A},\lambda \omega )=\lambda ^{-k}f(\underline{A},\omega
)$ for every $\lambda \in H^{0}(T,\mathcal{O}_{T})^{\times }$

\item  The ``rule'' $f$ is compatible with base change $T^{\prime }/T.$
\end{itemize}

Indeed, if $f$ is an element of $M_{k}(N,R),$ then given such an $\underline{%
A}$ and $\omega ,$ the universal property of $S$ produces a unique morphism $%
\varphi :T\rightarrow S$ over $R,$ $\varphi ^{*}\mathcal{A}=A,$ and we may
let $f(\underline{A},\omega )=\varphi ^{*}f/\omega ^{k}.$ Conversely, given
such a rule $f$ we may cover $S$ by Zariski open sets $T$ where $\mathcal{L}$
is trivialized, and then the sections $f(\mathcal{A}_{T},\omega _{T})\omega
_{T}^{k}$ ($\omega _{T}$ a trivializing section over $T$) glue to give $f\in
M_{k}(N,R).$ While viewing $f$ as a ``rule'' rather than a section is a
matter of language, it is sometimes more convenient to use this language.

Let $R\rightarrow R^{\prime }$ be a homomorphism of $R_{0}$-algebras. Then
Bella\"{i}che proved the following theorem ([Bel], 1.1.5).

\begin{theorem}
If $k\ge 3$ (resp. $k\ge 6$) then $M_{k}^{0}(N,R)$ (resp. $M_{k}(N,R)$) is a
locally free finite $R$-module, and the base-change homomorphism 
\begin{equation}
R^{\prime }\otimes M_{k}^{0}(N,R)\simeq M_{k}^{0}(N,R^{\prime })
\end{equation}
is an isomorphism (resp. base change for $M_{k}(N,R)$).
\end{theorem}

Bella\"{i}che considers only weights divisible by 3, but his proofs
generalize to all $k$ (\emph{cf} remark on the bottom of p.43 in [Bel]).

Over $\Bbb{C},$ pulling back to $\frak{X}$ and using the trivialization of $%
\mathcal{L}$ given by the nowehere vanishing section $2\pi i\cdot d\zeta
_{3} $, a modular form of weight $k$ is a collection $(f_{j})_{1\le j\le m}$
of holomorphic functions on $\frak{X}$ satisfying 
\begin{equation}
f_{j}(\gamma (z,u))=j(\gamma ;z,u)^{k}f_{j}(z,u)\,\,\,\,\forall \gamma \in
\Gamma _{j}
\end{equation}
(the Koecher principle means that no condition has to be imposed at the
cusps).

\subsection{The Kodaira Spencer isomorphism}

Let $\pi :A\rightarrow S$ be an abelian scheme of relative dimension 3, as
in the Picard moduli problem. The Gauss-Manin connection 
\begin{equation}
\nabla :H_{dR}^{1}(A/S)\rightarrow H_{dR}^{1}(A/S)\otimes _{\mathcal{O}%
_{S}}\Omega _{S}^{1}
\end{equation}
defines the Kodaira-Spencer map 
\begin{equation}
KS\in Hom_{\mathcal{O}_{S}}(\omega _{A}\otimes _{\mathcal{O}_{S}}\omega
_{A^{t}},\Omega _{S}^{1})
\end{equation}
as the composition of the maps 
\begin{eqnarray}
\omega _{A} &=&H^{0}(A,\Omega _{A/S}^{1})\hookrightarrow H_{dR}^{1}(A/S)%
\overset{\nabla }{\rightarrow }H_{dR}^{1}(A/S)\otimes _{\mathcal{O}%
_{S}}\Omega _{S}^{1}  \notag \\
&\twoheadrightarrow &R^{1}\pi _{*}\mathcal{O}_{A}\otimes _{\mathcal{O}%
_{S}}\Omega _{S}^{1}\simeq \omega _{A^{t}}^{\vee }\otimes _{\mathcal{O}%
_{S}}\Omega _{S}^{1},
\end{eqnarray}
and finally using $Hom(L,M^{\vee }\otimes N)=Hom(L\otimes M,N).$ Recall that
if $A$ is endowed with an $\mathcal{O}_{\mathcal{K}}$ action via $\iota $
then the induced action of $a\in \mathcal{O}_{K}$ on $A^{t}$ is induced from
the action on $Pic(A),$ taking a line bundle $\mathcal{M}$ to $\iota (a)^{*}%
\mathcal{M}$. As the polarization $\lambda :A\rightarrow A^{t}$ is $\mathcal{%
O}_{S}$-linear but satisfies $\lambda \circ \iota (a)=\iota (a^{\rho })\circ
\lambda ,$ it follows that the induced $\mathcal{O}_{\mathcal{K}}$ action on 
$A^{t}$ is of type $(1,2)$, hence $\omega _{A^{t}}^{\vee }$ is of type $%
(1,2).$

\begin{lemma}
The map $KS$ induces maps 
\begin{eqnarray}
KS(\Sigma ):\omega _{A}(\Sigma )\rightarrow \omega _{A^{t}}^{\vee }(\Sigma
)\otimes _{\mathcal{O}_{S}}\Omega _{S}^{1}  \notag \\
KS(\bar{\Sigma}):\omega _{A}(\bar{\Sigma})\rightarrow \omega _{A^{t}}^{\vee
}(\bar{\Sigma})\otimes _{\mathcal{O}_{S}}\Omega _{S}^{1}
\end{eqnarray}
hence maps, denoted by the same symbols, 
\begin{eqnarray}
KS(\Sigma ):\omega _{A}(\Sigma )\otimes _{\mathcal{O}_{S}}\omega
_{A^{t}}(\Sigma )\rightarrow \Omega _{S}^{1}  \notag \\
KS(\bar{\Sigma}):\omega _{A}(\bar{\Sigma})\otimes _{\mathcal{O}_{S}}\omega
_{A^{t}}(\bar{\Sigma})\rightarrow \Omega _{S}^{1}.
\end{eqnarray}
The CM-type-reversing isomorphism $\lambda ^{*}:\omega _{A^{t}}\rightarrow
\omega _{A}$ induced by the principal polarization satisfies 
\begin{equation}
KS(\Sigma )(\lambda ^{*}x\otimes y)=KS(\bar{\Sigma})(\lambda ^{*}y\otimes x)
\end{equation}
for all $x\in \omega _{A^{t}}(\bar{\Sigma})$ and $y\in \omega
_{A^{t}}(\Sigma ).$
\end{lemma}

\begin{proof}
The first claim follows from the fact that the Gauss-Manin connection
commutes with the endomorphisms, hence preserves CM types. The second claim
is a consequence of the symmetry of the polarization, see [Fa-Ch], Prop. 9.1
on p.81 (in the Siegel modular case).
\end{proof}

Observe that $\omega _{A}(\Sigma )\otimes _{\mathcal{O}_{S}}\omega
_{A^{t}}(\Sigma )$, as well as $\omega _{A}(\bar{\Sigma})\otimes _{\mathcal{O%
}_{S}}\omega _{A^{t}}(\bar{\Sigma}),$ are vector bundles of rank 2.

\begin{lemma}
If $S$ is the Picard modular surface and $A=\mathcal{A}$ is the universal
abelian variety, then 
\begin{equation}
KS(\Sigma ):\omega _{\mathcal{A}}(\Sigma )\otimes _{\mathcal{O}_{S}}\omega _{%
\mathcal{A}^{t}}(\Sigma )\rightarrow \Omega _{S}^{1}
\end{equation}
is an isomorphism, and so is $KS(\bar{\Sigma}).$
\end{lemma}

\begin{proof}
This is well-known and follows from deformation theory. For a self-contained
proof, see [Bel], Prop. II.2.1.5.
\end{proof}

\begin{proposition}
The Kodaira-Spencer map induces a canonical isomorphism of vector bundles
over $S$%
\begin{equation}
\mathcal{P}\otimes \mathcal{L}\simeq \Omega _{S}^{1}.
\end{equation}
\end{proposition}

\begin{proof}
We need only use $\lambda ^{*}$ to identify $\omega _{\mathcal{A}%
^{t}}(\Sigma )$ with $\omega _{\mathcal{A}}(\bar{\Sigma}).$
\end{proof}

We refer to Corollary \ref{KS at cusps} for an extension of this result to $%
\bar{S}.$

\begin{corollary}
There is an isomorphism of line bundles $\mathcal{L}^{3}\simeq \Omega
_{S}^{2}.\label{KS}$
\end{corollary}

\begin{proof}
Take determinants and use $\det \mathcal{P}\simeq \mathcal{L}.$ We emphasize
that while $KS(\Sigma )$ is canonical, the identification of $\det \mathcal{P%
}$ with $\mathcal{L}$ depends on a choice, which we shall fix later on once
and for all.
\end{proof}

The last corollary should be compared to the case of the open modular curve $%
Y(N)$, where the \emph{square} of the Hodge bundle $\omega _{\mathcal{E}}$
of the universal elliptic curve becomes isomorphic to $\Omega _{Y(N)}^{1}.$
Over $\Bbb{C},$ as the isomorphism between $\mathcal{L}^{3}$ and $\Omega
_{S}^{2}$ takes $d\zeta _{3}^{\otimes 3}$ to a constant multiple of $%
dz\wedge du$ (see Corollary \ref{L^3}), the differential form corresponding
to a modular form $($ $f_{j})_{1\le j\le m}$ of weight 3, is (up to a
constant) $(f_{j}(z,u)dz\wedge du)_{1\le j\le m}.$

\subsection{Extensions to the boundary of $S$}

\subsubsection{The vector bundles $\mathcal{P}$ and $\mathcal{L}$ over $C%
\label{PL over C}$}

Let $E\subset C_{R_{N}}$ be a connected component of the cuspidal divisor
(over the integral closure $R_{N}$ of $R_{0}$ in the ray class field $%
\mathcal{K}_{N}).$ As we have seen, $E$ is an elliptic curve with CM by $%
\mathcal{O}_{\mathcal{K}}.$ If the cusp at which $E$ sits is of type $(\frak{%
a},B)$ ($\frak{a}$ an ideal of $\mathcal{O}_{\mathcal{K}},$ $B$ an elliptic
curve with CM by $\mathcal{O}_{\mathcal{K}}$ defined over $R_{N})$ then $E$
maps via an isogeny to $\delta _{\mathcal{K}}\frak{a}\otimes _{\mathcal{O}_{%
\mathcal{K}}}B^{t}=Ext_{\mathcal{O}_{\mathcal{K}}}^{1}(B,\frak{a}\otimes 
\Bbb{G}_{m}).$ In particular, $E$ and $B$ are isogenous over $\mathcal{K}%
_{N}.$

Consider $\mathcal{G},$ the universal semi-abelian $\mathcal{O}_{\mathcal{K}%
} $-three-fold of type $(\frak{a},B),$ over $\delta _{\mathcal{K}}\frak{a}%
\otimes _{\mathcal{O}_{\mathcal{K}}}B^{t}.$ The semi-abelian scheme $%
\mathcal{A}$ over $E$ is the pull-back of this $\mathcal{G}.$ Clearly, $%
\omega _{\mathcal{A}/E}=\mathcal{P}\oplus \mathcal{L}$ and $\mathcal{P}%
=\omega _{\mathcal{A}/E}(\Sigma )$ admits over $E$ a canonical rank 1
sub-bundle $\mathcal{P}_{0}=\omega _{B}.$ Since the toric part and the
abelian part of $\mathcal{G}$ are constant, $\mathcal{L},\mathcal{P}_{0}$
and $\mathcal{P}_{\mu }=\mathcal{P}/\mathcal{P}_{0}$ are all trivial line
bundles when restricted to $E$. It can be shown that $\mathcal{P}$ itself is
not trivial over $E$.

\subsubsection{More identities over $\bar{S}$}

We have seen that $\Omega _{S}^{2}\simeq \mathcal{L}^{3}.$ For the following
proposition, compare [Bel], Lemme II.2.1.7.

\begin{proposition}
\label{canonical class}Working over $\mathcal{K}_{N},$ let $E_{j}$ ($1\le
j\le h$) be the connected components of $C$. Then 
\begin{equation}
\Omega _{\bar{S}}^{2}\simeq \mathcal{L}^{3}\otimes \bigotimes_{j=1}^{h}%
\mathcal{O}(E_{j})^{\vee }.
\end{equation}
\end{proposition}

\begin{proof}
By [Hart] II.6.5, $\Omega _{\bar{S}}^{2}\simeq \mathcal{L}^{3}\otimes
\bigotimes_{j=1}^{h}\mathcal{O}(E_{j})^{n_{j}}$ for some integers $n_{j}$
and we want to show that $n_{j}=-1$ for all $j.$ By the adjunction formula
on the smooth surface $\bar{S},$ if we denote by $K_{\bar{S}}$ a canonical
divisor, $\mathcal{O}(K_{\bar{S}})=\Omega _{\bar{S}}^{2},$ then 
\begin{equation}
0=2g_{E_{j}}-2=E_{j}.(E_{j}+K_{\bar{S}}).
\end{equation}
We conclude that 
\begin{equation}
\deg (\Omega _{\bar{S}}^{2}|_{E_{j}})=E_{j}.K_{\bar{S}}=-E_{j}.E_{j}>0.
\end{equation}
Here $E_{j}.E_{j}<0$ because $E_{j}$ can be contracted to a point (Grauert's
theorem). As $\mathcal{L}|_{E_{j}}$ and $\mathcal{O}(E_{i})|_{E_{j}}$ $%
(i\neq j)$ are trivial we get 
\begin{equation}
-E_{j}.E_{j}=n_{j}E_{j}.E_{j},
\end{equation}
hence $n_{j}=-1$ as desired.
\end{proof}

\subsection{Fourier-Jacobi expansions}

\subsubsection{The infinitesimal retraction\label{retraction}}

We follow the arithmetic theory of Fourier-Jacobi expansions as developed in
[Bel]. Let $\widehat{S}$ be the formal completion of $\bar{S}$ along the
cuspidal divisor $C=\bar{S}\backslash S.$ We work over $R_{0},$ and denote
by $C^{(n)}$ the $n$-th infinitesimal neighborhood of $C$ in $\bar{S}.$ The
closed immersion $i:C\hookrightarrow \widehat{S}$ admits a canonical left
inverse $r:\widehat{S}\rightarrow C,$ a \emph{retraction} satisfying $r\circ
i=Id_{C}.$ This is not automatic, but rather a consequence of the rigidity
of tori, as explained in [Bel], Proposition II.2.4.2. As a corollary, the
universal semi-abelian scheme $\mathcal{A}_{/C^{(n)}}$ is the pull-back of $%
\mathcal{A}_{/C}$ via $r.$ The same therefore holds for $\mathcal{P}$ and $%
\mathcal{L},$ namely there are natural isomorphisms $r^{*}(\mathcal{P}%
|_{C})\simeq \mathcal{P}|_{C^{(n)}}$ and $r^{*}(\mathcal{L}|_{C})\simeq 
\mathcal{L}|_{C^{(n)}}.$ As a consequence, the filtration 
\begin{equation}
0\rightarrow \mathcal{P}_{0}\rightarrow \mathcal{P}\rightarrow \mathcal{P}%
_{\mu }\rightarrow 0
\end{equation}
extends canonically to $C^{(n)}.$ Since $\mathcal{L},\mathcal{P}_{0}$ and $%
\mathcal{P}_{\mu }$ are trivial on $C,$ they are trivial over $C^{(n)}$ as
well.

\subsubsection{Arithmetic Fourier-Jacobi expansions\label{AFJ}}

We fix an arbitrary noetherian $R_{0}$-algebra $R$ and consider all our
schemes over $R,$ without a change in notation. As usual, we let $\mathcal{O}%
_{\widehat{S}}=\lim_{\leftarrow }\mathcal{O}_{C^{(n)}}$ (a sheaf in the
Zariski topology on $C$). Via $r^{*},$ this is a sheaf of $\mathcal{O}_{C}$%
-modules. Choose a global nowhere vanishing section $s\in H^{0}(C,\mathcal{L}%
)$ trivializing $\mathcal{L}.$ Such a section is unique up to a unit of $R$
on each connected component of $C$. This $s$ determines an isomorphism 
\begin{equation}
\mathcal{L}^{k}|_{\widehat{S}}\simeq \mathcal{O}_{\widehat{S}%
},\,\,\,\,f\mapsto f/(r^{*}s)^{k}
\end{equation}
for each $k,$ hence a ring homomorphism 
\begin{equation}
FJ:\oplus _{k=0}^{\infty }M_{k}(N,R)\rightarrow H^{0}(C,\mathcal{O}_{%
\widehat{S}}).
\end{equation}
We call $FJ(f)$ the (arithmetic) Fourier-Jacobi expansion of $f.$ It depends
on $s$ in an obvious way.

To understand the structure of $H^{0}(C,\mathcal{O}_{\widehat{S}})$ let $%
\mathcal{I}\subset \mathcal{O}_{\bar{S}}$ be the sheaf of ideals defining $%
C, $ so that $C^{(n)}$ is defined by $\mathcal{I}^{n}.$ The conormal sheaf $%
\mathcal{N}=\mathcal{I}/\mathcal{I}^{2}$ is the restriction $i^{*}\mathcal{O}%
_{\bar{S}}(-C)$ of $\mathcal{I}=\mathcal{O}_{\bar{S}}(-C)$ to $C.$ It is an
ample invertible sheaf on $C,$ since (over $R_{N}$) its degree on each
component $E_{j}$ is $-E_{j}^{2}>0.$

Now $r^{*}$ supplies, for every $n\ge 2,$ a canonical splitting of 
\begin{equation}
0\rightarrow \mathcal{I}/\mathcal{I}^{n}\rightarrow \mathcal{O}_{\bar{S}}/%
\mathcal{I}^{n}\overset{\curvearrowleft }{\rightarrow }\mathcal{O}_{\bar{S}}/%
\mathcal{I}\rightarrow 0.
\end{equation}
Inductively, we get a direct sum decomposition 
\begin{equation}
\mathcal{O}_{\bar{S}}/\mathcal{I}^{n}\simeq \bigoplus_{m=0}^{n-1}\mathcal{I}%
^{m}/\mathcal{I}^{m+1}
\end{equation}
as $\mathcal{O}_{C}$-modules, hence, since $\mathcal{I}^{m}/\mathcal{I}%
^{m+1}\simeq \mathcal{N}^{m}$, an isomorphism 
\begin{equation}
H^{0}(C,\mathcal{O}_{C^{(n)}})\simeq \bigoplus_{m=0}^{n-1}H^{0}(C,\mathcal{N}%
^{m}),\,\,\,\,f\mapsto \sum_{m=0}^{n-1}c_{m}(f).
\end{equation}
This isomorphism respects the multiplicative structure, so is a ring
isomorphism. Going to the projective limit, and noting that the $c_{m}(f)$
are independent of $n,$ we get 
\begin{equation}
FJ(f)=\sum_{m=0}^{\infty }c_{m}(f)\in \prod_{m=0}^{\infty }H^{0}(C,\mathcal{N%
}^{m}).  \label{FJ}
\end{equation}

\subsubsection{Fourier-Jacobi expansions over $\Bbb{C}$}

Working over $\Bbb{C},$ we shall now relate the infinitesimal retraction $r$
to the geodesic retraction, and the powers of the conormal bundle $\mathcal{N%
}$ to theta functions. Recall the analytic compactification of $X_{\Gamma }$
described in Proposition \ref{disk bundle}. Let $E$ be the connected
component of $\bar{X}_{\Gamma }\backslash X_{\Gamma }$ corresponding to the
standrad cusp $c_{\infty }$. As before, we denote by $E^{(n)}$ its $n$th
infinitesimal neighborhood. The line bundle $\mathcal{T}|_{E}$ is just the
analytic normal bundle to $E$, hence we have an isomorphism 
\begin{equation}
\mathcal{N}_{an}\simeq \mathcal{T}^{\vee }
\end{equation}
between the analytification of $\mathcal{N}=\mathcal{I}/\mathcal{I}^{2}$ and
the dual of $\mathcal{T}$.

\begin{lemma}
The infinitesimal retraction $r:E^{(n)}\rightarrow E$ coincides with the map
induced by the geodesic retraction (\ref{geodesic}).
\end{lemma}

\begin{proof}
The meaning of the lemma is this. The infinitesimal retraction induces a map
of ringed spaces 
\begin{equation}
r_{an}:E_{an}^{(n)}\rightarrow E_{an}
\end{equation}
where $E_{an}$ is the analytic space associated to $E$ with its sheaf of
analytic functions $\mathcal{O}_{E}^{hol}$, and $E_{an}^{(n)}$ is the same
topological space with the sheaf $\mathcal{O}_{\bar{S}}^{hol}/\mathcal{I}%
_{an}^{n}.$ The geodesic retraction (sending $(z,u)$ to $u\mod \Lambda $%
) is an analytic map $r_{geo}:E_{an}(\varepsilon )\rightarrow E_{an},$ where 
$E_{an}(\varepsilon )$ is our notation for some tubular neighborhood of $%
E_{an}$ in $\bar{S}_{an}.$ On the other hand, there is a canonical map $can$
of ringed spaces from $E_{an}^{(n)}$ to $E_{an}(\varepsilon ).$ We claim
that these three maps satisfy $r_{geo}\circ can=r_{an}.$

To prove the lemma, note that the infinitesimal retraction $%
r:E^{(n)}\rightarrow E$ is uniquely characterized by the fact that the $%
\mathcal{O}_{\mathcal{K}}$-semi-abelian variety $\mathcal{A}_{x}=x^{*}%
\mathcal{A}$ at any point $x:Spec(R)\rightarrow E^{(n)}$ is equal to $%
\mathcal{A}_{r\circ x}$ (an equality respecting the PEL structures). See
[Bel], II.2.4.2. The computations of Section \ref{degeneration} show that
the same is true for the infinitesimal retraction obtained from the geodesic
retraction. We conclude that the two retractions agree on the level of
``truncated Taylor expansions''.
\end{proof}

Consider now a modular form of weight $k$ and level $N$ over $\Bbb{C},$ $%
f\in M_{k}(N,\Bbb{C}).$ Using the trivialization of $\mathcal{L}_{an}$ over
the symmetric space $\frak{X}$ given by $2\pi i\cdot d\zeta _{3}$ as
discussed in Section \ref{automorphy factors}, we identify $f$ with a
collection of \emph{functions} $f_{j}$ on $\frak{X},$ transforming under $%
\Gamma _{j}$ according to the automorphy factor $j(\gamma ;z,u)^{k}.$ As
usual we look at $\Gamma =\Gamma _{1}$ only, and at the expansion of $%
f=f_{1} $ at the standard cusp $c_{\infty },$ the other cusps being in
principle similar. On the arithmetic FJ expansion side this means that we
concentrate on one connected component $E$ of $C,$ which lies on the
connected component of $S_{\Bbb{C}}$ corresponding to $g_{1}=1.$ It also
means that as the section $s$ used to trivialize $\mathcal{L}$ along $E,$ we
must use a section that, analytically, coincides with $2\pi i\cdot d\zeta
_{3}.$

Pulling back the sheaf $\mathcal{N}_{an}$ from $E=\Bbb{C}/\Lambda $ to $\Bbb{%
C},$ it is clear that $q=q(z)=e^{2\pi iz/M}$ maps, at each $u\in \Bbb{C},$
to a generator of $\mathcal{T}^{\vee }=\mathcal{N}_{an}=\mathcal{I}_{an}/%
\mathcal{I}_{an}^{2}$, and we denote by $q^{m}$ the corresponding generator
of $\mathcal{N}_{an}^{m}=\mathcal{I}_{an}^{m}/\mathcal{I}_{an}^{m+1}.$ If 
\begin{equation}
f(z,u)=\sum_{m=0}^{\infty }\theta _{m}(u)e^{2\pi imz/M}=\sum_{m=0}^{\infty
}\theta _{m}(u)q^{m}
\end{equation}
is the complex analytic Fourier expansion of $f$ at a neighborhood of $%
c_{\infty },$ then $c_{m}(z,u)=\theta _{m}(u)q^{m}\in H^{0}(E,\mathcal{N}%
_{an}^{m})$ is just the restriction of the section denoted above by $%
c_{m}(f) $ to $E.$ The functions $\theta _{m}$ are classical elliptic theta
functions (for the lattice $\Lambda $).

\subsection{The Gauss-Manin connection in a neighborhood of a cusp\label%
{KS at cusp}}

\subsubsection{A computation of $\nabla $ in the complex model}

We shall now compute the Gauss-Manin connection in the complex model near
the standard cusp $c_{\infty }.$ Recall that we use the coordinates $%
(z,u,\zeta _{1},\zeta _{2},\zeta _{3})$ as in Section \ref{moving lattice}.
Here $d\zeta _{1}$ and $d\zeta _{2}$ form a basis for $\mathcal{P}$ and $%
d\zeta _{3}$ for $\mathcal{L}.$ The same coordinates served to define also
the semi-abelian variety $\mathcal{G}_{u}$ (denoted also $\mathcal{A}_{u}$)
over the cuspidal component $E$ at $c_{\infty },$ \emph{cf} Section \ref
{degeneration}. As explained there (\ref{semi-abelian}), the projection to
the abelian part is given by the coordinate $\zeta _{1}$ (modulo $\mathcal{O}%
_{\mathcal{K}}$), so $d\zeta _{1}$ is a basis for the sub-line-bundle of $%
\omega _{\mathcal{A}/E}$ coming from the abelian part, which was denoted $%
\mathcal{P}_{0}.$ In Section \ref{retraction} above it was explained how to
extend the filtration $\mathcal{P}_{0}\subset \mathcal{P}$ canonically to
the formal neighborhood $\widehat{S}$ of $E$ using the retraction $r,$ by
pulling back from the boundary. It was also noted that complex analytically,
the retraction $r$ is the germ of the geodesic retraction introduced
earlier. From the analytic description of the degeneration of $\mathcal{A}%
_{(z,u)}$ along a geodesic, it becomes clear that $\mathcal{P}_{0}=r^{*}(%
\mathcal{P}_{0}|_{E})$ is just the line bundle $\mathcal{O}_{\widehat{S}%
}\cdot d\zeta _{1}\subset \omega _{\mathcal{A}/\widehat{S}}.$ It follows
that $\mathcal{P}_{\mu }=\mathcal{O}_{\widehat{S}}\cdot d\zeta _{2}\mod %
\mathcal{P}_{0}.$

We shall now pull back these vector bundles to $\frak{X}$, and compute the
Gauss Manin connection $\nabla $ complex analytically on $\omega _{\mathcal{A%
}/\frak{X}}$. We write $\mathcal{P}_{0}=\mathcal{O}_{\frak{X}}\cdot d\zeta
_{1}$ for $\mathcal{P}_{0,an}$ etc. dropping the decoration \emph{an. }%
Recalling that $\mathcal{O}_{\mathcal{K}}=\Bbb{Z}\oplus \Bbb{Z}\omega _{%
\mathcal{K}}$ we let 
\begin{equation}
\alpha _{1}=\left( 
\begin{array}{l}
0 \\ 
1 \\ 
1
\end{array}
\right) ,\alpha _{2}=\left( 
\begin{array}{l}
1 \\ 
0 \\ 
u
\end{array}
\right) ,\alpha _{3}=\left( 
\begin{array}{l}
u \\ 
-z/\delta \\ 
z/\delta
\end{array}
\right)
\end{equation}
and 
\begin{equation}
\alpha _{1}^{\prime }=\iota ^{\prime }(\omega _{\mathcal{K}})\alpha
_{1}=\left( 
\begin{array}{l}
0 \\ 
\omega _{\mathcal{K}} \\ 
\bar{\omega}_{\mathcal{K}}
\end{array}
\right) ,\alpha _{2}^{\prime }=\iota ^{\prime }(\omega _{\mathcal{K}})\alpha
_{2}=\left( 
\begin{array}{l}
\omega _{\mathcal{K}} \\ 
0 \\ 
\bar{\omega}_{\mathcal{K}}u
\end{array}
\right) ,
\end{equation}
\begin{equation*}
\alpha _{3}^{\prime }=\iota ^{\prime }(\omega _{\mathcal{K}})\alpha
_{3}=\left( 
\begin{array}{l}
\omega _{\mathcal{K}}u \\ 
-\omega _{\mathcal{K}}z/\delta \\ 
\bar{\omega}_{\mathcal{K}}z/\delta
\end{array}
\right) .
\end{equation*}
These 6 vectors span $L_{(z,u)}^{\prime }$ over $\Bbb{Z}.$ Let $\beta
_{1},\dots ,\beta _{3}^{\prime }$ be the dual basis to $\{\alpha _{1},\dots
,\alpha _{3}^{\prime }\}$ in $H_{dR}^{1}(\mathcal{A}/\mathcal{O}_{\frak{X}%
}), $ i.e. $\int_{\alpha _{1}}\beta _{1}=1$ etc. As the periods of the $%
\beta _{i}$'s along the integral homology are constant, the $\beta $-basis
is horizontal for the Gauss-Manin connection. The first coordinate of the $%
\alpha _{i}$ and $\alpha _{i}^{\prime }$ gives us 
\begin{equation}
d\zeta _{1}=0\cdot \beta _{1}+1\cdot \beta _{2}+u\cdot \beta _{3}+0\cdot
\beta _{1}^{\prime }+\omega _{\mathcal{K}}\cdot \beta _{2}^{\prime }+\omega
_{\mathcal{K}}u\cdot \beta _{3}^{\prime },
\end{equation}
and we find that 
\begin{equation}
\nabla (d\zeta _{1})=(\beta _{3}+\omega _{\mathcal{K}}\beta _{3}^{\prime
})\otimes du.
\end{equation}
Similarly, we find 
\begin{eqnarray}
\nabla (d\zeta _{2}) &=&-\delta ^{-1}(\beta _{3}+\omega _{\mathcal{K}}\beta
_{3}^{\prime })\otimes dz \\
\nabla (d\zeta _{3}) &=&(\beta _{2}+\bar{\omega}_{\mathcal{K}}\beta
_{2}^{\prime })\otimes du+\delta ^{-1}(\beta _{3}+\bar{\omega}_{\mathcal{K}%
}\beta _{3}^{\prime })\otimes dz.  \notag
\end{eqnarray}

\subsubsection{A computation of $KS$ in the complex model}

We go on to compute the Kodaira-Spencer map on $\mathcal{P},$ i.e. the map
denoted $KS(\Sigma ).$ For that we have to take $\nabla (d\zeta _{1})$ and $%
\nabla (d\zeta _{2})$ and project them to $R^{1}\pi _{*}\mathcal{O}_{%
\mathcal{A}}(\Sigma )\otimes \Omega _{\frak{X}}^{1}.$ We then pair the
result, using the polarization form $\left\langle ,\right\rangle _{\lambda }$
on $H_{dR}^{1}(\mathcal{A})$ (reflecting the isomorphism 
\begin{equation}
R^{1}\pi _{*}\mathcal{O}_{\mathcal{A}}(\Sigma )=Lie(\mathcal{A}^{t})(\Sigma
)=\omega _{\mathcal{A}^{t}}^{\vee }(\Sigma )\simeq \mathcal{L}^{\vee }(\rho )
\end{equation}
coming from $\lambda $), with $d\zeta _{3}.$

To perform the computation we need two lemmas.

\begin{lemma}
The Riemann form on $L_{x}^{\prime }$, associated to the polarization $%
\lambda ,$ is given in the basis $\alpha _{1},\alpha _{2},\alpha _{3},\alpha
_{1}^{\prime },\alpha _{2}^{\prime },\alpha _{3}^{\prime }$ by the matrix 
\begin{equation}
J=\left( 
\begin{array}{llllll}
&  &  &  &  & 1 \\ 
&  &  &  & -1 &  \\ 
&  &  & 1 &  &  \\ 
&  & -1 &  &  &  \\ 
& 1 &  &  &  &  \\ 
-1 &  &  &  &  & 
\end{array}
\right) .
\end{equation}
\end{lemma}

\begin{proof}
This is an easy computation using the transition map $T$ between $L$ and $%
L_{x}^{\prime }$ and the fact that on $L$ the Riemann form is the
alternating form $\left\langle ,\right\rangle =\text{Im}_{\delta }(,).$
\end{proof}

For the formulation of the next lemma recall that if $A$ is a complex
abelian variety, a polarization $\lambda :A\rightarrow A^{t}$ induces an
alternating form $\left\langle ,\right\rangle _{\lambda }$ on $H_{dR}^{1}(A)$
as well as a Riemann form on the integral homology $H_{1}(A,\Bbb{Z}).$ We
compare the two.

\begin{lemma}
Let $(A,\lambda )$ be a principally polarized complex abelian variety. If $%
\alpha _{1},\dots ,\alpha _{2g}$ is a symplectic basis for $H_{1}(A,\Bbb{Z})$
in which the associated Riemann form is given by a matrix $J$, and $\beta
_{1},\dots ,\beta _{2g}$ is the dual basis of $H_{dR}^{1}(A),$ then the
matrix of the bilinear form $\left\langle ,\right\rangle _{\lambda }$ on $%
H_{dR}^{1}(A)$ in the basis $\beta _{1},\dots ,\beta _{2g}$ is $(2\pi
i)^{-1}J.$
\end{lemma}

\begin{proof}
These are essentially Riemann's bilinear relations. For example, if $A$ is
the Jacobian of a curve $\mathcal{C}$ and the basis $\alpha _{1},\dots
,\alpha _{2g}$ has the standard intersection matrix 
\begin{equation}
J=\left( 
\begin{array}{ll}
0 & I \\ 
-I & 0
\end{array}
\right)
\end{equation}
then the lemma follows from the well-known formula for the cup product ($\xi
,\eta $ being differentials of the second kind on $\mathcal{C}$) 
\begin{equation}
\xi \cup \eta =\frac{1}{2\pi i}\sum_{i=1}^{g}\left( \int_{\alpha _{i}}\xi
\int_{\alpha _{i+g}}\eta -\int_{\alpha _{i}}\eta \int_{\alpha _{i+g}}\xi
\right) .
\end{equation}
\end{proof}

Using the two lemmas we get 
\begin{eqnarray}
KS(d\zeta _{1}\otimes d\zeta _{3}) &=&\left\langle \beta _{3}+\omega _{%
\mathcal{K}}\beta _{3}^{\prime },d\zeta _{3}\right\rangle _{\lambda }\cdot du
\\
&=& 
\begin{array}{c}
\langle \beta _{3}+\omega _{\mathcal{K}}\beta _{3}^{\prime },\,\,\beta
_{1}+u\beta _{2}+z\delta ^{-1}\beta _{3}+ \\ 
\bar{\omega}_{\mathcal{K}}\beta _{1}^{\prime }+\bar{\omega}_{\mathcal{K}%
}u\beta _{2}^{\prime }+\bar{\omega}_{\mathcal{K}}z\delta ^{-1}\beta
_{3}^{\prime }\rangle _{\lambda }\cdot du
\end{array}
\notag \\
&=&-\delta (2\pi i)^{-1}du.  \notag
\end{eqnarray}
Similarly, 
\begin{eqnarray}
KS(d\zeta _{2}\otimes d\zeta _{3}) &=&\left\langle -\delta ^{-1}(\beta
_{3}+\omega _{\mathcal{K}}\beta _{3}^{\prime }),d\zeta _{3}\right\rangle
_{\lambda }\cdot dz \\
&=&(2\pi i)^{-1}dz.  \notag
\end{eqnarray}
We summarize.

\begin{proposition}
Let $z,u,\zeta _{1},\zeta _{2},\zeta _{3}$ be the standard coordinates in a
neighborhood of the cusp $c_{\infty }.$ Then, complex analytically, the
Kodaira-Spencer isomorphism 
\begin{equation}
KS(\Sigma ):\mathcal{P}\otimes \mathcal{L}\simeq \Omega _{\frak{X}}^{1}
\end{equation}
is given by the formulae 
\begin{equation}
KS(d\zeta _{1}\otimes d\zeta _{3})=-\delta (2\pi i)^{-1}du,\,\,\,KS(d\zeta
_{2}\otimes d\zeta _{3})=(2\pi i)^{-1}dz.  \label{KS formula}
\end{equation}
\end{proposition}

\begin{corollary}
\label{KS at cusps}The Kodaira-Spencer isomorphism $\mathcal{P}\otimes 
\mathcal{L}\simeq \Omega _{S}^{1}$ extends meromorphically over $\bar{S}.$
Moreover, in a formal neighborhood $\widehat{S}$ of $C,$ its restriction to
the line sub-bundle $\mathcal{P}_{0}\otimes \mathcal{L}$ is holomorphic, and
on any direct complement of $\mathcal{P}_{0}\otimes \mathcal{L}$ in $%
\mathcal{P}\otimes \mathcal{L}$ it has a simple pole along $C.$
\end{corollary}

\begin{proof}
As we have seen, $d\zeta _{1}\otimes d\zeta _{3}$ and $d\zeta _{2}\otimes
d\zeta _{3}$ define a basis of $\mathcal{P}\otimes \mathcal{L}$ at the
boundary, with $d\zeta _{1}\otimes d\zeta _{3}$ spanning the line sub-bundle 
$\mathcal{P}_{0}\otimes \mathcal{L}.$ On the other hand $du$ is holomorphic
there, while $dz$ has a simple pole along the boundary.
\end{proof}

\begin{corollary}
\label{complex kappa}The induced map 
\begin{equation}
\psi :\Omega _{\frak{X}}^{1}\rightarrow \mathcal{P}_{\mu }\otimes \mathcal{L}
\end{equation}
($\mathcal{P}_{\mu }=\mathcal{P}/\mathcal{P}_{0}$) obtained by inverting the
isomorphism $KS(\Sigma )$ and dividing $\mathcal{P}$ by $\mathcal{P}_{0}$
kills $du$ and maps $dz$ to $2\pi i\cdot d\zeta _{2}\otimes d\zeta _{3}.$
\end{corollary}

\begin{proof}
As we have seen, $d\zeta _{1}$ is a basis for $\mathcal{P}_{0}.$
\end{proof}

\begin{corollary}
\label{L^3}The isomorphism $\mathcal{L}^{3}\simeq \Omega _{S}^{2}$ maps $%
d\zeta _{3}^{\otimes 3}$ to a constant multiple of $dz\wedge du.$
\end{corollary}

\begin{proof}
The isomorphism $\det \mathcal{P}\simeq \mathcal{L}$ carries $d\zeta
_{1}\wedge d\zeta _{2}$ to a constant multiple of $d\zeta _{3},$ so the
corollary follows from (\ref{KS formula}).
\end{proof}

\subsubsection{Transferring the results to the algebraic category}

The computations leading to (\ref{KS formula}) of course descend (still in
the analytic category) to $S_{\Bbb{C}},$ because they are local in nature.
They then hold \emph{a fortiori} in the formal completion $\widehat{S}_{\Bbb{%
C}}$ along the cuspidal component $E$. Recall that the sections $d\zeta
_{1},d\zeta _{3}$ and $d\zeta _{2}\mod \left\langle d\zeta
_{1}\right\rangle $ (respectively $du$ and $dz\mod \left\langle
du\right\rangle $) are well-defined in $\widehat{S}_{\Bbb{C}},$ because as
global sections defined over $\frak{X}$ they are invariant under $\Gamma
_{cusp}$ (see Lemma \ref{section invariance}). But the Gauss-Manin and
Kodaira-Spencer maps are defined algebraically on $S,$ and both $\Omega _{%
\widehat{S}}^{1}$ and $\omega _{\mathcal{A}/\widehat{S}}$ are flat over $%
R_{0},$ so from the validity of the formulae over $\Bbb{C}$ we deduce their
validity in $\widehat{S}$ over $R_{0},$ provided we identify the
differential forms figuring in them (suitably normalized) with elements of $%
\Omega _{\widehat{S}}^{1}$ and $\omega _{\mathcal{A}/\widehat{S}}$ defined
over $R_{0}.$ In particular, they hold in the characteristic $p$ fiber as
well.

From the relation 
\begin{equation}
\frac{dq}{q}=\frac{2\pi i}{M}dz
\end{equation}
we deduce that the map $\psi $ has a simple zero along the cuspidal divisor.

Finally, although we have done all the computations at one specific cusp, it
is clear that similar computations hold at any other cusp.

\subsection{Fields of rationality\label{Fields of rationality}}

\subsubsection{Rationality of local sections of $\mathcal{P}$ and $\mathcal{L%
}$}

We have compared the arithmetic surface $S$ with the complex analytic
surfaces $\Gamma _{j}\backslash \frak{X}$ ($1\le j\le m$), and the
compactifications of these two models. We have also compared the universal
semi-abelian scheme $\mathcal{A}$ and the automorphic vector bundles $%
\mathcal{P}$ and $\mathcal{L}$ in both models. In this section we want to
compare the local parameters obtained from the two presentations, and settle
the question of rationality. For simplicity, we shall work rationally and
not integrally, which is all we need. In order to work integrally one would
have to study degeneration and periods of abelian varieties integrally,
which is more delicate, see [La1], Ch.I, Sections 3,4.

We shall need to look at local parameters at the cusps, and as the cusps are
defined only over $\mathcal{K}_{N},$ we shall work with $S_{\mathcal{K}_{N}}$
instead of $S_{\mathcal{K}}.$ With a little more care, working with Galois
orbits of cusps, we could probably prove rationality over $\mathcal{K},$ but
for our purpose $\mathcal{K}_{N}$ is good enough.

If $\xi $ and $\eta $ belong to a $\mathcal{K}_{N}$-module, we write $\xi
\sim \eta $ to mean that $\eta =c\xi $ for some $c\in \mathcal{K}%
_{N}^{\times }.$

We begin with the vector bundles $\mathcal{P}$ and $\mathcal{L}.$ Over $\Bbb{%
C}$ they yield analytic vector bundles $\mathcal{P}_{an}$ and $\mathcal{L}%
_{an}$ on each $X_{\Gamma _{j}}$ ($1\le j\le m$). Assume for the rest of
this section that $j=1$ and write $\Gamma =\Gamma _{1}.$ Similar results
will hold for every $j.$ The vector bundles $\mathcal{P}$ and $\mathcal{L}$
are trivialized over the unit ball $\frak{X}$ by means of the nowhere
vanishing sections $d\zeta _{3}\in H^{0}(\frak{X},\mathcal{L}_{an})$ and $%
d\zeta _{1},d\zeta _{2}\in H^{0}(\frak{X},\mathcal{P}_{an}).$ These sections
do not descend to $X_{\Gamma },$ but 
\begin{equation}
\sigma _{an}=(d\zeta _{1}\wedge d\zeta _{2})\otimes d\zeta _{3}^{-1}\in
H^{0}(X_{\Gamma },\det \mathcal{P}\otimes \mathcal{L}^{-1})
\end{equation}
does, as the factors of automorphy of $d\zeta _{1}\wedge d\zeta _{2}$ and $%
d\zeta _{3}$ are the same (\emph{cf }Section \ref{automorphy factors}).
Furthermore, this factor of automorphy (i.e. $j(\gamma ;z,u)$) is trivial on 
$\Gamma _{cusp},$ the stabilizer of $c_{\infty }$ in $\Gamma ,$ so $d\zeta
_{1}\wedge d\zeta _{2}$ and $d\zeta _{3}$ define sections of $\det \mathcal{P%
}$ and $\mathcal{L}$ on $\widehat{S}_{\Bbb{C}},$ the formal completion of $%
\bar{S}_{\Bbb{C}}$ along the cuspidal divisor $E_{c}=p^{-1}(c_{\infty
})\subset \bar{S}_{\Bbb{C}}.$ The same also holds for $d\zeta _{1}$ and $%
d\zeta _{2}\mod \left\langle d\zeta _{1}\right\rangle $ individually
(Lemma \ref{section invariance}). Along $E_{c},\mathcal{P}$ has a canonical
filtration 
\begin{equation}
0\rightarrow \mathcal{P}_{0}\rightarrow \mathcal{P}\rightarrow \mathcal{P}%
_{\mu }\rightarrow 0
\end{equation}
and $d\zeta _{1}$ is a generator of $\mathcal{P}_{0}.$ (Compare (\ref
{semi-abelian extension}) and (\ref{complex extension}) and note that the
projection to $\Bbb{C}/\mathcal{O}_{\mathcal{K}}=B(\Bbb{C})$ is via the
coordinate $\zeta _{1},$ so $d\zeta _{1}$ is a generator of $\mathcal{P}%
_{0}|_{E_{c}}=\omega _{B}$.) As we have shown in Section \ref{retraction},
this filtration extends to the formal neighborhood $\widehat{S}_{\Bbb{C}}$
of $E_{c}.$ The vector bundles $\mathcal{P}$ and $\mathcal{L},$ as well as
the filtration on $\mathcal{P},$ are defined over $\mathcal{K}_{N}$ . It
makes sense therefore to ask if certain sections are $\mathcal{K}_{N}$%
-rational. Recall that the cusp $c_{\infty }$ is of type $(\mathcal{O}_{%
\mathcal{K}},\mathcal{O}_{\mathcal{K}}).$

\begin{proposition}
(i) $2\pi i\cdot d\zeta _{3}\in H^{0}(\widehat{S}_{\mathcal{K}_{N}},\mathcal{%
L})$. In other words, this section is $\mathcal{K}_{N}$-rational.

(ii) Similarly $2\pi i\cdot d\zeta _{2}$ projects (modulo $\mathcal{P}_{0}$)
to a $\mathcal{K}_{N}$-rational section of $\mathcal{P}_{\mu }.$

(iii) Let $B$ be the elliptic curve over $\mathcal{K}_{N}$ associated with
the cusp $c_{\infty }$ as in Section \ref{universal semi abelian}. Let $%
\Omega _{B}\in \Bbb{C}^{\times }$ be a period of a basis $\omega $ of $%
\omega _{B}=H^{0}(B,\Omega _{B/\mathcal{K}_{N}}^{1})$ (i.e. the lattice of
periods of $\omega $ is $\Omega _{B}\cdot \mathcal{O}_{\mathcal{K}}$). This $%
\Omega _{B}$ is well-defined up to an element of $\mathcal{K}_{N}^{\times }.$
Then $\Omega _{B}\cdot d\zeta _{1}\in H^{0}(\widehat{S}_{\mathcal{K}_{N}},%
\mathcal{P}_{0})$ is $\mathcal{K}_{N}$-rational.
\end{proposition}

\begin{proof}
Let $E$ be the component of $C_{\mathcal{K}_{N}}$ which over $\Bbb{C}$
becomes $E_{c}.$ Let $\mathcal{G}$ be the universal semi-abelian scheme over 
$E.$ Then $\mathcal{G}$ is a semi-abelian scheme which is an extension of $%
B\times _{\mathcal{K}_{N}}E$ by the torus $(\mathcal{O}_{\mathcal{K}}\otimes 
\Bbb{G}_{m,\mathcal{K}_{N}})\times _{\mathcal{K}_{N}}E.$ At any point $u\in
E(\Bbb{C})$ we have the analytic model $\mathcal{G}_{u}$ (\ref{semi-abelian}%
) for the fiber of $\mathcal{G}$ at $u,$ but the abelian part and the toric
part are constant. Over $E$ the line bundle $\mathcal{P}_{0}$ is (by
definition) $\omega _{B\times E/E}.$ As the lattice of periods of a suitable 
$\mathcal{K}_{N}$-rational differential is $\Omega _{B}\cdot \mathcal{O}_{%
\mathcal{K}},$ while the lattice of periods of $d\zeta _{1}$ is $\mathcal{O}%
_{\mathcal{K}},$ part (iii) follows. For parts (i) and (ii) observe that the
toric part of $\mathcal{G}$ is in fact defined over $\mathcal{K}$, and that $%
e_{\mathcal{O}_{\mathcal{K}}}^{*}$ maps the cotangent space of $\mathcal{O}_{%
\mathcal{K}}\otimes \Bbb{G}_{m,\mathcal{K}}$ isomorphically to the $\mathcal{%
K}$-span of $2\pi id\zeta _{2}$ and $2\pi id\zeta _{3}$.
\end{proof}

\begin{corollary}
\label{rationality}$\Omega _{B}\cdot \sigma _{an}$ is a nowhere vanishing
global section of $\det \mathcal{P}\otimes \mathcal{L}^{-1}$ over $S_{\Gamma
},$ rational over $\mathcal{K}_{N}.$
\end{corollary}

\begin{proof}
Recall that we denote by $S_{\Gamma }$ the connected component of $S_{%
\mathcal{K}_{N}}$ whose associated analytic space is the complex manifold $%
X_{\Gamma }.$ We have seen that as an analytic section $\Omega _{B}\cdot
\sigma _{an}$ descends to $X_{\Gamma }$ and extends to the smooth
compactification $\bar{X}_{\Gamma }.$ By GAGA, it is algebraic. Since $\bar{X%
}_{\Gamma }$ is connected, to check its field of definition, it is enough to
consider it at one of the cusps. By the Proposition, its restriction to the
formal neighborhood of $E_{c}$ ($c=c_{\infty }$) is defined over $\mathcal{K}%
_{N}.$
\end{proof}

The complex periods $\Omega _{B}$ (and their powers) appear as the
transcendental parts of special values of $L$-functions associated with
Grossencharacters of $\mathcal{K}$. They are therefore instrumental in the
construction of $p$-adic $L$-functions on $\mathcal{K}.$ We expect them to
appear in the $p$-adic interpolation of holomorphic Eisenstein series on the
group $\mathbf{G},$ much as powers of $2\pi i$ (values of $\zeta (2k)$)
appear in the $p$-adic interpolation of Eisenstein series on $GL_{2}(\Bbb{Q}%
).$

\subsubsection{Rationality of local parameters at the cusps}

We keep the assumptions and the notation of the previous section.
Analytically, neighborhoods of $E_{c_{\infty }}$ were described in Section 
\ref{complex smooth compactification} with the aid of the parameters $(z,u).$
Let $\widehat{S}$ denote the formal completion of $\bar{S}_{\mathcal{K}_{N}}$
along $E.$ Let $r:\widehat{S}\rightarrow E$ be the infinitesimal retraction
discussed in Section \ref{retraction}. If $i:E\hookrightarrow \widehat{S}$
is the closed embedding then $r\circ i=Id_{E}.$ If $\mathcal{I}$ is the
sheaf of definition of $E,$ then $\mathcal{N}=\mathcal{I}/\mathcal{I}^{2}$
is the conormal bundle to $E,$ hence its analytification is the dual of the
line bundle $\mathcal{T},$

\begin{equation}
\mathcal{N}_{an}=\mathcal{T}^{\vee }.
\end{equation}

Consider $r^{*}\mathcal{N}$ on $\widehat{S}.$ The retraction allows us to
split the exact sequence 
\begin{equation}
0\rightarrow \mathcal{N}\rightarrow i^{*}\Omega _{\widehat{S}%
}^{1}\rightarrow \Omega _{E}^{1}\rightarrow 0
\end{equation}
using $\Omega _{E}^{1}=i^{*}r^{*}\Omega _{E}^{1}\subset i^{*}\Omega _{%
\widehat{S}}^{1}$. Thus $i^{*}\Omega _{\widehat{S}}^{1}=\mathcal{N}\times
\Omega _{E}^{1}.$ The map $i\circ r:\widehat{S}\rightarrow \widehat{S}$
induces a sheaf homomorphism $r^{*}i^{*}\Omega _{\widehat{S}}^{1}\rightarrow
\Omega _{\widehat{S}}^{1},$ which becomes the identity if we restrict it to $%
E$ (i.e. follow it with $i^{*}$). By Nakayama's lemma, it is an isomorphism.
It follows that 
\begin{equation}
\Omega _{\widehat{S}}^{1}=r^{*}i^{*}\Omega _{\widehat{S}}^{1}=r^{*}\mathcal{N%
}\times r^{*}\Omega _{E}^{1}.
\end{equation}

Let $x\in E$ and represent it by $u\in \Bbb{C}$ (modulo $\Lambda $). Then $%
q=e^{2\pi iz/M}$, where $M$ is the width of the cusp (\ref{width}), is a
local analytic parameter on a classical neighborhood $U_{x}$ of $x$ which
vanishes to first order along $E$. Note that $q$ depends on the choice of $u$
(see Remark below). It follows that $dq,$ the image of $q$ in $\mathcal{I}%
_{an}/\mathcal{I}_{an}^{2},$ is a basis of $\mathcal{N}_{an}$ (on $U_{x}\cap
E$). But 
\begin{equation}
2\pi i\cdot dz=M\frac{dq}{q}
\end{equation}
($\mod \left\langle du\right\rangle $) is independent of $u$ (see (\ref
{unipotent action})), so represents a global meromorphic section of $r^{*}%
\mathcal{N}_{an},$ with a simple pole along $E\subset \widehat{S}_{\Bbb{C}}.$
By GAGA, this section is (meromorphic) algebraic.

\begin{proposition}
(i) The section $2\pi i\cdot dz$ $\mod \left\langle du\right\rangle $
is $\mathcal{K}_{N}$-rational, i.e. it is the analytification of a section
of $r^{*}\mathcal{N}.$ (ii) The section $\Omega _{B}\cdot du$ is $\mathcal{K}%
_{N}$-rational, i.e. belongs to $H^{0}(E,\Omega _{E/\mathcal{K}_{N}}^{1}).$
\end{proposition}

\begin{proof}
The proof relies on the Kodaira-Spencer isomorphism $KS(\Sigma )$ (\ref{KS
formula}), which is a $\mathcal{K}_{N}$-rational (even $\mathcal{K}$%
-rational) algebraic isomorphism between $\mathcal{P}\otimes \mathcal{L}$
and $\Omega _{S}^{1}.$ As we have shown, it extends to a meromorphic
homomorphism from $\mathcal{P}\otimes \mathcal{L}$ to $\Omega _{\bar{S}}^{1}$
over $\bar{S}.$ Over $\widehat{S}$ it induces an isomorphism of $\mathcal{P}%
_{0}\otimes \mathcal{L}$ onto $r^{*}\Omega _{E}^{1}\subset \Omega _{\widehat{%
S}}^{1}$ carrying the $\mathcal{K}_{N}$-rational section $\Omega _{B}d\zeta
_{1}\otimes 2\pi id\zeta _{3}$ to $-\Omega _{B}\delta \cdot du$, proving
part (ii) of the proposition. It also carries $2\pi id\zeta _{2}\otimes 2\pi
id\zeta _{3}$ to $2\pi idz,$ but the latter is only meromorphic. We may
summarize the situation over $\widehat{S}$ by the following commutative
diagram with exact rows: 
\begin{equation}
\begin{array}{lllllllll}
0 & \rightarrow & \widehat{\mathcal{I}}\otimes \mathcal{P}_{0}\otimes 
\mathcal{L} & \rightarrow & \widehat{\mathcal{I}}\otimes \mathcal{P}\otimes 
\mathcal{L} & \rightarrow & \widehat{\mathcal{I}}\otimes \mathcal{P}_{\mu
}\otimes \mathcal{L} & \rightarrow & 0 \\ 
&  & \downarrow &  & \downarrow KS(\Sigma ) &  & \downarrow &  &  \\ 
0 & \rightarrow & r^{*}\Omega _{E}^{1} & \rightarrow & \Omega _{\widehat{S}%
}^{1} & \rightarrow & r^{*}\mathcal{N} & \rightarrow & 0
\end{array}
.
\end{equation}
Let $h$ be a $\mathcal{K}_{N}$-rational local equation of $E,$ i.e. a $%
\mathcal{K}_{N}$-rational section of $\mathcal{I}$ in some Zariski open $U\ $%
intersecting $E$ non-trivially, vanishing to first order along $E\cap U.$
The differential $\eta =h\cdot (2\pi idz)$ is regular on $U,$ and to prove
that it is $\mathcal{K}_{N}$-rational we may restrict it to $\widehat{S}$
and check rationality there. But in $\widehat{S}$ we have a $\mathcal{K}_{N}$%
-rational product decomposition $\Omega _{\widehat{S}}^{1}=r^{*}\mathcal{N}%
\times r^{*}\Omega _{E}^{1}$ and the projection of $\eta $ to the second
factor is 0, so it is enough to prove rationality of its projection to $r^{*}%
\mathcal{N}.$ This projection is the image, under $KS(\Sigma ),$ of $h\cdot
(2\pi id\zeta _{2}\otimes 2\pi id\zeta _{3}\mod \mathcal{P}_{0}\otimes 
\mathcal{L}),$ so our assertion follows from parts (i) and (ii) of the
previous proposition. This proves that $\eta ,$ hence $h^{-1}\eta =2\pi idz$
is a $\mathcal{K}_{N}$-rational differential. An alternative proof of part
(ii) is to note that $E$ is isogenous over $\mathcal{K}_{N}$ to $B,$ so up
to a $\mathcal{K}_{N}$-multiple has the same period.
\end{proof}

\begin{remark}
The parameter $q$ is not a well-defined parameter at $x$, and depends not
only on $x,$ but also on the point $u$ used to uniformize it. If we change $%
u $ to $u+s$ ($s\in \Lambda $) then $q$ is multiplied by the factor $e^{2\pi
i\delta \bar{s}(u+s/2)/M},$ so although $\mathcal{O}_{\bar{S}_{\Bbb{C}%
},x}^{hol}\subset \widehat{\mathcal{O}}_{\bar{S}_{\Bbb{C}},x}$ and analytic
parameters may be considered as formal parameters, the question whether $q$
itself is $\mathcal{K}_{N}$-rational is not well-defined (in sharp contrast
to the case of modular curves!).
\end{remark}

\subsubsection{Normalizing the isomorphism $\det \mathcal{P}\simeq \mathcal{L%
}$}

Let us fix a nowhere vanishing section 
\begin{equation}
\sigma \in H^{0}(S_{\mathcal{K}},\det \mathcal{P}\otimes \mathcal{L}^{-1}).
\end{equation}
This section is determined up to $\mathcal{K}^{\times }.$ From now on we
shall use this section to identify $\det \mathcal{P}$ with $\mathcal{L}$
whenever such an identification is needed. From Corollary \ref{rationality}
we deduce that when we base change to $\Bbb{C},$ on each connected component 
$X_{\Gamma }$%
\begin{equation}
\sigma \sim \Omega _{B}\cdot \sigma _{an}.
\end{equation}

\section{Picard modular schemes modulo an inert prime}

\subsection{The stratification}

\subsubsection{The three strata}

Let $p$ be a rational prime which is inert in $\mathcal{K}$ and relatively
prime to $2N.$ Then $\kappa _{0}=R_{0}/pR_{0}$ is isomorphic to $\Bbb{F}%
_{p^{2}}.$ We fix an algebraic closure $\kappa $ of $\kappa _{0}$ and
consider the characteristic $p$ fiber 
\begin{equation}
\bar{S}_{\kappa }=\bar{S}\times _{R_{0}}\kappa .
\end{equation}
\emph{Unless otherwise specified}, in this section we let $S$ and $\bar{S}$
denote the characterstic $p$ fibers $S_{\kappa }$ and $\bar{S}_{\kappa }.$
We also use the abbreviation $\omega _{\mathcal{A}}$ for $\omega _{\mathcal{A%
}/\bar{S}}$ etc.

Recall that an abelian variety over an algebraically closed field of
characteristic $p$ is called \emph{supersingular} if the Newton polygon of
its $p$-divisible group has a constant slope $1/2.$ It is called \emph{%
superspecial }if it is isomorphic to a product of supersingular elliptic
curves. The following theorem combines various results proved in [Bu-We],
[V] and [We]. See also [dS-G], Theorem 2.1.

\begin{theorem}
\label{Vollaard}(i) There exists a closed reduced $1$-dimensional subscheme $%
S_{ss}\subset \bar{S}$, disjoint from the cuspidal divisor (i.e. contained
in $S$), which is uniquely characterized by the fact that for any geometric
point $x$ of $S,$ the abelian variety $\mathcal{A}_{x}$ is supersingular if
and only if $x$ lies on $S_{ss}$. The scheme $S_{ss}$ is defined over $%
\kappa _{0}.$

(ii) Let $S_{ssp}$ be the singular locus of $S_{ss}.$ Then $x$ lies in $%
S_{ssp}$ if and only if $\mathcal{A}_{x}$ is superspecial. If $x\in S_{ssp}$
then 
\begin{equation}
\widehat{\mathcal{O}}_{S_{ss},x}\simeq \kappa%
[[u,v]]/(u^{p+1}+v^{p+1}).
\end{equation}

(iii) Assume that $N$ is large enough (depending on $p$). Then the
irreducible components of $S_{ss}$ are nonsingular, and in fact are all
isomorphic to the Fermat curve $\mathcal{C}_{p}$ given by the equation 
\begin{equation}
x^{p+1}+y^{p+1}+z^{p+1}=0.
\end{equation}
There are $p^{3}+1$ points of $S_{ssp}$ on each irreducible component, and
through each such point pass $p+1$ irreducible components. Any two
irreducible components are either disjoint or intersect transversally at a
unique point.

(iv) Without the assumption of $N$ being large (but under $N\ge 3$ as usual)
the irreducible components of $S_{ss}$ may have multiple intersections with
each other, including self-intersections. Their normalizations are
nevertheless still isomorphic to $\mathcal{C}_{p}.$
\end{theorem}

We call $\bar{S}_{\mu }=\bar{S}\backslash S_{ss}$ (or $S_{\mu }=\bar{S}_{\mu
}\cap S$) the $\mu $\emph{-ordinary} or \emph{generic} locus, $%
S_{gss}=S_{ss}\backslash S_{ssp}$ the \emph{general supersingular locus},
and $S_{ssp}$ the \emph{superspecial} \emph{locus}. Then $\bar{S}=\bar{S}%
_{\mu }\cup S_{gss}\cup S_{ssp}$ is a stratification.

\subsubsection{The $p$-divisible group\label{p-div}}

Let $x:Spec(k)\rightarrow S$ ($k$ an algebraically closed field) be a
geometric point of $S,$ $\mathcal{A}_{x}$ the corresponding fiber of $%
\mathcal{A}$, and $\mathcal{A}_{x}(p)$ its $p$-divisible group. Let $\frak{G}
$ be the $p$-divisible group of a supersingular elliptic curve over $k$ (the
group denoted by $G_{1,1}$ in the Manin-Dieudonn\'{e} classification). The
following theorem can be deduced from [Bu-We] and [V].

\begin{theorem}
(i) If $x\in S_{\mu }$ then 
\begin{equation}
\mathcal{A}_{x}(p)\simeq (\mathcal{O}_{\mathcal{K}}\otimes \mu _{p^{\infty
}})\times \frak{G}\times (\mathcal{O}_{\mathcal{K}}\otimes \Bbb{Q}_{p}/\Bbb{Z%
}_{p}).
\end{equation}

(ii) If $x\in S_{ss}$ then $\mathcal{A}_{x}(p)$ is isogenous to $\frak{G}%
^{3},$ and $x\in S_{ssp}$ if and only if the two groups are isomorphic.
\end{theorem}

While the $p$-divisible group of a $\mu $-ordinary geometric fiber actually
splits as a product of its multiplicative, local-local and \'{e}tale parts,
over the whole of $S_{\mu }$ we only get a filtration 
\begin{equation}
0\subset Fil^{2}\mathcal{A}(p)\subset Fil^{1}\mathcal{A}(p)\subset Fil^{0}%
\mathcal{A}(p)=\mathcal{A}(p)
\end{equation}
by $\mathcal{O}_{\mathcal{K}}$-stable $p$-divisible groups. Here $%
gr^{2}=Fil^{2}$ is of multiplicative type, $gr^{1}=Fil^{1}/Fil^{2}$ is a
local-local group and $gr^{0}=Fil^{0}/Fil^{1}$ is \'{e}tale, each of height $%
2$ ($\mathcal{O}_{\mathcal{K}}$-height 1).

\subsection{New relations between $\mathcal{P}$ and $\mathcal{L}$ in
characteristic $p$}

For proofs and more details on this sub-section see [dS-G], Section 2.2.

\subsubsection{The line bundles $\mathcal{P}_{0}$ and $\mathcal{P}_{\mu }$
over $\bar{S}_{\mu }$}

Consider the universal semi-abelian variety $\mathcal{A}$ over the Zariski
open set $\bar{S}_{\mu }.$ Over the cuspidal divisor $C=\bar{S}\backslash S,$
$\mathcal{P}=\omega _{\mathcal{A}}(\Sigma )$ admits a canonical filtration 
\begin{equation}
0\rightarrow \mathcal{P}_{0}\rightarrow \mathcal{P}\rightarrow \mathcal{P}%
_{\mu }\rightarrow 0  \label{Fil on P}
\end{equation}
where $\mathcal{P}_{0}$ is the cotangent space to the abelian part of $%
\mathcal{A},$ and $\mathcal{P}_{\mu }$ is the $\Sigma $-component of the
cotangent space to the toric part of $\mathcal{A}.$ This filtration exists
already in characteristic 0, but when we reduce the Picard surface modulo $p$
it extends to the whole of $\bar{S}_{\mu }.$ Over the non-cuspidal part $%
S_{\mu }$ we may set 
\begin{equation}
\mathcal{P}_{0}=\ker \left( \omega _{\mathcal{A}[p]^{0}}\rightarrow \omega _{%
\mathcal{A}[p]^{\mu }}\right)
\end{equation}
where $\mathcal{A}[p]^{\mu }$ is the $p$-torsion in $\mathcal{A}(p)^{\mu
}=Fil^{2}\mathcal{A}(p).$ Then $\mathcal{P}_{\mu }$ is identified with $%
\omega _{\mathcal{A}[p]^{\mu }}(\Sigma ).$

\subsubsection{Frobenius and Verschiebung}

Let $\mathcal{A}^{(p)}=\mathcal{A}\times _{\bar{S},\Phi }\bar{S}$ be the
base change of $\mathcal{A}$ with respect to the absolute Frobenius morphism 
$\Phi $ of degree $p$ of $\bar{S}.$ The \emph{relative Frobenius} is an $%
\mathcal{O}_{\bar{S}}$-linear isogeny $Frob_{\mathcal{A}}:\mathcal{A}%
\rightarrow \mathcal{A}^{(p)},$ characterized by the fact that $pr_{1}\circ
Frob_{\mathcal{A}}$ is the absolute Frobenius morphism of $\mathcal{A}.$
Over $S$ (but not over the boundary $C$) we have the dual abelian scheme $%
\mathcal{A}^{t}$, and the \emph{Verschiebung }$Ver_{\mathcal{A}}:\mathcal{A}%
^{(p)}\rightarrow \mathcal{A}$ is the $\mathcal{O}_{S}$-linear isogeny which
is dual to $Frob_{\mathcal{A}^{t}}:\mathcal{A}^{t}\rightarrow (\mathcal{A}%
^{t})^{(p)}.$

We clearly have $\omega _{\mathcal{A}^{(p)}}=\omega _{\mathcal{A}}^{(p)}$,
and we let 
\begin{equation}
F:\omega _{\mathcal{A}}^{(p)}\rightarrow \omega _{\mathcal{A}%
},\,\,\,V:\omega _{\mathcal{A}}\rightarrow \omega _{\mathcal{A}}^{(p)}
\end{equation}
be the $\mathcal{O}_{\bar{S}}$-linear maps of vector bundles induced by the
isogenies $Frob_{\mathcal{A}}$ and $Ver_{\mathcal{A}}$ on the cotangent
spaces. We refer to [dS-G] for a discussion how to define $V$ over the whole
of $\bar{S},$ despite the fact that $Ver_{\mathcal{A}}$ is only defined over 
$S$.

Taking $\Sigma $-components we get the map 
\begin{equation}
V_{\mathcal{P}}:\mathcal{P}=\omega _{\mathcal{A}}(\Sigma )\rightarrow \omega
_{\mathcal{A}}^{(p)}(\Sigma )=\omega _{\mathcal{A}}(\bar{\Sigma})^{(p)}=%
\mathcal{L}^{(p)},
\end{equation}
and taking the $\bar{\Sigma}$-component we similarly get 
\begin{equation}
V_{\mathcal{L}}:\mathcal{L}\rightarrow \mathcal{P}^{(p)}.
\end{equation}

\begin{proposition}
Over $\bar{S}_{\mu }$ both $V_{\mathcal{P}}$ and $V_{\mathcal{L}}$ are of
rank 1, 
\begin{equation}
\mathcal{P}_{0}=\ker V_{\mathcal{P}}
\end{equation}
and the image of $V_{\mathcal{L}}$ is a direct sum complement to $\mathcal{P}%
_{0}^{(p)}:$%
\begin{equation}
\mathcal{P}^{(p)}=\mathcal{P}_{0}^{(p)}\oplus V(\mathcal{L}).
\end{equation}
\end{proposition}

Recall that over any base scheme in characteristic $p,$ and for any \emph{%
line} bundle $\mathcal{M},$ its base change $\mathcal{M}^{(p)}$ under the
absolute Frobenius, is canonically isomorphic to its $p$th power $\mathcal{M}%
^{p}.$

\begin{corollary}
Over $\bar{S}_{\mu }$, $\mathcal{P}_{\mu }\simeq \mathcal{L}^{p},\,\,%
\mathcal{P}_{0}\simeq \mathcal{L}^{1-p},\,\,$and $\mathcal{L}^{p^{2}}\simeq 
\mathcal{L}$. For $k\ge 1$ odd, $\mathcal{P}^{(p^{k})}\simeq \mathcal{L}%
^{p-1}\oplus \mathcal{L}.$ For $k\ge 2$ even, $\mathcal{P}^{(p^{k})}\simeq 
\mathcal{L}^{1-p}\oplus \mathcal{L}^{p},$ but for $k=0$ we only have an
exact sequence 
\begin{equation}
0\rightarrow \mathcal{L}^{1-p}\rightarrow \mathcal{P}\rightarrow \mathcal{L}%
^{p}\rightarrow 0.
\end{equation}
\end{corollary}

\begin{corollary}
Over $\bar{S}_{\mu },$ $\mathcal{L}^{p^{2}-1},\mathcal{P}_{\mu }^{p^{2}-1}$
and $\mathcal{P}_{0}^{p+1}$ are trivial line bundles.
\end{corollary}

\subsubsection{Extending the filtration on $\mathcal{P}$ over $S_{gss}$}

In order to determine to what extent the filtration on $\mathcal{P}$ and the
relation between $\mathcal{L}$ and the two graded pieces of the filtration
extend into the supersingular locus, we have to employ Dieudonn\'{e} theory.
The following is proved in [dS-G].

\begin{proposition}
(i) Let $\mathcal{P}_{0}=\ker V_{\mathcal{P}}$ (this agrees with what was
denoted by $\mathcal{P}_{0}$ over $\bar{S}_{\mu }$). Then outside $S_{ssp},$ 
$V(\mathcal{P})=\mathcal{L}^{(p)}$ and $\mathcal{P}_{0}$ is a rank 1
submodule.

(ii) Let $\mathcal{P}_{\mu }=\mathcal{P}/\mathcal{P}_{0}.$ Then outside $%
S_{ssp}$ we have $\mathcal{P}_{\mu }\simeq \mathcal{L}^{p}$ and $\mathcal{P}%
_{0}\simeq \mathcal{L}^{1-p}$.
\end{proposition}

For $V_{\mathcal{L}}$ we similarly get the following.

\begin{proposition}
Outside $S_{ssp}$, $V_{\mathcal{L}}$ maps $\mathcal{L}$ injectively onto a
sub-line-bundle of $\mathcal{P}^{(p)}.$
\end{proposition}

At a superspecial point, both $V_{\mathcal{P}}$ and $V_{\mathcal{L}}$ vanish.

\subsubsection{The Hasse invariant}

As we have just seen, the fact that $V_{\mathcal{P}}$ and $V_{\mathcal{L}}$
are both of rank $1$ ``extends'' across the general supersingular locus $%
S_{gss}.$ However, while $\text{Im}(V_{\mathcal{L}})$ and $\ker (V_{\mathcal{%
P}}^{(p)})=\mathcal{P}_{0}^{(p)}$ made up a frame of $\mathcal{P}$ over $%
\bar{S}_{\mu },$ over $S_{gss}$ these two line bundles coincide. To state a
more precise result, we make the following definition.

\begin{definition}
The \emph{Hasse invariant} is 
\begin{equation}
h_{\bar{\Sigma}}=V_{\mathcal{P}}^{(p)}\circ V_{\mathcal{L}}\in Hom(\mathcal{L%
},\mathcal{L}^{(p^{2})}).
\end{equation}
\end{definition}

As $\mathcal{L}^{(p^{2})}\simeq \mathcal{L}^{p^{2}},$ the Hasse invariant is
a global section of $\mathcal{L}^{p^{2}-1}$, i.e. a modular form of weight $%
p^{2}-1$ over $\kappa ,$%
\begin{equation}
h_{\bar{\Sigma}}\in M_{p^{2}-1}(N,\kappa ).
\end{equation}

It turns out that $h_{\bar{\Sigma}}$ has a \emph{simple} zero along the
supersingular locus $S_{ss}$. Once again, this requires a little computation
with Dieudonn\'{e} modules. Equivalently, we have the following theorem.

\begin{theorem}
\label{Hasse invariant}The divisor of $h_{\bar{\Sigma}}$ is $S_{ss}$ (with
its reduced subscheme structure).
\end{theorem}

\subsection{The open Igusa surfaces}

\subsubsection{The Igusa scheme}

Let $N\ge 3$ as always, and let $\mathcal{M}$ be the moduli problem of
Section \ref{Moduli}. Let $n\ge 1$ and consider the following moduli problem
on $\kappa _{0}$-algebras:

\begin{itemize}
\item  $\mathcal{M}_{Ig(p^{n})}(R)$ is the set of isomorphism classes of
pairs $(\underline{A},\varepsilon )$ where $\underline{A}\in \mathcal{M}(R)$
and 
\begin{equation}
\varepsilon :\delta _{\mathcal{K}}^{-1}\mathcal{O}_{\mathcal{K}}\otimes \mu
_{p^{n}}\hookrightarrow A[p^{n}]
\end{equation}
is a closed immersion of $\mathcal{O}_{\mathcal{K}}$-group schemes over $R.$
\end{itemize}

It is clear that if $(\underline{A},\varepsilon )\in \mathcal{M}%
_{Ig(p^{n})}(R)$ then $A$ is fiber-by-fiber $\mu $-ordinary and therefore $%
\underline{A}\in \mathcal{M}(R)$ defines an $R$-point of $S_{\mu }.$ The
image of $\varepsilon $ is then $A[p^{n}]^{\mu }$, the maximal
subgroup-scheme of $A[p^{n}]$ of multiplicative type. It is also clear that
the functor $R\rightsquigarrow \mathcal{M}_{Ig(p^{n})}(R)$ is relatively
representable over $\mathcal{M},$ and therefore as $N\ge 3$ and $\mathcal{M}$
is representable, this functor is also representable by a scheme $Ig_{\mu
}(p^{n})$ which maps to $S_{\mu }.$ See [Ka-Ma] for the notion of relative
representability. We call $Ig_{\mu }(p^{n})$ the \emph{Igusa scheme of level 
}$p^{n}.$

\begin{proposition}
The morphism $\tau :Ig_{\mu }(p^{n})\rightarrow S_{\mu }$ is finite and
\'{e}tale, with the Galois group $\Delta (p^{n})=(\mathcal{O}_{\mathcal{K}%
}/p^{n}\mathcal{O}_{\mathcal{K}})^{\times }$ acting as a group of deck
transformations.
\end{proposition}

\begin{proof}
Every $\mu $-ordinary abelian variety has a \emph{unique} finite flat $%
\mathcal{O}_{\mathcal{K}}$-subgroup scheme of multiplicative type $%
A[p^{n}]^{\mu }$ of rank $p^{2n}.$ Such a subgroup scheme is, locally in the
\'{e}tale topology, isomorphic to $\delta _{\mathcal{K}}^{-1}\mathcal{O}_{%
\mathcal{K}}\otimes \mu _{p^{n}},$ and any two isomorphisms differ by a
unique automorphism of $\delta _{\mathcal{K}}^{-1}\mathcal{O}_{\mathcal{K}%
}\otimes \mu _{p^{n}}.$ But $\Delta (p^{n})=Aut_{\mathcal{O}_{\mathcal{K}%
}}(\delta _{\mathcal{K}}^{-1}\mathcal{O}_{\mathcal{K}}\otimes \mu _{p^{n}}).$
If we let $\gamma \in \Delta (p^{n})$ act on the pair $(\underline{A}%
,\varepsilon )$ via 
\begin{equation}
\gamma ((\underline{A},\varepsilon ))=(\underline{A},\varepsilon \circ
\gamma ^{-1})  \label{gamma action}
\end{equation}
$\Delta (p^{n})$ becomes a group of deck transformation and the proof is
complete.
\end{proof}

\subsubsection{A compactification over the cusps}

The proof of the following proposition mimics the construction of $\bar{S}.$
We omit it.

\begin{proposition}
Let $\overline{Ig}_{\mu }(p^{n})$ be the normalization of $\bar{S}_{\mu }=%
\bar{S}\backslash S_{ss}$ in $Ig_{\mu }(p^{n}).$ Then $\overline{Ig}_{\mu
}(p^{n})\rightarrow \bar{S}_{\mu }$ is finite \'{e}tale and the action of $%
\Delta (p^{n})$ extends to it. The boundary $\overline{Ig}_{\mu
}(p^{n})\backslash Ig_{\mu }(p^{n})$ is non-canonically identified with $%
\Delta (p^{n})\times C.$
\end{proposition}

We define similarly $Ig_{\mu }^{*},$ and note that it is finite \'{e}tale
over $S_{\mu }^{*}.$

\begin{proposition}
Let $\mathcal{A}$ denote the pull-back of the universal semi-abelian variety
from $\bar{S}_{\mu }$ to $\overline{Ig}_{\mu }(p^{n}).$ Then $\mathcal{A}$
is equipped with a canonical Igusa level structure 
\begin{equation}
\varepsilon :\delta _{\mathcal{K}}^{-1}\mathcal{O}_{\mathcal{K}}\otimes \mu
_{p^{n}}\simeq \mathcal{A}[p^{n}]^{\mu }.
\end{equation}
Over $C$ and after base change to $R_{N}/pR_{N}$ the toric part of $\mathcal{%
A}$ is locally Zariski of the form $\frak{a}\otimes \Bbb{G}_{m}$ and $%
\varepsilon $ is then an $\mathcal{O}_{\mathcal{K}}$-linear isomorphism
between $\delta _{\mathcal{K}}^{-1}\mathcal{O}_{\mathcal{K}}\otimes \mu
_{p^{n}}$ and $\frak{a}\otimes \mu _{p^{n}}.$
\end{proposition}

\subsubsection{A trivialization of $\mathcal{L}$ over the Igusa surface}

From now on we focus on $\overline{Ig}_{\mu }=\overline{Ig}_{\mu }(p)$
although similar results hold when $n>1,$ and would be instrumental in the
study of $p$-adic modular forms. The vector bundle $\omega _{\mathcal{A}}$
pulls back to a similar vector bundle over $\overline{Ig}_{\mu }.$ But there 
\begin{equation}
\omega _{\mathcal{A}}^{\mu }:=\omega _{\mathcal{A}[p]^{\mu }}
\end{equation}
is a rank 2 quotient bundle stable under $\mathcal{O}_{\mathcal{K}}$ (of
type $(1,1)$), and the isomorphism $\varepsilon $ induces an isomorphism 
\begin{equation}
\varepsilon ^{*}:\omega _{\mathcal{A}}^{\mu }\simeq \omega _{\delta _{%
\mathcal{K}}^{-1}\mathcal{O}_{\mathcal{K}}\otimes \mu _{p}}.
\end{equation}
Now $Lie(\delta _{\mathcal{K}}^{-1}\mathcal{O}_{\mathcal{K}}\otimes \mu
_{p})=\delta _{\mathcal{K}}^{-1}\mathcal{O}_{\mathcal{K}}\otimes Lie(\mu
_{p})=\delta _{\mathcal{K}}^{-1}\mathcal{O}_{\mathcal{K}}\otimes Lie(\Bbb{G}%
_{m})$ and by duality 
\begin{equation}
\omega _{\delta _{\mathcal{K}}^{-1}\mathcal{O}_{\mathcal{K}}\otimes \mu
_{p}}=\mathcal{O}_{\mathcal{K}}\otimes \omega _{\Bbb{G}_{m}},
\end{equation}
with $1\otimes dT/T$ as a generator (if $T$ is the parameter of $\Bbb{G}_{m}$%
). Here we have used the fact that the $\Bbb{Z}$-dual of $\delta _{\mathcal{K%
}}^{-1}\mathcal{O}_{\mathcal{K}}$ is $\mathcal{O}_{\mathcal{K}}$ via the
trace pairing. This is the constant vector bundle $\mathcal{O}_{\mathcal{K}%
}\otimes R=R(\Sigma )\oplus R(\bar{\Sigma}).$

\begin{proposition}
The line bundles $\mathcal{L},$ $\mathcal{P}_{0}$ and $\mathcal{P}_{\mu }$
are trivial over $\overline{Ig}_{\mu }.$
\end{proposition}

\begin{proof}
Use $\varepsilon ^{*}$ as an isomorphism between vector bundles and note
that $\mathcal{L}=\omega _{\mathcal{A}}^{\mu }(\bar{\Sigma})$ and $\mathcal{P%
}_{\mu }=\omega _{\mathcal{A}}^{\mu }(\Sigma ).$ The relation $\mathcal{P}%
_{0}\otimes \mathcal{P}_{\mu }=\det \mathcal{P}\simeq \mathcal{L}$ implies
the triviality of $\mathcal{P}_{0}$ as well.
\end{proof}

Note that the trivialization of $\mathcal{L}$ and $\mathcal{P}_{\mu }$ is
canonical, because it uses only the tautological map $\varepsilon $ which
exists over the Igusa scheme. The trivialization of $\mathcal{P}_{0}$ on the
other hand depends on how we realize the isomorphism $\det \mathcal{P}\simeq 
\mathcal{L}.$

We can now give an alternative proof to the fact that $\mathcal{L}^{p^{2}-1}$
and $\mathcal{P}_{\mu }^{p^{2}-1}$ are trivial on $\bar{S}_{\mu }.$ Denote
by $\mathcal{O}_{Ig}$ the structure sheaf of $\overline{Ig}_{\mu }.$ By the
projection formula, $\tau _{*}(\tau ^{*}\mathcal{L)}\simeq \mathcal{L}%
\otimes \tau _{*}\mathcal{O}_{Ig}$. Taking determinants we get 
\begin{equation}
\det \tau _{*}(\tau ^{*}\mathcal{L)}\simeq \mathcal{L}^{p^{2}-1}\otimes \det
\tau _{*}\mathcal{O}_{Ig}.
\end{equation}
As $\tau ^{*}\mathcal{L}\simeq \mathcal{O}_{Ig}$, we get that $\mathcal{L}%
^{p^{2}-1}\simeq \mathcal{O}_{\bar{S}}$. The same argument works for $%
\mathcal{P}_{\mu }$ and for $\mathcal{P}_{0}.$ The fact that $\mathcal{P}%
_{0}^{p+1}$ is already trivial could be deduced by a similar argument had we
worked out an analogue of $Ig(p)$ classifying \emph{symplectic isomorphisms}
of $\frak{G}[p]$ with $gr^{1}A[p].$ The role of $\Delta (p)$ for such a
moduli space would be assumed by 
\begin{equation}
\Delta ^{1}(p)=\ker (N:(\mathcal{O}_{\mathcal{K}}/p\mathcal{O}_{\mathcal{K}%
})^{\times }\rightarrow \Bbb{F}_{p}^{\times }),
\end{equation}
which is a group of order $p+1.$ We do not go any further in this direction
here.

\subsection{Compactification of the Igusa surface along the supersingular
locus}

\subsubsection{Extracting a $p^{2}-1$ root from $h_{\bar{\Sigma}}$ over $%
\overline{Ig}_{\mu }$}

Let $a$ be the canonical nowhere vanishing section of $\mathcal{L}$ over $%
\overline{Ig}_{\mu }$ which is sent to $e_{\bar{\Sigma}}\cdot (1\otimes
dT/T) $ under the trivialization 
\begin{equation}
\varepsilon ^{*}:\mathcal{L}=\omega _{\mathcal{A}}^{\mu }(\bar{\Sigma}%
)\simeq (\mathcal{O}_{\mathcal{K}}\otimes \omega _{\Bbb{G}_{m}})(\bar{\Sigma}%
)=R(\bar{\Sigma}).
\end{equation}
Here $R$ is any $R_{0}/pR_{0}$-algebra over which we choose to work. In
other words, $a=(\varepsilon ^{*})^{-1}(e_{\bar{\Sigma}}\cdot 1\otimes
dT/T). $ Dually, $a$ is the homomorphism from $Lie(\mathcal{A})(\bar{\Sigma}%
) $ to $\delta _{\mathcal{K}}^{-1}\otimes Lie(\Bbb{G}_{m})(\bar{\Sigma})$
arising from $\varepsilon ^{-1}.$ Let $a(k)=a^{\otimes k}\in H^{0}(\overline{%
Ig}_{\mu },\mathcal{L}^{k}).$

\begin{proposition}
\label{Root of Hasse}(i) Let $\gamma \in \Delta (p)=\mathcal{(O}_{\mathcal{K}%
}/p\mathcal{O}_{\mathcal{K}})^{\times }.$ Then $\Delta (p)$ acts on $H^{0}(%
\overline{Ig}_{\mu },\mathcal{L)}$ and 
\begin{equation}
\gamma ^{*}a=\bar{\Sigma}(\gamma )^{-1}\cdot a.
\end{equation}
(ii) The section $a$ is a $p^{2}-1$ root of the Hasse invariant over $%
\overline{Ig}_{\mu }$, i.e. 
\begin{equation}
a(p^{2}-1)=h_{\bar{\Sigma}}.
\end{equation}
\end{proposition}

\begin{proof}
(i) This part is a restatement of the action of $\Delta (p).$ At two points
of $Ig_{\mu }(R)$ lying over the same point of $S_{\mu }(R)$ and differing
by the action of $\gamma \in \Delta (p),$ the canonical embeddings 
\begin{equation}
\delta _{\mathcal{K}}^{-1}\otimes \mu _{p}\hookrightarrow A[p]
\end{equation}
differ by $\iota (\gamma )$ (\ref{gamma action}). The induced
trivializations of $Lie(A)(\bar{\Sigma})$ differ by $\bar{\Sigma}(\gamma )$
and by duality we get (i).

(ii) Since over any $\Bbb{F}_{p}$-base, $Ver_{\Bbb{G}_{m}}=1$, we have a
commutative diagram 
\begin{equation}
\begin{array}{lll}
Lie(\mathcal{A})(\bar{\Sigma})^{(p^{2})} & \overset{V_{*}^{2}}{\rightarrow }
& Lie(\mathcal{A})(\bar{\Sigma}) \\ 
\downarrow a^{(p^{2})} &  & \downarrow a \\ 
\delta _{\mathcal{K}}^{-1}\otimes Lie(\Bbb{G}_{m})(\bar{\Sigma}) & = & 
\delta _{\mathcal{K}}^{-1}\otimes Lie(\Bbb{G}_{m})(\bar{\Sigma})
\end{array}
.
\end{equation}
Using the isomorphism $Lie(\mathcal{A})(\bar{\Sigma})^{(p^{2})}\simeq Lie(%
\mathcal{A})(\bar{\Sigma})^{p^{2}}$ we get the commutative diagram 
\begin{equation}
\begin{array}{lll}
Lie(\mathcal{A})(\bar{\Sigma})^{p^{2}} & \overset{h_{\bar{\Sigma}}}{%
\rightarrow } & Lie(\mathcal{A})(\bar{\Sigma}) \\ 
\downarrow a(p^{2}) &  & \downarrow a \\ 
\delta _{\mathcal{K}}^{-1}\otimes Lie(\Bbb{G}_{m})(\bar{\Sigma}) & = & 
\delta _{\mathcal{K}}^{-1}\otimes Lie(\Bbb{G}_{m})(\bar{\Sigma})
\end{array}
,
\end{equation}
from which we deduce that $h_{\bar{\Sigma}}=a(p^{2}-1)$.
\end{proof}

\subsubsection{The compactification $\overline{Ig}$ of $\overline{Ig}_{\mu }$%
}

In this section we follow the method outlined in [An-Go, Sections 6-9] and
[Gor] for Hilbert modular varieties. Quite generally, let $L\rightarrow X$
be a line bundle associated with an invertible sheaf $\mathcal{L}$ on a
scheme $X$. Write $L^{n}$ for the line bundle $L^{\otimes n}$ over $X$. Let $%
s:X\rightarrow L^{n}$ be a section. Consider the fiber product 
\begin{equation}
Y=L\times _{L^{n}}X
\end{equation}
where the two maps to $L^{n}$ are $\lambda \mapsto \lambda ^{n}$ and $s.$
Let $p:Y\overset{pr_{2}}{\rightarrow }X$ be the projection which factors
also as $Y\overset{pr_{1}}{\rightarrow }L\rightarrow L^{n}\rightarrow X$
(since $X\overset{s}{\rightarrow }L^{n}\rightarrow X$ is the identity).
Consider 
\begin{equation}
p^{*}L=L\times _{X}(L\times _{L^{n}}X).
\end{equation}
This line bundle on $Y$ has a tautological section $t:Y\rightarrow p^{*}L,$%
\begin{equation}
t:y=(\lambda ,x)\mapsto (\lambda ,y)=(\lambda ,(\lambda ,x))
\end{equation}
Here $s(x)=\lambda ^{n}$ and 
\begin{equation}
t^{n}(y)=(\lambda ^{n},y)=(s(x),y)=p^{*}s(y)
\end{equation}
so $t$ is an $n$th root of $p^{*}s.$ Moreover, $Y$ has the \emph{universal
property} with respect to extracting $n$th roots from $s$: If $%
p_{1}:Y_{1}\rightarrow X$, and $t_{1}\in \Gamma (Y_{1},p_{1}^{*}L)$ is such
that $t_{1}^{n}=p_{1}^{*}s,$ then there exists a unique morphism $%
h:Y_{1}\rightarrow Y$ covering the two maps to $X$ such that $t_{1}=h^{*}t.$

The map $L\rightarrow L^{n}$ is \emph{finite flat of degree }$n$ and if $n$
is invertible on the base, finite \'{e}tale away from the zero section.
Indeed, locally on $X$ it is the map $\Bbb{A}^{1}\times X\rightarrow \Bbb{A}%
^{1}\times X$ which is just raising to $n$th power in the first coordinate.
By base-change, it follows that the same is true for the map $p:Y\rightarrow
X:$ this map is finite flat of degree $n$ and \'{e}tale away from the
vanishing locus of the section $s$ (assuming $n$ is invertible). We remark
that if $L$ is the trivial line bundle, we recover usual Kummer theory.

Applying this in our example with $n=p^{2}-1$ we define the \emph{complete
Igusa surface of level} $p,$ $\overline{Ig}=\overline{Ig}(p)$ as 
\begin{equation}
\overline{Ig}=\mathcal{L}\times _{\mathcal{L}^{p^{2}-1}}\bar{S}
\end{equation}
where the map $\overline{S}\rightarrow \mathcal{L}^{p^{2}-1}$ is $h_{\bar{%
\Sigma}}.$ From the universal property and part (ii) of Proposition \ref
{Root of Hasse} we get a map of $\bar{S}$-schemes 
\begin{equation}
\overline{Ig}_{\mu }\rightarrow \overline{Ig}.
\end{equation}
This map is an isomorphism over $\bar{S}_{\mu }$ because both schemes are
\'{e}tale torsors for $\Delta (p)=(\mathcal{O}_{\mathcal{K}}/p\mathcal{O}_{%
\mathcal{K}})^{\times }$ and the map respects the action of this group. We
summarize the discussion in the following theorem (for the last point,
consult [Mu2], Proposition 2, p.198).

\begin{theorem}
\label{Igusa}The morphism $\tau :\overline{Ig}\rightarrow \bar{S}$ satisfies
the following properties:

(i) It is finite flat of degree $p^{2}-1,$ \'{e}tale over $\bar{S}_{\mu },$
totally ramified over $S_{ss}.$

(ii) $\Delta (p)$ acts on $\overline{Ig}$ as a group of deck transformations
and the quotient is $\bar{S}.$

(iii) Let $s_{0}\in S_{gss}(\Bbb{\bar{F}}_{p}).$ Then there exist local
parameters $u,v$ at $s_{0}$ such that $\widehat{\mathcal{O}}_{S,s_{0}}=\Bbb{%
\bar{F}}_{p}[[u,v]]$, $S_{gss}\subset S$ is formally defined by $u=0,$ and
if $\tilde{s}_{0}\in Ig$ maps to $s_{0}$ under $\tau ,$ then $\widehat{%
\mathcal{O}}_{Ig,\tilde{s}_{0}}=\Bbb{\bar{F}}_{p}[[w,v]]$ where $%
w^{p^{2}-1}=u.$ In particular, $Ig$ is regular in codimension 1.

(iv) Let $s_{0}\in S_{ssp}(\Bbb{\bar{F}}_{p}).$ Then there exist local
parameters $u,v$ at $s_{0}$ such that $\widehat{\mathcal{O}}_{S,s_{0}}=\Bbb{%
\bar{F}}_{p}[[u,v]]$, $S_{ss}\subset S$ is formally defined at $s_{0}$ by $%
u^{p+1}+v^{p+1}=0,$ and if $\tilde{s}_{0}\in Ig$ maps to $s_{0}$ under $\tau
,$ then 
\begin{equation}
\widehat{\mathcal{O}}_{Ig,\tilde{s}_{0}}=\Bbb{\bar{F}}%
_{p}[[w,u,v]]/(w^{p^{2}-1}-u^{p+1}-v^{p+1})
\end{equation}
In particular, $\tilde{s}_{0}$ is a normal singularity of $Ig.$
\end{theorem}

\subsubsection{\label{Ig irreducibility}Irreducibility of $Ig$}

So far we have avoided the delicate question of whether $\overline{Ig}$ is
``relatively irreducible'', i.e. whether $\tau ^{-1}(T)$ is irreducible if $%
T\subset \bar{S}$ is an irreducible (equivalently, connected) component.
Using an idea of Katz, and following the approach taken by Ribet in [Ri],
the irreducibility of $\tau ^{-1}(T)$ could be proven for any level $p^{n}$
if we could prove the following:

\begin{itemize}
\item  Let $q=p^{2}.$ For any $r$ sufficiently large and for any $\gamma \in
(\mathcal{O}_{\mathcal{K}}/p^{n}\mathcal{O}_{\mathcal{K}})^{\times }$ there
exists a $\mu $-ordinary abelian variety with PEL structure $\underline{A}%
\in S_{\mu }(\Bbb{F}_{q^{r}})$ such that the image of $Gal(\Bbb{\bar{F}}_{q}/%
\Bbb{F}_{q^{r}})$ in 
\begin{equation}
Aut\left( Isom_{\Bbb{\bar{F}}_{q}}(\delta _{\mathcal{K}}^{-1}\otimes \mu
_{p^{n}},A[p^{n}]^{\mu })\right) =(\mathcal{O}_{\mathcal{K}}/p^{n}\mathcal{O}%
_{\mathcal{K}})^{\times }
\end{equation}
contains $\gamma .$
\end{itemize}

See also the discussion in \ref{Irreducibility again}. Instead, we shall
give a different argument valid for the case $n=1.$

\begin{proposition}
\label{irreducibility}The morphism $\tau :\overline{Ig}\rightarrow \bar{S}$
induces a bijection on irreducible components.
\end{proposition}

\begin{proof}
Since $\overline{Ig}$ is a normal surface, connected components and
irreducible components are the same. Let $T$ be a connected component of $%
\bar{S}$ and $T_{ss}=T\cap S_{ss}.$ Let $\tau ^{-1}(T)=\coprod Y_{i}$ be the
decomposition into connected components. As $\tau $ is finite and flat, each 
$\tau (Y_{i})=T.$ Since $\tau $ is totally ramified over $T_{ss},$ there is
only one $Y_{i}.$
\end{proof}

\section{Modular forms modulo $p$ and the theta operator}

\subsection{Modular forms $\mod p$ as functions on $Ig$}

\subsubsection{Representing modular forms by functions on $Ig$}

The Galois group $\Delta (p)=(\mathcal{O}_{\mathcal{K}}/p\mathcal{O}_{%
\mathcal{K}})^{\times }$ acts on the coordinate ring $H^{0}(Ig_{\mu },%
\mathcal{O})$ and we let $H^{0}(Ig_{\mu },\mathcal{O})^{(k)}$ be the
subspace where it acts via the character $\bar{\Sigma}^{k}$. Then 
\begin{equation}
H^{0}(Ig_{\mu },\mathcal{O})=\bigoplus_{k=0}^{p^{2}-2}H^{0}(Ig_{\mu },%
\mathcal{O})^{(k)}
\end{equation}
and each $H^{0}(Ig_{\mu },\mathcal{O})^{(k)}$ is free of rank 1 over $%
H^{0}(S_{\mu },\mathcal{O})=H^{0}(Ig_{\mu },\mathcal{O})^{(0)}.$

For any $0\le k$ the map $f\mapsto f/a(k)$ is an embedding 
\begin{equation}
M_{k}(N,\kappa _{0})\hookrightarrow H^{0}(Ig_{\mu },\mathcal{O})^{(k)}.
\end{equation}

\begin{lemma}
Fix $0\le k<p^{2}-1.$ Then we have a \emph{surjective} homomorphism 
\begin{equation}
\bigoplus_{n\ge 0}M_{k+n(p^{2}-1)}(N,\kappa _{0})\twoheadrightarrow
H^{0}(Ig_{\mu },\mathcal{O})^{(k)}.
\end{equation}
\end{lemma}

\begin{proof}
Take $f\in H^{0}(Ig_{\mu },\mathcal{O})^{(k)},$ so that $f\cdot a(k)\in
H^{0}(Ig_{\mu },\mathcal{L}^{k})^{(0)},$ hence descends to $g\in
H^{0}(S_{\mu },\mathcal{L}^{k}).$ This $g$ may have poles along $S_{ss},$
but some $h_{\bar{\Sigma}}^{n}g$ will extend holomorphically to $S$, hence
represents a modular form of weight $k+n(p^{2}-1),$ which will map to $f$
because $a(k+n(p^{2}-1))=h_{\bar{\Sigma}}^{n}a(k).$
\end{proof}

\begin{proposition}
The resulting ring homomorphism 
\begin{equation}
r:\bigoplus_{k\ge 0}M_{k}(N,\kappa _{0})\twoheadrightarrow H^{0}(Ig_{\mu },%
\mathcal{O)}
\end{equation}
obtained by dividing a modular form of weight $k$ by $a(k)$ is surjective,
respects the $\Bbb{Z}/(p^{2}-1)\Bbb{Z}$-grading on both sides, and its
kernel is the ideal generated by $(h_{\bar{\Sigma}}-1)$.
\end{proposition}

\begin{proof}
We only have to prove that anything in $\ker (r)$ is a multiple of $h_{\bar{%
\Sigma}}-1,$ the rest being clear. Since $r$ respects the grading, we may
assume that for some $k\ge 0$ we have $f_{j}\in M_{k+j(p^{2}-1)}(S,\kappa
_{0})$ and $f=\sum_{j=0}^{m}f_{j}\in \ker (r),$ i.e. 
\begin{equation}
\sum_{j=0}^{m}a(k)^{-1}h_{\bar{\Sigma}}^{-j}f_{j}=0.
\end{equation}
But then $f_{m}=-h_{\bar{\Sigma}}^{m}\left( \sum_{j=0}^{m-1}h_{\bar{\Sigma}%
}^{-j}f_{j}\right) ,$ so $\sum_{j=0}^{m}f_{j}=\sum_{j=0}^{m-1}(1-h_{\bar{%
\Sigma}}^{m-j})f_{j}$ belongs to $(1-h_{\bar{\Sigma}}).$
\end{proof}

As a result we get that 
\begin{equation}
Ig_{\mu }^{*}=Spec\left( \bigoplus_{k\ge 0}M_{k}(N,\kappa _{0})/(h_{\bar{%
\Sigma}}-1)\right)
\end{equation}
and 
\begin{equation}
S_{\mu }^{*}=Spec\left( \bigoplus_{k\ge 0}M_{k(p^{2}-1)}(N,\kappa _{0})/(h_{%
\bar{\Sigma}}-1)\right) .
\end{equation}

\subsubsection{Fourier-Jacobi expansions modulo $p$}

The arithmetic Fourier-Jacobi expansion (\ref{FJ}) depended on a choice of a
nowhere vanishing section $s$ of $\mathcal{L}$ along the boundary $C=\bar{S}%
\backslash S$ of $\bar{S}.$ As the boundary $\tilde{C}=\overline{Ig}_{\mu
}\backslash Ig_{\mu }$ is (non-canonically) identified with $\Delta
(p)\times C,$ we may ``compute'' the Fourier-Jacobi expansion on the Igusa
surface rather than on $S.$ But on the Igusa surface, $a$ is a canonical
choice for such an $s.$ We may therefore associate a \emph{canonical }%
Fourier-Jacobi expansion 
\begin{equation}
\widetilde{FJ}(f)=\sum_{m=0}^{\infty }c_{m}(f)\in \prod_{m=0}^{\infty }H^{0}(%
\tilde{C},\mathcal{N}^{m})
\end{equation}
along the boundary of $Ig,$ with every 
\begin{equation}
f\in M_{*}(N,R)=\bigoplus_{k=0}^{\infty }M_{k}(N,R)
\end{equation}
($R$ a $\kappa _{0}$-algebra). The following proposition becomes almost a 
\emph{tautology.}

\begin{proposition}
The Fourier-Jacobi expansion $\widetilde{FJ}(h_{\bar{\Sigma}})$ of the Hasse
invariant is 1. Moreover, for $f_{1}$ and $f_{2}$ in the graded ring $%
M_{*}(N,R),$ $r(f_{1})=r(f_{2})$ if and only if $\widetilde{FJ}(f_{1})=%
\widetilde{FJ}(f_{2}).$
\end{proposition}

\begin{proof}
The first statement is tautologically true. For the second, note that for $%
f\in M_{k}(N,R),$ $\widetilde{FJ}(f)$ is the (expansion of the) image of $%
f/a(k)$ in $H^{0}(\tilde{C},\mathcal{O}_{\widehat{Ig}})$ where $\widehat{Ig}$
is the formal completion of $Ig$ along $\tilde{C},$ while $r(f)$ is the
image of $f/a(k)$ in $H^{0}(\overline{Ig}_{\mu },\mathcal{O}).$ The
proposition follows from the fact that by Proposition \ref{irreducibility}
the irreducible components of $\overline{Ig}_{\mu }$ are in bijection with
the connected components of $\bar{S},$ so every irreducible component of $%
\overline{Ig}_{\mu }$ contains at least one cuspidal component (``$q$%
-expansion principle''). A function on $\overline{Ig}_{\mu }$ that vanishes
in the formal neighborhood of any cuspidal component must therefore vanish
on any irreducible component, so is identically $0.$
\end{proof}

\subsubsection{The filtration of a modular form modulo $p$}

Let $f\in M_{k}(N,R),$ where $R$ is a $\kappa _{0}$-algebra as before.
Define the \emph{filtration} $\omega (f)$ to be the minimal $j\ge 0$ such
that $r(f)=r(f^{\prime })$ (equivalently $FJ(f)=FJ(f^{\prime })$) for some $%
f^{\prime }\in M_{j}(N,R).$ The following proposition follows immediately
from previous results.

\begin{proposition}
Let $f\in M_{k}(N,R).$ Then $0\le \omega (f)\le k$ and 
\begin{equation}
\omega (f)\equiv k\mod (p^{2}-1).
\end{equation}
Let $\omega (f)=k-(p^{2}-1)n.$ Then $n$ is\emph{\ the order of vanishing of} 
$f$ along $S_{ss}.$ Equivalently, $k-\omega (f)$ is the order of vanishing
of the pull-back of $f$ to $Ig$ along $Ig_{ss}.$ In addition, $\omega
(f^{m})=m\omega (f).$
\end{proposition}

\subsection{The theta operator}

\subsubsection{Definition of $\Theta (f)$}

We work over $\kappa =\Bbb{\bar{F}}_{p}$. Let $S$ be the (open) Picard
surface over $\kappa $ and $Ig=Ig(p)$ the Igusa surface of level $p$
(completed along the supersingular locus as explained above). To simplify
the notation we denote by $Z=S_{ss}=S\backslash S_{\mu }$ the supersingular
locus of $S,$ by $\tilde{Z}=Ig_{ss}=Ig\backslash Ig_{\mu }$ its pre-image
under the covering map $\tau :Ig\rightarrow S,$ by $Z^{\prime
}=S_{gss}=S_{ss}\backslash S_{ssp}$ the smooth part of $Z,$ and by $\tilde{Z}%
^{\prime }=Ig_{gss}=Ig_{ss}\backslash Ig_{ssp}$ the pre-image of $Z^{\prime
} $ under $\tau .$

Let $f\in H^{0}(S,\mathcal{L}^{k}).$ Then $\tau ^{*}f/a^{k}\in H^{0}(Ig_{\mu
},\mathcal{O})$ has a pole of order at most $k$ along $\tilde{Z},$ and the
Galois group acts on it via $\bar{\Sigma}^{k}.$ Let 
\begin{equation}
\eta _{f}=d(\tau ^{*}f/a^{k})\in H^{0}(Ig_{\mu },\Omega
_{Ig}^{1})=H^{0}(Ig_{\mu },\tau ^{*}\Omega _{S}^{1}).
\end{equation}
The Kodaira-Spencer isomorphism $KS(\Sigma )$ is an isomorphism 
\begin{equation}
KS(\Sigma ):\mathcal{P}\otimes \mathcal{L}\simeq \Omega _{S}^{1}.
\end{equation}
Let 
\begin{equation}
\psi =(V_{\mathcal{P}}\otimes 1)\circ KS(\Sigma )^{-1}:\Omega
_{S}^{1}\rightarrow \mathcal{L}^{(p)}\otimes \mathcal{L}\simeq \mathcal{L}%
^{p+1}.  \label{kappa}
\end{equation}

We denote by $\psi $ also the map induced on the base-change of these vector
bundles by $\tau ^{*}$ to $Ig$ and consider $\psi (\eta _{f}).$ As $\Delta
(p)$ still acts on $\psi (\eta _{f})$ via $\bar{\Sigma}^{k},$ its action on $%
a^{k}\psi (\eta _{f})$ is trivial, so this section descends to $S_{\mu }.$
We define 
\begin{equation}
\Theta (f)=a^{k}\psi (\eta _{f})\in H^{0}(S_{\mu },\mathcal{L}^{k+p+1}).
\end{equation}
A priori, this extends only to a meromorphic modular form of weight $k+p+1,$
as it may have poles along $Z$.

\subsubsection{The main theorem}

For the formulation of the next theorem we need to define what we mean by
the \emph{standard cuspidal component} of $\bar{S}$ or $\overline{Ig}.$
Since its definition involves a transition back and forth between $\Bbb{C}$
and $\kappa $ we need to fix, besides the embedding of $R_{N}$ in $\Bbb{C}$
also a homomorphism 
\begin{equation}
R_{N}\rightarrow \kappa
\end{equation}
extending the map $R_{0}\rightarrow \kappa _{0}\subset \kappa $, and we let $%
\frak{P}$ be its kernel (a prime above $p$).

Recall that according to [Bel] and [La1] the cuspidal scheme $C=\bar{S}%
\backslash S$ classifies $\mathcal{O}_{\mathcal{K}}$-semi-abelian varieties
with level $N$ structure. The \emph{standard component} of $C$ over $\Bbb{C}$
is the component which classifies extensions of the elliptic curve $\Bbb{C}/%
\mathcal{O}_{\mathcal{K}}$ by the $\mathcal{O}_{\mathcal{K}}$-torus $%
\mathcal{O}_{\mathcal{K}}\otimes \Bbb{C}^{\times }$ (thus sits over a cusp
of type $(\mathcal{O}_{\mathcal{K}},\mathcal{O}_{\mathcal{K}})$ in $S_{\Bbb{C%
}}^{*}$), together with a level-$N$ structure $(\alpha ,\beta ,\gamma )$
(see [Bel], I.4.2 and Section \ref{semi abelian model}), where 
\begin{equation}
\alpha :\mathcal{O}_{\mathcal{K}}/N\mathcal{O}_{\mathcal{K}}=\mathcal{O}_{%
\mathcal{K}}\otimes \Bbb{Z}/N\Bbb{Z} \rightarrow \mathcal{O}_{\mathcal{K}%
}\otimes \Bbb{C}^{\times }
\end{equation}
is given by $1\otimes (a\mapsto \exp (2\pi ia/N))$ and 
\begin{equation}
\beta :\mathcal{O}_{\mathcal{K}}/N\mathcal{O}_{\mathcal{K}}=N^{-1}\mathcal{O}%
_{\mathcal{K}}/\mathcal{O}_{\mathcal{K}}\rightarrow \Bbb{C}/\mathcal{O}_{%
\mathcal{K}}
\end{equation}
is the canonical embedding. (The splitting $\gamma $ varies along the
component.) The standard component of $C$ over $R_{N}$ is the one which
becomes this component after base change to $\Bbb{C}$. The standard
component of $C$ over $\kappa $ is the reduction modulo $\frak{P}$ of the
standard component of $C$ over $R_{N}$. Finally, $\overline{Ig}$ maps to $%
\bar{S}$ (over $\kappa $) and the cuspidal components mapping to a given
component $E$ of $C$ are classified by the embedding of $\delta _{\mathcal{K}%
}^{-1}\mathcal{O}_{\mathcal{K}}\otimes \mu _{p}$ in the toric part of $%
\mathcal{A}.$ Since the toric part of the universal semi-abelian variety
over the standard component is $\mathcal{O}_{\mathcal{K}}\otimes \Bbb{G}%
_{m}, $ we may define the standard cuspidal component of $\overline{Ig}$ to
be the component where the map 
\begin{equation}
\varepsilon :\delta _{\mathcal{K}}^{-1}\mathcal{O}_{\mathcal{K}}\otimes \mu
_{p}\rightarrow \mathcal{O}_{\mathcal{K}}\otimes \Bbb{G}_{m}
\label{standard Igusa}
\end{equation}
is the natural embedding. Here we use the fact that 
\begin{equation}
\delta _{\mathcal{K}}^{-1}\mathcal{O}_{\mathcal{K}}\otimes \mu _{p}=\mathcal{%
O}_{\mathcal{K}}\otimes \mu _{p}
\end{equation}
since $\delta _{\mathcal{K}}$ is invertible in $\mathcal{O}_{\mathcal{K}}/p%
\mathcal{O}_{\mathcal{K}}.$ Let $\tilde{E}\subset \tilde{C}=$ $\overline{Ig}%
\backslash Ig$ be \emph{this} standard component.

\begin{theorem}
\label{Main Theorem}(i) The operator $\Theta $ maps $H^{0}(S,\mathcal{L}%
^{k}) $ to $H^{0}(S,\mathcal{L}^{k+p+1}).$

(ii) The effect of $\Theta $ on Fourier-Jacobi expansions is a ``Tate
twist''. More precisely, let 
\begin{equation}
\widetilde{FJ}(f)=\sum_{m=0}^{\infty }c_{m}(f)
\end{equation}
be the canonical Fourier-Jacobi expanison of $f$ along $\tilde{E}$ (thus $%
c_{m}(f)\in H^{0}(\tilde{E},\mathcal{N}^{m})$). Then 
\begin{equation}
\widetilde{FJ}(\Theta (f))=M^{-1}\sum_{m=0}^{\infty }mc_{m}(f).
\label{qd/dq}
\end{equation}
Here $M$ (equal to $N|D_{\mathcal{K}}|$ or $2^{-1}N|D_{\mathcal{K}}|$) is
the width of the cusp.

(iii) If $f\in H^{0}(S,\mathcal{L}^{k})$ and $g\in H^{0}(S,\mathcal{L}^{l})$
then 
\begin{equation}
\Theta (fg)=f\Theta (g)+\Theta (f)g.
\end{equation}

(iv) $\Theta (h_{\bar{\Sigma}}f)=h_{\bar{\Sigma}}\Theta (f)$ (equivalently, $%
\Theta (h_{\bar{\Sigma}})=0$).
\end{theorem}

\begin{corollary}
The operator $\Theta $ extends to a derivation of the graded ring of modular
forms $\mod p$, and for any $f,$ $\Theta (f)$ is a cusp form.
\end{corollary}

Parts (iii) and (iv) of the theorem are clear from the construction. The
proof of (i), that $\Theta (f)$ is in fact $\emph{holomorphic}$ along $%
S_{ss},$ will be given in \ref{Theta on ss}. We shall now study its effect
on Fourier-Jacobi expansions, i.e. part (ii). That a factor like $M^{-1}$ is
necessary in (ii) becomes evident if we consider what happens to FJ
expansions under level change. If $N$ is replaced by $N^{\prime }=NQ$ then
the conormal bundle becomes the $Q$-th power of the conormal bundle of level 
$N^{\prime },$ i.e. $\mathcal{N}=\mathcal{N}^{\prime Q}$ (see Section \ref
{level change}). It follows that what was the $m$-th FJ coefficient at level 
$N$ becomes the $Qm$-th coefficient at level $N^{\prime }.$ The operator $%
\Theta $ commutes with level-change, but the factor $M^{-1},$ which changes
to $(QM)^{-1}$, takes care of this.

\subsection{The effect of $\Theta $ on FJ expansions}

Let $E$ be the standard cuspidal component of $\bar{S}$ (over the ring $%
R_{N} $). We have earlier trivialized the line bundle $\mathcal{L}$ along $E$
in two seemingly different ways, that we must now compare. On the one hand,
after reducing modulo $\frak{P}$ (the prime of $R_{N}$ above $p$ fixed
above) and pulling $\mathcal{L}$ back to the Igusa surface, we got a \emph{%
canonical} nowhere vanishing section $a$ trivializing $\mathcal{L}$ over $%
\overline{Ig}_{\mu },$ and in particular along any of the $p^{2}-1$ cuspidal
components lying over $E$ in $\overline{Ig}_{\mu }.$ Using $\tilde{E}$ as a
reference, there is a unique section of $\mathcal{L}$ along $E$ which pulls
back to $a|_{\tilde{E}}.$ On the other hand, extending scalars from $R_{N}$
to $\Bbb{C},$ shifting to the analytic category, restricting to the
connected component $\bar{X}_{\Gamma }$ on which $E$ lies, and then pulling
back to a neighborhood of the cusp $c_{\infty }$ in the unit ball $\frak{X},$
we have trivialized $\mathcal{L}|_{E}$ by means of the section $2\pi id\zeta
_{3},$ which we showed to be $\mathcal{K}_{N}$-rational.

\begin{lemma}
The sections $a|_{\tilde{E}}$ and $2\pi id\zeta _{3}$ ``coincide'' in the
sense that they come from the same section in $H^{0}(E,\mathcal{L}).$
\end{lemma}

\begin{proof}
Let $A$ be the universal semi-abelian variety over $E.$ Its toric part is $%
\mathcal{O}_{\mathcal{K}}\otimes \Bbb{G}_{m}$, hence, taking $\bar{\Sigma}$%
-component of the cotangent space at the origin 
\begin{equation}
\mathcal{L}|_{\tilde{E}}=\omega _{A/\tilde{E}}(\bar{\Sigma})=(\delta _{%
\mathcal{K}}^{-1}\mathcal{O}_{\mathcal{K}}\otimes \omega _{\Bbb{G}_{m}})(%
\bar{\Sigma})
\end{equation}
admits the canonical section $e_{\bar{\Sigma}}\cdot (1\otimes dT/T).$
Tracing back the definitions and using (\ref{semi-abelian}), this section
becomes, under the base change $R_{N}\hookrightarrow \Bbb{C},$ just $2\pi
id\zeta _{3}.$ On the other hand, when we reduce it modulo $\frak{P}$ and
use the Igusa level structure $\varepsilon $ at the standard cusp, it pulls
back to the section ``with the same name'' $e_{\bar{\Sigma}}\cdot (1\otimes
dT/T)$, because along $\tilde{E}$ (\ref{standard Igusa}) induces the
identity on cotangent spaces. The lemma follows from the fact that, by
definition, $\varepsilon ^{*}a=e_{\bar{\Sigma}}\cdot (1\otimes dT/T)$ too.
\end{proof}

\begin{lemma}
The sections $a|_{\tilde{E}}^{p+1}$ and $2\pi id\zeta _{2}\otimes 2\pi
id\zeta _{3}\mod \mathcal{P}_{0}\otimes \mathcal{L}$ ``coincide'' in
the sense that they come from the same section in $H^{0}(E,\mathcal{P}_{\mu
}\otimes \mathcal{L}).$
\end{lemma}

\begin{proof}
Let $\sigma _{2}$ (resp. $\sigma _{3}$) be the $\mathcal{K}_{N}$-rational
section of $\mathcal{P}_{\mu }$ (resp. $\mathcal{L}$) along $E$, which over $%
\Bbb{C}$ becomes the section $2\pi id\zeta _{2}$ (resp. $2\pi id\zeta _{3}$%
). We have just seen that modulo $\frak{P},$ when we identify $\tilde{E}$
with $E$ (via the covering map $\tau :\overline{Ig}\rightarrow \bar{S}$), $%
\sigma _{3}$ reduces to $a.$ To conclude, we must show that the map 
\begin{equation}
V:\mathcal{P}/\mathcal{P}_{0}=\mathcal{P}_{\mu }\simeq \mathcal{L}^{(p)}
\end{equation}
carries $\sigma _{2}$ to $\sigma _{3}^{(p)}.$ This will map, under $\mathcal{%
L}^{(p)}\simeq \mathcal{L}^{p},$ to $a^{p}.$ Along $E$ the line bundles $%
\mathcal{P}_{\mu }$ and $\mathcal{L}$ are just the $\Sigma $- and $\bar{%
\Sigma}$-parts of the cotangent space at the origin of the torus $\mathcal{O}%
_{\mathcal{K}}\otimes \Bbb{G}_{m},$ and $\sigma _{2}$ and $\sigma _{3}$ are
the sections 
\begin{equation}
\sigma _{2}=e_{\Sigma }\cdot (1\otimes dT/T),\,\,\sigma _{3}=e_{\bar{\Sigma}%
}\cdot (1\otimes dT/T).
\end{equation}
Since in characteristic $p,$ $V=Ver^{*}:\omega _{\Bbb{G}_{m}}\rightarrow
\omega _{\Bbb{G}_{m}}^{(p)}$ maps $dT/T$ to $(dT/T)^{(p)},$ for the $%
\mathcal{O}_{\mathcal{K}}$-torus, $V(\sigma _{2})=\sigma _{3}^{(p)}$, and we
are done.
\end{proof}

To prove part (ii) of the main theorem we argue as follows. Let $g=f/a^{k}$
be the function on $\overline{Ig}_{\mu }$ obtained by trivializing the line
bundle $\mathcal{L}.$ We have to study the FJ expansion along $\tilde{E}$ of 
$\psi (dg)/a^{p+1},$ where $\psi $ is the map defined in (\ref{kappa}). For
that purpose we may restrict to a formal neighborhood of $\tilde{E}$. This
formal neighborhood is isomorphic, under the covering map $\tau :$ $%
\overline{Ig}_{\mu }\rightarrow \bar{S}_{\mu },$ to the formal neighborhood $%
\widehat{S}$ of $E$ in $S.$ We may therefore regard $dg$ as an element of $%
\Omega _{\widehat{S}}^{1}.$ Now 
\begin{equation}
\psi :\Omega _{\widehat{S}}^{1}\rightarrow \mathcal{P}_{\mu }\otimes 
\mathcal{L}
\end{equation}
is a homomorphism of $\mathcal{O}_{\widehat{S}}$-modules defined over $R_{N}$
so, having restricted to $\widehat{S}$, we may study the effect of $\psi $
on FJ expansions by embedding $\widehat{S}_{\Bbb{C}}$ in a tubular
neighborhood $\bar{S}(\varepsilon )$ of $E$ and using complex analytic
Fourier-Jacobi expansions. We are thus reduced to a complex-analytic
computation, near the standard cusp at infinity.

Let 
\begin{equation}
g(z,u)=\sum_{m=0}^{\infty }\theta _{m}(u)q^{m}
\end{equation}
where $q=e^{2\pi iz/M}$ and $\theta _{m}$ is a theta function, so that $%
\theta _{m}(u)q^{m}$ is a section of $\mathcal{N}^{m}$ along $E$ (now over $%
\Bbb{C}$). Then 
\begin{equation}
dg=2\pi iM^{-1}\sum_{m=0}^{\infty }m\theta _{m}(u)q^{m}dz+\sum_{m=0}^{\infty
}\theta _{m}^{\prime }(u)q^{m}du.
\end{equation}
According to Corollary \ref{complex kappa}, $\psi (du)=0,$ and $\psi
(dz)=2\pi id\zeta _{2}\otimes d\zeta _{3}.$ It follows that 
\begin{equation}
\psi (dg)=M^{-1}\sum_{m=0}^{\infty }m\theta _{m}(u)q^{m}\cdot 2\pi id\zeta
_{2}\otimes 2\pi id\zeta _{3}.
\end{equation}
Recalling that in characteristic $p,$ $2\pi id\zeta _{2}\otimes 2\pi id\zeta
_{3}$ reduced to $a^{p+1},$ the proof of part (ii) of the theorem is now
complete. For the convenience of the reader we summarize the transitions
between complex and $p$-adic maps in the following diagram: 
\begin{equation}
\begin{array}{lllllll}
/\text{\thinspace }\kappa &  & \Omega _{\bar{S}_{\mu }/\kappa %
}^{1} & \overset{KS(\Sigma )^{-1}}{\rightarrow } & \mathcal{P}\otimes 
\mathcal{L} &  &  \\ 
&  & \cap &  & \downarrow \mod \mathcal{P}_{0} &  &  \\ 
/\kappa &  & \Omega _{\widehat{S}/\kappa }^{1} & \overset{\psi 
}{\rightarrow } & \mathcal{P}_{\mu }\otimes \mathcal{L} & \overset{V\otimes 1%
}{\simeq } & \mathcal{L}^{p+1} \\ 
&  & \uparrow \mod \frak{P} &  & \uparrow &  &  \\ 
/R_{N} &  & \Omega _{\widehat{S}/R_{N}}^{1} & \overset{\psi }{\rightarrow }
& \mathcal{P}_{\mu }\otimes \mathcal{L} &  &  \\ 
&  & \downarrow \otimes _{R_{N}}\Bbb{C} &  & \downarrow &  &  \\ 
/\Bbb{C} &  & \Omega _{\widehat{S}/\Bbb{C}}^{1} & \overset{\psi }{%
\rightarrow } & \mathcal{P}_{\mu }\otimes \mathcal{L} &  &  \\ 
&  & \cup &  & \uparrow \mod \mathcal{P}_{0} &  &  \\ 
/\Bbb{C} &  & \Omega _{\bar{S}(\varepsilon )/\Bbb{C}}^{1} & \overset{%
KS(\Sigma )_{an}^{-1}}{\rightarrow } & \mathcal{P}\otimes \mathcal{L} &  & 
\end{array}
.
\end{equation}

We next turn to part (i).

\subsection{A study of the theta operator along the supersingular locus\label%
{Theta on ss}}

\subsubsection{De Rham cohomology in characteristic $p$}

We continue to consider the Picard surface $S$ over $\kappa$ and
recall some facts about de Rham cohomology in characteristic $p$. Let $%
U=Spec(R)\hookrightarrow S$ be a closed point $s_{0}$ ($R= \kappa =$ $%
\mathcal{O}_{S,s_{0}}/\frak{m}_{S,s_{0}}$), a nilpotent thickening of a
closed point, or an affine open subset of $S.$ We consider the restriction
of the universal abelian scheme to $R$ and denote it by $A/R.$ Let $%
A^{(p)}=R\otimes _{\phi ,R}A$ be its base change with respect to the map $%
\phi (x)=x^{p}.$ Let 
\begin{equation}
D=H_{dR}^{1}(A/R),
\end{equation}
a locally free $R$-module of rank $6.$ The de Rham cohomology of $A^{(p)}$
is 
\begin{equation}
D^{(p)}=R\otimes _{\phi ,R}D.
\end{equation}
The $R$-linear Frobenius and Verschiebung morphisms $Frob:A\rightarrow
A^{(p)},$ $Ver:A^{(p)}\rightarrow A$ induce (by pull-back) linear maps 
\begin{equation}
F:D^{(p)}\rightarrow D,\,\,V:D\rightarrow D^{(p)}.
\end{equation}
Both $F$ and $V$ are everywhere of rank 3, which implies that their kernel
and image are locally free direct summands. Moreover, $\text{Im} F=\ker V$
and $\text{Im} V=\ker F=\omega _{A^{(p)}/R}.$ The maps $F$ and $V$ preserve
the types $\Sigma ,\bar{\Sigma},$ but note that $D^{(p)}(\Sigma )=D(\bar{%
\Sigma})^{(p)}$ etc.

The principal polarization on $A$ induces one on $A^{(p)}$, and these
polarizations induce symplectic forms 
\begin{equation}
\left\langle ,\right\rangle :D\times D\rightarrow R,\,\,\left\langle
,\right\rangle ^{(p)}:D^{(p)}\times D^{(p)}\rightarrow R
\end{equation}
where the second form is just the base-change of the first. For $x\in
D^{(p)},y\in D$ we have 
\begin{equation}
\left\langle Fx,y\right\rangle =\left\langle x,Vy\right\rangle ^{(p)}.
\label{FVduality}
\end{equation}
In addition, for $a\in \mathcal{O}_{\mathcal{K}}$ 
\begin{equation}
\left\langle \iota (a)x,y\right\rangle =\left\langle x,\iota (\bar{a}%
)y\right\rangle .
\end{equation}
As $VF=FV=0,$ the first relation implies that $\text{Im} F$ and $\text{Im} V$
are isotropic subspaces. So is $\omega _{A/R}.$

The Gauss-Manin connection is an integrable connection 
\begin{equation}
\nabla :D\rightarrow \Omega _{R}^{1}\otimes D.
\end{equation}
It is a priori defined (e.g. in [Ka-O]) when $R$ is smooth over $\kappa 
$, but we can define it by base change also when $R$ is a nilpotent
thickening of a point of $S$ (see [Kob], where $R$ is a local Artinian ring).

We shall need to deal only with the first infinitesimal neighborhood of a
point, $R=\mathcal{O}_{S,s_{0}}/\frak{m}_{S,s_{0}}^{2}.$ In this case, $D$
has a basis of horizontal sections. Indeed, $R=\kappa %
[u,v]/(u^{2},uv,v^{2})$ where $u$ and $v$ are local parameters at $s_{0},$
and 
\begin{equation*}
\Omega _{R}^{1}=(Rdu+Rdv)/\left\langle udu,\,vdv,\,udv+vdu\right\rangle
\end{equation*}
($p$ is odd). If $x\in D$ and 
\begin{equation}
\nabla x=du\otimes x_{1}+dv\otimes x_{2}
\end{equation}
then $\tilde{x}=x-ux_{1}-vx_{2}$ is horizontal, so the horizontal sections
span $D$ over $R$ by Nakayama's lemma. It follows that if $D_{0}=D^{\nabla }$
is the space of hoizontal sections, 
\begin{equation}
R\otimes _{\kappa }D_{0}=D,
\end{equation}
$\nabla =d\otimes 1$ and we can idnetify $D_{0}=H_{dR}^{1}(A_{s_{0}}/
\kappa ),$ i.e. every de Rham class at $s_{0}$ has a \emph{unique}
extension to a horizontal section $x\in H_{dR}^{1}(A/R).$

There is a similar connection on $D^{(p)}.$ The isogenies $Frob$ and $Ver,$
like any isogeny, take horizontal sections with respect to the Gauss-Manin
connection to horizontal sections, e.g. if $x\in D$ and $\nabla x=0$ then $%
Vx\in D^{(p)}$ satisfies $\nabla (Vx)=0.$

The pairing $\left\langle ,\right\rangle $ is horizontal for $\nabla ,$ i.e. 
\begin{equation}
d\left\langle x,y\right\rangle =\left\langle \nabla x,y\right\rangle
+\left\langle x,\nabla y\right\rangle .
\end{equation}

\begin{remark}
In the theory of Dieudonn\'{e} modules one works over a perfect base. It is
then customary to identify $D$ with $D^{(p)}$ via $x\leftrightarrow 1\otimes
x.$ This identification is only $\sigma $-linear where $\sigma =\phi $, now
viewed as an \emph{automorphism} of $R.$ The operator $F$ becomes $\sigma $%
-linear, $V$ becomes $\sigma ^{-1}$-linear and (\ref{FVduality}) reads $%
\left\langle Fx,y\right\rangle =\left\langle x,Vy\right\rangle ^{\sigma }.$
With this convention $F$ and $V$ switch types, rather than preserve them.
\end{remark}

\subsubsection{The Dieudonn\'{e} module at a \emph{gss} point}

Assume from now on that $s_{0}\in Z^{\prime }=S_{gss}$ is a closed point of
the general supersingular locus. We write $D_{0}$ for $H_{dR}^{1}(A_{s_{0}}/%
\kappa ).$

\begin{lemma}
\label{Braid}There exists a basis $e_{1},e_{2},f_{3},f_{1},f_{2},e_{3}$ of $%
D_{0}$ with the following properties. Denote by $e_{1}^{(p)}=1\otimes
e_{1}\in D_{0}^{(p)}$ etc.

(i) $\mathcal{O}_{\mathcal{K}}$ acts on the $e_{i}$ via $\Sigma $ and on the 
$f_{i}$ via $\bar{\Sigma}$ (hence it acts on the $e_{i}^{(p)}$ via $\bar{%
\Sigma}$ and on the $f_{i}^{(p)}$ via $\Sigma ).$

(ii) The symplectic pairing satisfies 
\begin{equation}
\left\langle e_{i},f_{j}\right\rangle =-\left\langle
f_{j},e_{i}\right\rangle =\delta _{ij},\,\,\,\left\langle
e_{i},e_{j}\right\rangle =\left\langle f_{i},f_{j}\right\rangle =0.
\end{equation}

(iii) The vectors $e_{1},e_{2},f_{3}$ form a basis for the cotangent space $%
\omega _{A_{0}/\kappa }.$ Hence $e_{1}$ and $e_{2}$ span $\mathcal{P}$
and $f_{3}$ spans $\mathcal{L}.$

(iv) $\ker (V)$ is spanned by $e_{1},f_{2},e_{3}.$ Hence $\mathcal{P}_{0}=%
\mathcal{P}\cap \ker (V)$ is spanned by $e_{1}.$

(v) $Ve_{2}=f_{3}^{(p)},\,Vf_{3}=e_{1}^{(p)},\,Vf_{1}=e_{2}^{(p)}.$

(vi) $Ff_{1}^{(p)}=-e_{3},\,Ff_{2}^{(p)}=-e_{1},Fe_{3}^{(p)}=-f_{2}.$
\end{lemma}

\begin{proof}
Up to a slight change of notation, this is the unitary Dieudonn\'{e} module
which B\"{u}ltel and Wedhorn call a ``braid of length 3'' and denote by $%
\bar{B}(3),$ cf [Bu-We] (3.2). The classification in loc. cit. Proposition
3.6 shows that the Dieudonn\'{e} module of a $\mu $-ordinary abelian variety
is isomorphic to $\bar{B}(2)\oplus \bar{S},$ that of a gss abelian variety
is isomorphic to $\bar{B}(3)$ and in the superspecial case we get $\bar{B}%
(1)\oplus \bar{S}^{2}.$
\end{proof}

\subsubsection{Infinitesimal deformations\label{PD deformation}}

Let $\mathcal{O}_{S,s_{0}}$ be the local ring of $S$ at $s_{0},$ $\frak{m}$
its maximal ideal, and $R=\mathcal{O}_{S,s_{0}}/\frak{m}^{2}.$ This $R$ is a
truncated polynomial ring in two variables, isomorphic to $\kappa %
[u,v]/(u^{2},uv,v^{2}).$

As remarked above, the de Rham cohomology $D=H_{dR}^{1}(A/R)$ has a basis of
horizontal sections and we may identify $D^{\nabla }$ with $D_{0}$ and $D$
with $R\otimes _{\kappa }D_{0}.$

Grothendieck tells us that $A/R$ is completely determined by $A_{0}$ and by
the Hodge filtration $\omega _{A/R}\subset D=R\otimes _{\kappa }D_{0}.$
Since $A$ \emph{is} the universal infinitesimal deformation of $A_{0}$, we
may choose the coordinates $u$ and $v$ so that 
\begin{equation}
\mathcal{P}=Span_{R}\{e_{1}+ue_{3},e_{2}+ve_{3}\}.
\end{equation}
The fact that $\omega _{A/R}$ is isotropic implies then that 
\begin{equation}
\mathcal{L}=Span_{R}\{f_{3}-uf_{1}-vf_{2}\}.
\end{equation}

Consider the abelian scheme $A^{(p)}.$ It is \emph{not} the universal
deformation of $A_{0}^{(p)}$ over $R.$ In fact, the map $\phi :R\rightarrow
R $ factors as 
\begin{equation}
R\overset{\pi }{\rightarrow }\kappa \overset{\phi }{\rightarrow }
\kappa \overset{i}{\rightarrow }R,
\end{equation}
and therefore $A^{(p)},$ \emph{unlike} $A,$ is constant: $%
A^{(p)}=Spec(R)\times _{Spec(\kappa )}A_{0}^{(p)}.$ As with $D,$ $%
D^{(p)}=R\otimes _{\Bbb{\kappa }}D_{0}^{(p)},$ $\nabla =d\otimes 1,$ but
this time the basis of horizontal sections can be obtained also from the
trivalization of $A^{(p)},$ and $\omega
_{A^{(p)}/R}=Span_{R}\{e_{1}^{(p)},e_{2}^{(p)},f_{3}^{(p)}\}.$

Since $V$ and $F$ preserve horizontality, $e_{1},f_{2},e_{3}$ span $\ker (V)$
over $R$ in $D,$ and the relations in \emph{(v)} and \emph{(vi)} of Lemma 
\ref{Braid} continue to hold. Indeed, the matrix of $V$ in the basis at $%
s_{0}$ prescribed by that lemma, continues to represent $V$ over $Spec(R)$
by ``horizontal continuation''. The matrix of $F$ is then derived from the
relation (\ref{FVduality}).

The Hodge filtration nevertheless varies, so we conclude that 
\begin{equation}
\mathcal{P}_{0}=\mathcal{P}\cap \ker (V)=Span_{R}\{e_{1}+ue_{3}\}.
\end{equation}
The condition $V(\mathcal{L})=\mathcal{P}_{0}^{(p)},$ which is the
``equation'' of the closed subscheme $Z^{\prime }\cap Spec(R)$ (see Theorem 
\ref{Hasse invariant}) means 
\begin{equation}
V(f_{3}-uf_{1}-vf_{2})=e_{1}^{(p)}-ue_{2}^{(p)}\in R\cdot e_{1}^{(p)}
\end{equation}
and this holds if and only if $u=0.$ We have proved the following lemma.

\begin{lemma}
\label{ssg_equation}Let $s_{0}\in S_{gss}$ and the notation be as above.
Then the closed subscheme $S_{gss}\cap Spec(R)$ is given by the equation $%
u=0.$
\end{lemma}

\subsubsection{The Kodaira-Spencer isomorphism along the general
supersingular locus}

We keep the assumptions of the previous subsections, and compute what the
Gauss-Manin connection does to $\mathcal{P}_{0}.$ A typical element of $%
\mathcal{P}_{0}$ is $g(e_{1}+ue_{3})$ for some $g\in R.$ Then 
\begin{equation}
\nabla (g(e_{1}+ue_{3}))=dg\otimes (e_{1}+ue_{3})+gdu\otimes e_{3}.
\label{Gauss-Manin}
\end{equation}
Note that when we divide by $\omega _{A/R}$ and project $H_{dR}^{1}(A/R)$ to 
$H^{1}(A,\mathcal{O}),$ $e_{1}+ue_{3}$ dies, and the image $\overline{e_{3}}$
of $e_{3}$ becomes a basis for the line bundle that we called $\mathcal{L}%
^{\vee }(\rho )=H^{1}(A,\mathcal{O})(\Sigma )$. Recall the definition of $%
\psi $ given in (\ref{kappa}), but note that this definition only makes
sense over $Spec(\mathcal{O}_{S,s_{0}})$ or its completion, where $KS(\Sigma
)$ is an isomorphism, and can be inverted.

\begin{proposition}
Let $s_{0}\in Z^{\prime }=S_{gss}.$ Choose local parameters $u$ and $v$ at $%
s_{0}$ so that in $\mathcal{O}_{S,s_{0}}$ the local equation of $Z^{\prime }$
becomes $u=0.$ Then at $s_{0},$ $\psi (du)$ has a zero along $Z^{\prime }.$
\end{proposition}

\begin{proof}
Let $i:Z^{\prime }\hookrightarrow S$ be the locally closed embedding. We
must show that in a suitable Zariski neighborhood of $s_{0},$ where $u=0$ is
the local equation of $Z^{\prime },$ $i^{*}\psi (du)=0.$ It is enough to
show that the image of $\psi (du)$ in the \emph{fiber} at every point $s$ of 
$Z^{\prime }$ near $s_{0},$ vanishes. All points being alike, it is enough
to do it at $s_{0}.$ In other words, we denote by $\psi _{0}$ the map 
\begin{equation}
\psi _{0}:\Omega _{S,s_{0}}^{1}\rightarrow \mathcal{P}_{\mu }\otimes 
\mathcal{L}|_{s_{0}}\simeq \mathcal{L}^{p+1}|_{s_{0}}.
\end{equation}
and show that $\psi _{0}(du)=0.$ We may now work over $Spec(R),$ where $R=%
\mathcal{O}_{S,s_{0}}/\frak{m}^{2}$. It is enough to show that in the
diagram 
\begin{equation}
\begin{array}{lll}
\mathcal{P}_{R}\otimes \mathcal{L}_{R} & \overset{KS(\Sigma )}{\rightarrow }
& \Omega _{R}^{1} \\ 
\downarrow &  & \downarrow \\ 
\mathcal{P}_{s_{0}}\otimes \mathcal{L}_{s_{0}} & \simeq & \Omega
_{S,s_{0}}^{1}
\end{array}
\end{equation}
$KS(\Sigma )$ maps the line sub-bundle $\mathcal{P}_{0,R}\otimes \mathcal{L}%
_{R}$ onto $Rdu.$ Once we have passed to the infinitesimal neighborhood $%
Spec(R)$ we can replace the local parameters $u,v$ by any two formal
parameters for which $u=0$ defines $Z^{\prime }\cap Spec(R).$ We may
therefore assume, in view of Lemma \ref{ssg_equation}, that $u$ and $v$ have
been chosen as in Section \ref{PD deformation}. But then (\ref{Gauss-Manin})
shows that the restriction of $KS(\Sigma )$ to $Z^{\prime },$ i.e. the
homomorphism $i^{*}KS(\Sigma ),$ maps $i^{*}\mathcal{P}_{0}$ onto $%
i^{*}R\cdot du\otimes \overline{e_{3}}$. This concludes the proof.
\end{proof}

\subsubsection{A computation of poles along the supersingular locus}

We are now ready to prove the following.

\begin{proposition}
Let $k\ge 0,$ and let $f\in H^{0}(S,\mathcal{L}^{k})$ be a modular form of
weight $k$ in characteristic $p.$ Then $\Theta (f)\in H^{0}(S,\mathcal{L}%
^{k+p+1})$.
\end{proposition}

\begin{proof}
A priori, the definition that we have given for $\Theta (f)$ produces a
meromorphic section of $\mathcal{L}^{k+p+1}$ which is holomorphic on the $%
\mu $-ordinary part $S_{\mu }$ but may have a pole along $Z=S_{ss}.$ Since $%
S $ is a non-singular surface, it is enough to show that $\Theta (f)$ does
not have a pole along $Z^{\prime }=S_{gss},$ the non-singular part of the
divisor $Z$. Consider the degree $p^{2}-1$ covering $\tau :Ig\rightarrow S,$
which is finite, \'{e}tale over $S_{\mu }$ and totally ramified along $Z.$
Let $s_{0}\in Z^{\prime }$ and let $\tilde{s}_{0}\in Ig$ be the closed point
above it. Let $u,v$ be formal parameters at $s_{0}$ for which $Z^{\prime }$
is given by $u=0,$ as in Theorem \ref{Igusa}. As explained there we may
choose formal parameters $w,v$ at $\tilde{s}_{0}$ where $w^{p^{2}-1}=u$ (and 
$v$ is the same function $v$ pulled back from $S$ to $Ig$). It follows that
in $\Omega _{Ig}^{1}$ we have 
\begin{equation}
du=-w^{p^{2}-2}dw.
\end{equation}
We now follow the steps of our construction. Dividing $f$ by $a^{k}$ we get
a function $g=f/a^{k}$ on $Ig$ with a pole of order at most $k$ along $%
\tilde{Z},$ the supersingular divisor on $Ig,$ whose local equation is $w=0.$
In $\widehat{\mathcal{O}}_{Ig,\tilde{s}_{0}}$we may write 
\begin{equation}
g=\sum_{l=-k}^{\infty }g_{l}(v)w^{l}.
\end{equation}
Then 
\begin{eqnarray}
dg &=&\sum_{l=-k}^{\infty }lg_{l}(v)w^{l-1}dw+\sum_{l=-k}^{\infty
}w^{l}g_{l}^{\prime }(v)dv  \notag \\
&=&-\sum_{l=-k}^{\infty }lg_{l}w^{l-(p^{2}-1)}du+\sum_{l=-k}^{\infty
}w^{l}g_{l}^{\prime }(v)dv.
\end{eqnarray}
Applying the map $\psi $ (extended $\mathcal{O}_{Ig}$-linearly from $S$ to $%
Ig$), and noting that $\psi (du)$ has a zero along $Z^{\prime },$ hence a
zero of order $p^{2}-1$ along $\tilde{Z}^{\prime },$ we conclude that $\psi
(dg)$ has a pole of order $k$ (at most) along $\tilde{Z}^{\prime }.$ Finally 
$\Theta (f)=a^{k}\cdot \psi (dg)$ becomes holomorphic along $\tilde{Z}%
^{\prime },$ and also descends to $S$. It is therefore a holomorphic section
of $\mathcal{P}_{\mu }\otimes \mathcal{L}^{k+1}\simeq \mathcal{L}^{k+p+1}.$
\end{proof}

It is amusing to compare the reasons for the increase by $p+1$ in the weight
of $\Theta (f)$ for modular curves and for Picard modular suraces. In the
case of modular curves the Kodaira-Spencer isomorphism is responsible for a
shift by 2 in the weight, but the section acquires simple poles at the
supersingular points. One has to multiply it by the Hasse invariant, which
has weight $p-1,$ to make the section holomorphic, hence a total increase by 
$p+1=2+(p-1)$ in the weight. In our case, the map $\psi $ is responsible for
a shift by $p+1$ (the $p$ coming from $\mathcal{P}_{\mu }\simeq \mathcal{L}%
^{p}$), but the section turns out to be holomorphic along the supersingular
locus. See Section \ref{compatability}.

\section{Further results on $\Theta $}

\subsection{Relation to the filtration and theta cycles\label{theta cycles}}

In part (ii) of Theorem \ref{Main Theorem} we have described the way $\Theta 
$ acts on Fourier-Jacobi expansions at the standard cusp. A similar formula
holds at all the other cusps. We deduce from it that modular forms in the
image of $\Theta $ have vanishing FJ coefficients in degrees divisible by $%
p. $ Moreover, for such a form $f\in \text{Im}(\Theta )$, $\Theta ^{p-1}(f)$
and $f$ have the same FJ expansions, and hence the same filtration. Note
also that if $r(f_{1})=r(f_{2})$ then $r(\Theta (f_{1}))=r(\Theta (f_{2})).$
We may therefore define unambiguously 
\begin{equation}
\Theta (r(f))=r(\Theta (f)).
\end{equation}
As we clearly have 
\begin{equation}
\omega (\Theta (f))=\omega (f)+p+1-a(p^{2}-1)
\end{equation}
for some $a\ge 0$ we deduce the following result.

\begin{proposition}
\label{Prop 4.1}Let $f\in M_{k}(N,\Bbb{\kappa })$ be a modular form modulo $%
p,$ and assume that $r(f)\in \text{Im}(\Theta ).$ Then 
\begin{equation}
r(f)=r(\Theta ^{p-1}(f)).
\end{equation}
There exists a unique index $0\le i\le p-2$ such that 
\begin{equation}
\omega (\Theta ^{i+1}(f))=\omega (\Theta ^{i}(f))+p+1-(p^{2}-1).
\end{equation}
For any other $i$ in this range 
\begin{equation}
\omega (\Theta ^{i+1}(f))=\omega (\Theta ^{i}(f))+p+1.
\end{equation}
\end{proposition}

This is reminiscent of the ``theta cycles'' for classical (i.e. elliptic)
modular forms modulo $p,$ see [Se], [Ka2] and [Joc]. Recall that if $f$ is a 
$\mod p$ modular form of weight $k$ on $\Gamma _{0}(N)$ with $q$%
-expansion $\sum a_{n}q^{n}$ ($a_{n}\in \Bbb{\bar{F}}_{p}$), then $\theta
(f) $ is a $\mod p$ modular form of weight $k+p+1$ with $q$-expansion $%
\sum na_{n}q^{n}$ (Katz denotes $\theta (f)$ by $A\theta (f)$). One has $%
\omega (\theta (f))<\omega (f)+p+1$ \emph{if and only if} $\omega (f)\equiv 0%
\mod p$. In such a case we say that the filtration ``drops'' and we
have 
\begin{equation}
\omega (\theta (f))=\omega (f)+p+1-a(p-1)
\end{equation}
for some $a>0.$ As a corollary, $\omega (f)$ can never equal $1\mod p$
for an $f\in \text{Im}(\theta ).$ Assume now that $f\in \text{Im}(\theta )$
is a ``low point'' in its ``theta cycle'', namely, $\omega (f)$ is minimal
among all $\omega (\theta ^{i}(f)).$ Then $\omega (\theta ^{i+1}(f))<\omega
(\theta ^{i}(f))+p+1$ for \emph{one} or \emph{two} values of $i\in [0,p-2],$
which are completely determined by $\omega (f)\mod p$ [Joc].

This is not true anymore for Picard modular forms. Not only is the drop in
the theta cycle unique, but the question of when exactly it occurs is
mysterious and deserves further study. We make the following elementary
observation showing that whether a drop in the filtration occurs in passing
from $f$ to $\Theta (f)$ can \emph{not} be determined by $\omega (f)$ modulo 
$p$ alone. Let $f$ and $k$ be as in Proposition \ref{Prop 4.1}.

\begin{enumerate}
\item  If $k\le p^{2}-1$ then $\omega (f)=k.$

\item  If $k<p+1$ then $\omega (\Theta ^{i}(f))=k+i(p+1)$ for $0\le i\le p-2,
$ so the drop occurs at the last step of the theta cycle, i.e. at weight $%
k+(p-2)(p+1),$ which is congruent to $k-2$ modulo $p.$

\item  If $k<p+1$ but $r(f)\notin \text{Im}(\Theta )$ then starting with $%
\Theta (f)$ instead of $f,$ one sees that the drop in the theta cycle of $%
\Theta (f)$ occurs either in passing from $\Theta ^{p-2}(f)$ to $\Theta
^{p-1}(f)$, or in passing from $\Theta ^{p-1}(f)$ to $\Theta ^{p}(f).$
\end{enumerate}

\subsection{Compatibility between theta operators for elliptic and Picard
modular forms\label{compatability}}

\subsubsection{The theta operator for elliptic modular forms}

The theta operator for elliptic modular forms modulo $p$ was introduced by
Serre and Swinnerton-Dyer in terms of $q$-expansions, but its geometric
construction was given by Katz in [Ka1] and [Ka2]. Katz relied on a
canonical splitting of the Hodge filtration over the ordinary locus, but
Gross gave in [Gr], Proposition 5.8, the construction after which we modeled
our $\Theta $.

Let us quickly repeat Gross' construction as outlined in the introduction.
Let $X$ be the open modular curve $X(N)$ over $\Bbb{\bar{F}}_{p}$ ($N\ge
3,\,p\nmid N$) and $I_{ord}$ the Igusa curve of level $p$ lying over $%
X_{ord}=X\backslash X_{ss},$ the ordinary part of $X.$ Let $\bar{X}$ and $%
\bar{I}_{ord}$ be the curves obtained by adjoing the cusps to $X$ and $%
I_{ord}$ respectively. Let $\mathcal{L}=\omega _{E/X}$ be the cotangent
bundle of the universal elliptic curve, extended over the cusps as usual.
Classical modular forms of weight $k$ and level $N$ are sections of $%
\mathcal{L}^{k}$ over $\bar{X}.$ Let $a$ be the tautological nowehere
vanishing section of $\mathcal{L}$ over $\bar{I}_{ord}.$ Given a modular
form $f$ of weight $k,$ we consider $r(f)=\tau ^{*}f/a^{k}$ where $\tau :%
\bar{I}_{ord}\rightarrow \bar{X}$ is the covering map, and apply the inverse
of the Kodaira-Spencer isomorphism $KS:\mathcal{L}^{2}\rightarrow \Omega
_{I_{ord}}^{1}$ to get a section $KS^{-1}(dr(f))$ of $\mathcal{L}^{2}$ over $%
\bar{I}_{ord}.$ When multiplied by $a^{k}$ it descends to $\bar{X}_{ord},$
and when this is multiplied further by $h=a^{p-1},$ the Hasse invariant for
elliptic modular forms, it extends holomorphically over $X_{ss}$ to an
element 
\begin{equation}
\theta (f)=a^{k+p-1}KS^{-1}(dr(f))\in H^{0}(\bar{X},\mathcal{L}^{k+p+1}).
\end{equation}

\subsubsection{An embedding of a modular curve in $\bar{S}$}

To illustrate our idea, and to simplify the computations, we assume that $%
N=1 $ and $d_{\mathcal{K}}\equiv 1\mod 4,$ so that $D=D_{\mathcal{K}%
}=d_{\mathcal{K}}.$ This conflicts of course with our running hypothesis $%
N\ge 3,$ but for the current section does not matter much. We shall treat
only one special embedding of the modular curve $\bar{X}=X_{0}(D)$ into $%
\bar{S}$ (there are many more).

Embed $SL_{2}(\Bbb{R})=SU(1,1)$ in $G_{\infty }^{\prime }$ via 
\begin{equation}
\left( 
\begin{array}{ll}
a & b \\ 
c & d
\end{array}
\right) \mapsto \left( 
\begin{array}{lll}
a &  & b \\ 
& 1 &  \\ 
c &  & d
\end{array}
\right) .
\end{equation}
This embedding induces an embedding of symmetric spaces $\frak{H}%
\hookrightarrow \frak{X},$ $z\mapsto \,^{t}(z,0).$ One can easily compute
that the intersection of $\Gamma ,$ the stabilizer of the lattice $L_{0}$ in 
$G_{\infty }^{\prime },$ with $SL_{2}(\Bbb{R}),$ is the subgroup of $SL_{2}(%
\Bbb{Z})$ given by 
\begin{equation}
\Gamma ^{0}(D)=\left\{ \left( 
\begin{array}{ll}
a & b \\ 
c & d
\end{array}
\right) :\,D|b\right\} .
\end{equation}
Let $E_{0}=\Bbb{C}/\mathcal{O}_{\mathcal{K}}$, endowed with the canonical
principal polarization and $CM$ type $\Sigma .$ For $z\in \frak{H}$ let $%
\Lambda _{z}=\Bbb{Z}+\Bbb{Z}z$ and $E_{z}=\Bbb{C}/\Lambda _{z}.$ Let $M_{z}$
be the cyclic subgroup of order $D$ of $E_{z}$ generated by $D^{-1}z\mod%
\Lambda _{z}.$ Using the model (\ref{moving lattice model}) of the abelian
variety $A_{z}$ associated to the point $^{t}(z,0)\in \frak{X},$ we compute
that 
\begin{equation}
A_{z}\simeq E_{0}\times (\mathcal{O}_{\mathcal{K}}\otimes E_{z})/(\delta _{%
\mathcal{K}}\otimes M_{z})
\end{equation}
with the obvious $\mathcal{O}_{\mathcal{K}}$-structure. The group $\delta _{%
\mathcal{K}}\otimes M_{z}$ is a cyclic subgroup of $\mathcal{O}_{\mathcal{K}%
}\otimes E_{z}$ of order $D,$ generated by $\delta _{\mathcal{K}%
}^{-1}\otimes z\mod \mathcal{O}_{\mathcal{K}}\otimes \Lambda _{z}.$ The
principal polarization on $A_{z}$ provided by the complex uniformization is
the product of the canonical polarization of $E_{0}$ and the principal
polarization of $\mathcal{O}_{\mathcal{K}}\otimes E_{z}/\delta _{\mathcal{K}%
}\otimes M_{z}$ obtained by descending the polarization 
\begin{equation}
\lambda _{can}:\mathcal{O}_{\mathcal{K}}\otimes E_{z}\rightarrow \delta _{%
\mathcal{K}}^{-1}\otimes E_{z}=(\mathcal{O}_{\mathcal{K}}\otimes E_{z})^{t}
\end{equation}
of degree $D^{2},$ modulo the maximal isotropic subgroup $\delta _{\mathcal{K%
}}\otimes M_{z}$ of $\ker (\lambda _{can}).$

It is now clear that over any $R_{0}$-algebra $R$ we have the same moduli
theoretic construction, sending a pair $(E,M)$ where $M$ is a cyclic
subgroup of degree $D$ to $A(E,M),$ with $\mathcal{O}_{\mathcal{K}}$
structure and polarization given by the same formulae. This gives a modular
embedding $j:X\rightarrow S$ which is generically injective. To make this
precise at the level of schemes (rather than stacks) one would have to add a
level $N$ structure and replace the base ring $R_{0}$ by $R_{N}.$

\subsubsection{Comparison of the two theta operators}

From now on we work over $\Bbb{\bar{F}}_{p}.$ The modular interpretation of
the embedding $j:\bar{X}\rightarrow \bar{S}$ allows us to complete it to a
diagram 
\begin{equation}
\begin{array}{lll}
\bar{I}_{ord} & \overset{j}{\rightarrow } & \overline{Ig}_{\mu } \\ 
\tau \downarrow &  & \downarrow \tau \\ 
\bar{X}_{ord} & \overset{j}{\rightarrow } & \bar{S}_{\mu }
\end{array}
.
\end{equation}
Note that $j(X_{ss})\subset S_{ssp},$ i.e. the embedded modular curve cuts
the supersingular locus at superspecial points.

\begin{lemma}
The pull-back $j^{*}\mathcal{\omega }_{\mathcal{A}/S}$ decomposes as a
product $\omega _{E_{0}}\times (\mathcal{O}_{\mathcal{K}}\otimes \omega
_{E/X}).$ Under this isomorphism 
\begin{eqnarray}
j^{*}\mathcal{L} &=&(\mathcal{O}_{\mathcal{K}}\otimes \omega _{E/X})(\bar{%
\Sigma}) \\
j^{*}\mathcal{P}_{0} &=&\omega _{E_{0}}  \notag \\
j^{*}\mathcal{P}_{\mu } &\simeq &(\mathcal{O}_{\mathcal{K}}\otimes \omega
_{E/X})(\Sigma ).  \notag
\end{eqnarray}
\end{lemma}

The line bundle $j^{*}\mathcal{P}_{0}$ is constant, and $\mathcal{P}_{\mu },$
originally a quotient bundle of $\mathcal{P},$ becomes a direct summand when
restricted to $\bar{X}$.

\begin{proof}
This is straightforward from the construction of $j$, and the fact that $%
E_{0}$ is supersingular, while $E$ is ordinary over $\bar{X}_{ord}.$ Note
that $\mathcal{O}_{\mathcal{K}}\otimes E/\delta _{\mathcal{K}}\otimes M$ and 
$\mathcal{O}_{\mathcal{K}}\otimes E$ have the same cotangent space.
\end{proof}

\begin{proposition}
Identify $j^{*}\mathcal{L}$ with $\omega _{E/X}$ ($\mathcal{O}_{\mathcal{K}}$
acting via $\bar{\Sigma}$). Then for $f\in H^{0}(\bar{S},\mathcal{L}%
^{k})=M_{k}(N,\Bbb{\bar{F}}_{p})$%
\begin{equation}
\theta (j^{*}(f))=j^{*}(\Theta (f)).
\end{equation}
\end{proposition}

\begin{proof}
We abbreviate $I_{ord}$ by $I$ and $Ig_{\mu }$ by $Ig.$ The pull-back via $j$
of the tautological section $a$ of $\mathcal{L}$ over $Ig$ is the
tautological section $a$ of $j^{*}\mathcal{L=\omega }_{E/X}.$ We therefore
have 
\begin{equation}
j^{*}(dr(f))=dr(j^{*}(f))
\end{equation}
($r(f)=\tau ^{*}f/a^{k}$ is the function on $Ig$ denoted earlier also by $g$%
). It remains to check the commutativity of the following diagram 
\begin{equation}
\begin{array}{lllll}
\Omega _{Ig}^{1} & \overset{KS(\Sigma )^{-1}}{\rightarrow } & \mathcal{P}%
\otimes \mathcal{L} & \overset{V\otimes 1}{\rightarrow } & \mathcal{L}^{p+1}
\\ 
\downarrow j_{0}^{*} &  &  &  & \downarrow j^{*} \\ 
\Omega _{I}^{1} & \overset{KS^{-1}}{\rightarrow } & j^{*}\mathcal{L}^{2} & 
\overset{\times h}{\rightarrow } & j^{*}\mathcal{L}^{p+1}
\end{array}
.
\end{equation}
Here $j_{0}^{*}$ is the map $j^{*}\Omega _{Ig}^{1}\rightarrow \Omega
_{I}^{1} $ on differentials whose kernel is the conormal bundle of $I$ in $%
Ig $. For that we have to compare the Kodaira-Spencer maps on $S$ and on $X.$
As we have seen in the lemma, $\mathcal{P}/\mathcal{P}_{0}=\mathcal{P}_{\mu
} $ pulls back under $j$ to $\mathcal{L}(\rho )$ (the line bundle $\mathcal{L%
}$ with the $\mathcal{O}_{\mathcal{K}}$ action conjugated). But, $KS(\Sigma
)(\mathcal{P}_{0}\otimes \mathcal{L})$ maps under $j^{*}$ to the conormal
bundle, so we obtain a commutative diagram 
\begin{equation}
\begin{array}{lll}
\Omega _{Ig}^{1} & \overset{KS(\Sigma )}{\leftarrow } & \mathcal{P}\otimes 
\mathcal{L} \\ 
\downarrow j_{0}^{*} &  & \downarrow \mod \mathcal{P}_{0} \\ 
\Omega _{I}^{1} & \overset{KS}{\leftarrow } & j^{*}\mathcal{L}(\rho )\otimes
j^{*}\mathcal{L}
\end{array}
.
\end{equation}
The commutativity of the diagram 
\begin{equation}
\begin{array}{lll}
\mathcal{P}_{\mu } & \overset{V}{\rightarrow } & \mathcal{L}^{(p)} \\ 
\downarrow &  & \downarrow \\ 
j^{*}\mathcal{L}(\rho ) & \overset{\times h}{\rightarrow } & j^{*}\mathcal{L}%
^{(p)}
\end{array}
\end{equation}
follows from the definition of the Hasse invariant $h$ on $X.$ Identifying $%
\mathcal{L}^{(p)}$ with $\mathcal{L}^{p}$ as usual and tensoring the last
diagram with $\mathcal{L}$ provides the last piece of the puzzle.
\end{proof}

\begin{remark}
The proposition follows, of course, also from the effect of $\theta $ and $%
\Theta $ on $q$-expansions, once we compare FJ expansions on $\bar{S}$ to $q$%
-expansions on the embedded $\bar{X}.$ The geometric proof given here has
the advantage that it explains the precise way in which $V_{\mathcal{P}%
}\otimes 1$ replaces ``multiplication by $h".$
\end{remark}

\section{The Igusa tower and $p$-adic modular forms}

We shall be very brief, since from now on the development follows closely
the classical case of $p$-adic modular forms on $GL(2)$, with minor
modifications. A general reference for this section is Hida's book [Hi1],
although, strictly speaking, our case ($p$ inert) is excluded there.

\subsection{Geometry modulo $p^{m}$}

\subsubsection{The Picard surface modulo $p^{m}$}

Let $m\ge 1,$ and write $R_{m}=R_{0}/p^{m}R_{0}=\mathcal{O}_{\mathcal{K}%
}/p^{m}\mathcal{O}_{\mathcal{K}}.$ Let 
\begin{equation}
S^{(m)}=S\times _{Spec(R_{0})}Spec(R_{m})
\end{equation}
so that $S^{(1)}=S_{\kappa _{0}}$ is the special fiber, and use a similar
notation for the complete surface $\bar{S}^{(m)}.$ Write $S_{\mu }^{(m)}$
(resp. $\bar{S}_{\mu }^{(m)}$) for the Zariski open subset of points whose
image in $\bar{S}^{(1)}$ lies in $S_{\mu }^{(1)}$ (resp. in $\bar{S}_{\mu
}^{(1)}$).

The generic fiber (in the sense of Raynaud) of the formal scheme 
\begin{equation}
\underset{\rightarrow }{\lim }\,\bar{S}_{\mu }^{(m)}
\end{equation}
is a rigid analytic space which we shall denote by $\bar{S}_{\mu }^{rig}.$
We shall refer to its complement in $\bar{S}^{rig}$ (the rigid analytic
space associated to $\bar{S}$) as the \emph{supersingular tube. }Its\emph{\ }%
$\Bbb{C}_{p}$-points are the points of $\bar{S}(\Bbb{C}_{p})$ whose
reduction modulo $p$ lies in $S_{ss}(\Bbb{\bar{F}}_{p}).$

\subsubsection{$p$-adic modular forms of integral weight $k$}

The vector bundles $\mathcal{P}$ and $\mathcal{L}$ induce vector bundles on $%
\bar{S}^{(m)}$ and $\bar{S}_{\mu }^{rig}$ which we shall denote by the same
symbols (the latter in the rigid analytic category). Let $k\in \Bbb{Z}$ ($k$
may be negative). Let $R$ be a topological $\mathcal{K}_{p}$-algebra. We
define a $p$\emph{-adic modular form of weight} $k$ \emph{and tame level} $N$
\emph{over} $R$ to be an element $f$ of 
\begin{equation}
M_{k}^{p}(N;R):=H^{0}(\bar{S}_{\mu }^{rig}\widehat{\otimes }_{\mathcal{K}%
_{p}}R,\mathcal{L}^{k}).
\end{equation}
Note that $M_{k}^{p}(N;R)=R\widehat{\otimes }_{\mathcal{K}_{p}}M_{k}^{p}(N;%
\mathcal{K}_{p}).$ A $p$-adic modular form $f$ is said to be \emph{%
overconvergent} if there exist finitely many $\mathcal{K}_{p}$-affinoids $%
X_{i}$ contained in the supersingular tube and a section of $\mathcal{L}^{k}$
over $(\bar{S}^{rig}\backslash \bigcup X_{i})\widehat{\otimes }_{\mathcal{K}%
_{p}}R$ which restricts to $f.$ We denote the subspace of overconvergent
modular forms by $M_{k}^{oc}(N;R).$

Note that if $R$ is not of topologically finite type over $\mathcal{K}_{p}$
our definition of ``overconvergent'' is a priori stronger than asking $f$ to
extend to a strict neighborhood of $\bar{S}_{\mu }^{rig}\widehat{\otimes }_{%
\mathcal{K}_{p}}R$ in $\bar{S}^{rig}\widehat{\otimes }_{\mathcal{K}_{p}}R.$

The space $M_{k}^{p}(N;\mathcal{K}_{p})$ is a $p$-adic Banach space whose
unit ball is given by 
\begin{equation}
M_{k}^{p}(N;\mathcal{O}_{p})=\underset{\leftarrow }{\,\lim \,}H^{0}(\bar{S}%
_{\mu }^{(m)},\mathcal{L}^{k}).
\end{equation}

\subsubsection{$q$-expansion principle}

Whether we are dealing with an $f\in H^{0}(\bar{S}_{\mu }^{(m)},\mathcal{L}%
^{k})$ or an $f\in M_{k}^{p}(N;\mathcal{K}_{p})$ the same procedure as in
Section \ref{AFJ} allows us to associate to $f$ a Fourier-Jacobi expansion $%
FJ(f)$ (\ref{FJ}). Recall however that $FJ(f)$ depends on the section $s\in
H^{0}(C,\mathcal{L})$ used to trivialize $\mathcal{L}|_{C}.$ Note that if $%
f\in M_{k}^{p}(N;\mathcal{K}_{p}),$ the coefficients of $FJ(f)$ are theta
functions with \emph{bounded denominators}, since a suitable $\mathcal{K}%
_{p} $-multiple of $f$ lies in $M_{k}^{p}(N;\mathcal{O}_{p}).$

As with classical modular forms, we have the $q$-expansion principle,
stemming from the fact that $C$ meets every component of $\bar{S}_{\mu
}^{rig}.$

\begin{lemma}
If $FJ(f)=0$ then $f=0.$
\end{lemma}

\begin{corollary}
If $f\in M_{k}^{p}(N;\mathcal{O}_{p})$ and $FJ(f)$ is divisible by $p$ (in
the sense that every $c_{j}(f)\in H^{0}(C,\mathcal{N}^{j})$ is divisible by $%
p$ with respect to the integral structure on $\bar{S}$), then $f\in
pM_{k}^{p}(N;\mathcal{O}_{p}).$
\end{corollary}

\subsection{The Igusa scheme of level $p^{n}$}

\subsubsection{$\mu $-ordinary abelian schemes over $R_{m}$-algebras}

Let $m\ge 1$ and let $R$ be an $R_{m}$-algebra. If $\underline{A}\in S_{\mu
}^{(m)}(R)\subset \mathcal{M}(R)$ then $A$ is fiber-by-fiber $\mu $%
-ordinary, hence $A[p^{n}]^{\mu },$ the largest $R$-subgroup scheme of $%
A[p^{n}]$ of multiplicative type (dual to the \'{e}tale quotient $%
A[p^{n}]^{et}$), is a finite flat $\mathcal{O}_{\mathcal{K}}$-subgroup
scheme of rank $p^{2n}.$ Locally in the \'{e}tale topology it is isomorphic
to $\delta _{\mathcal{K}}^{-1}\mathcal{O}_{\mathcal{K}}\otimes \mu _{p^{n}}.$

\subsubsection{Igusa level structure of level $p^{n}$}

Fix $m\ge 1$ and $n\ge 1$ and consider the moduli problem associating to an $%
R_{m}$-algebra $R$ $\mu $-ordinary tuples $\underline{A}\in S_{\mu
}^{(m)}(R) $ together with an isomorphism of finite flat group schemes over $%
R$%
\begin{equation}
\varepsilon =\varepsilon _{n}^{(m)}:\delta _{\mathcal{K}}^{-1}\mathcal{O}_{%
\mathcal{K}}\otimes \mu _{p^{n}}\simeq A[p^{n}]^{\mu }.
\end{equation}
This moduli problem is representable by a scheme $Ig(p^{n})_{\mu }^{(m)},$
and the map ``forget $\varepsilon "$ is a finite \'{e}tale cover 
\begin{equation}
\tau =\tau _{n}^{(m)}:Ig(p^{n})_{\mu }^{(m)}\rightarrow S_{\mu }^{(m)}
\end{equation}
of degree $(p^{2}-1)p^{2(n-1)}.$ It extends to a finite \'{e}tale cover $%
\overline{Ig}(p^{n})_{\mu }^{(m)}$ of $\bar{S}_{\mu }^{(m)}.$ The group 
\begin{equation}
\Delta (p^{n})=(\mathcal{O}_{\mathcal{K}}/p^{n}\mathcal{O}_{\mathcal{K}%
})^{\times }=Aut_{\mathcal{O}_{\mathcal{K}}}(\delta _{\mathcal{K}}^{-1}%
\mathcal{O}_{\mathcal{K}}\otimes \mu _{p^{n}})
\end{equation}
acts on the covering $\tau $ as a group of deck transformations via 
\begin{equation}
\gamma (\underline{A},\varepsilon )=(\underline{A},\varepsilon \circ \gamma
^{-1}),
\end{equation}
and the pre-image of the cuspidal divisor $C$ is non-canonically isomorphic
to $\Delta (p^{n})\times C.$ These constructions satisfy the usual
compatibilities in $m$ and $n.$

\subsubsection{The trivialization of $\mathcal{L}$ when $m\le n$}

Assume now that $m\le n.$ In this case, multiplication by $p^{n}$ is $0$ on $%
R,$ so the inclusion of $A[p^{n}]$ in $A$ induces an isomorphism between the
cotangent spaces at the origin $\omega _{A[p^{n}]/R}$ and $\omega _{A/R}.$
To see it note that if $\mathcal{G}$ is either $A[p^{n}]$ or $A$ its Lie
algebra, by definition, is the finite flat $R$-module 
\begin{equation}
Lie(\mathcal{G})=\ker \left( \mathcal{G}(R[\epsilon ])\rightarrow \mathcal{G}%
(R)\right) .
\end{equation}
Here $R[\epsilon ]$ is the ring of dual numbers over $R$. It follows that 
\begin{equation}
Lie(A[p^{n}])=Lie(A)[p^{n}]=Lie(A),
\end{equation}
and dualizing we get $\omega _{A/R}=\omega _{A[p^{n}]/R}.$

The same holds of course for $\mu _{p^{n}}$ and $\Bbb{G}_{m}.$ The reasoning
used for $m=n=1$ applies and shows that $\varepsilon $ induces a canonical
isomorphism between $\mathcal{L}|_{\overline{Ig}(p^{n})_{\mu }^{(m)}}$ and $%
\mathcal{O}_{\overline{Ig}(p^{n})_{\mu }^{(m)}}.$ We denote by $%
a=a_{n}^{(m)} $ the section which corresponds to $1\in \mathcal{O}_{%
\overline{Ig}(p^{n})_{\mu }^{(m)}},$ i.e. the trivializing section.

The group $\Delta (p^{n})$ acts on $a$ via the character 
\begin{equation}
\bar{\Sigma}^{-1}:\Delta (p^{n})=(\mathcal{O}_{\mathcal{K}}/p^{n}\mathcal{O}%
_{\mathcal{K}})^{\times }\rightarrow (\mathcal{O}_{\mathcal{K}}/p^{m}%
\mathcal{O}_{\mathcal{K}})^{\times }=R_{m}^{\times }.
\end{equation}

From now on we take $n=m$ and use $a$ to trivialize $\mathcal{L}$ along $%
\tilde{C}=\tau ^{-1}(C),$ the cuspidal divisor in $\overline{Ig}(p^{m})_{\mu
}^{(m)}.$ If $f\in H^{0}(\bar{S}_{\mu }^{(m)},\mathcal{L}^{k})$ then $\tau
^{*}f/a^{k}$ is a function on $\overline{Ig}(p^{m})_{\mu }^{(m)}$ and we may
attach to it a \emph{canonical }FJ expansion 
\begin{equation}
\widetilde{FJ}(f)=\sum_{j=0}^{\infty }c_{j}(f)
\end{equation}
where $c_{j}(f)\in H^{0}(\tilde{C},\mathcal{N}^{j})$ as before. This FJ
expansion does not depend on any choice (but is defined along $\tilde{C}$
and not along $C$).

\subsubsection{Congruences between FJ expansions force congruences between
the weights}

Let $k_{1}\le k_{2}$ be two integers. The following lemma follows formally
from the definitions.

\begin{lemma}
Let $f_{i}\in H^{0}(\bar{S}_{\mu }^{(m)},\mathcal{L}^{k_{i}})$ and assume
that $f_{1}$ is not divisible by $p.$ Suppose $\widetilde{FJ}(f_{1})=%
\widetilde{FJ}(f_{2}).$ Then $k_{1}\equiv k_{2}\mod (p^{2}-1)p^{m-1}.$
\end{lemma}

\begin{proof}
Let $\tilde{T}$ be an irreducible component of $\overline{Ig}(p^{m})_{\mu
}^{(m)}.$ Then, $\tau $ being finite \'{e}tale, $\tau (\tilde{T})$ is both
open and closed in $\bar{S}_{\mu }^{(m)}$, so must be an irreducible
component $T$ of $\bar{S}_{\mu }^{(m)}$. It follows that $\tau (\tilde{T})$
meets $C,$ hence $\tilde{T}$ meets $\tilde{C},$ and the $q$-expansion
principle holds in $\overline{Ig}(p^{m})_{\mu }^{(m)}.$ We therefore have an
equality 
\begin{equation}
\tau ^{*}f_{1}/a^{k_{1}}=\tau ^{*}f_{2}/a^{k_{2}}
\end{equation}
between functions on $\overline{Ig}(p^{m})_{\mu }^{(m)}.$ Since the left
hand side is not divisible by $p$ by assumption, so is the right hand side.
The group $\Delta (p^{m})$ acts on the left hand side via $\bar{\Sigma}%
^{k_{1}}$ and on the right hand side via $\bar{\Sigma}^{k_{2}}.$ But these
two characters are equal if and only if $k_{1}\equiv k_{2}\mod %
(p^{2}-1)p^{m-1},$ because the exponent of the group $\Delta (p^{m})$ is $%
(p^{2}-1)p^{m-1}.$
\end{proof}

In practice, one would like to deduce the same result from congruences
between FJ expansions along $C,$ not along $\tilde{C}.$ This is deeper and
depends on Igusa's irreducibility theorem.

\begin{theorem}
Consider $\tau =\tau _{n}^{(1)}:\overline{Ig}(p^{n})_{\mu }^{(1)}\rightarrow 
\bar{S}_{\mu }^{(1)}=\bar{S}_{\mu ,\kappa _{0}}$ and extend scalars from $%
\kappa _{0}$ to $\kappa .$ Let $T$ be an irreducible component of $\bar{S}%
_{\mu ,\kappa }.$ Then $\tau ^{-1}(T)$ is irreducible in $\overline{Ig}%
(p^{n})_{\mu ,\kappa }.$
\end{theorem}

\begin{proof}
The theorem can be proved by the same method used by Hida in [Hi1, 8.4],
[Hi2], or by the method of Ribet to which we alluded in \ref{Ig
irreducibility}. In that section, we proved the theorem for $n=1$ by a third
method, due to Igusa, studying the image of inertia around $S_{ss}.$ See
also the discussion of the big Igusa tower $BigIg$ below, which turns out to
be reducible.
\end{proof}

\begin{theorem}
Let $f_{1}$ and $f_{2}$ be $\mod p^{m}$ modular forms as above, and
assume that $f_{1}$ is not divisible by $p$. Trivialize $\mathcal{L}|_{C}$
by choosing a lift of $C$ to $\tilde{C}$ (i.e. a section of the map $\tau |_{%
\tilde{C}}:\tilde{C}\rightarrow C$) and using the trivialization of $%
\mathcal{L}$ along this lift which is supplied by the section $a$. Then if $%
FJ(f_{1})=FJ(f_{2}),$ $k_{1}\equiv k_{2}\mod (p^{2}-1)p^{m-1}.$
\end{theorem}

Here $FJ(f)=\sum_{j=0}^{\infty }c_{j}(f)$ and $c_{j}(f)\in H^{0}(C,\mathcal{N%
}^{j})$. The lift of $C$ to $\tilde{C}$ exists since $\tilde{C}\simeq \Delta
(p^{m})\times C$ (non-canonically). If we change the lift (locally on the
base) by $\gamma \in \Delta (p^{m}),$ then $FJ(f_{i})$ changes by the factor 
$\bar{\Sigma}(\gamma )^{k_{i}}.$

\begin{proof}
By Igusa's irreducibility theorem, it is enough to know that $FJ(f_{i})$ ($%
i=1,2$) agree on the given lift of $C,$ to conclude that $\tau
^{*}f_{1}/a^{k_{1}}=\tau ^{*}f_{2}/a^{k_{2}}$ on the whole of $\overline{Ig}%
(p^{m})_{\mu }^{(m)},$ hence the result follows by the Lemma. Note that the
underlying topological spaces of $\overline{Ig}(p^{m})_{\mu }^{(m)}$ and $%
\overline{Ig}(p^{m})_{\mu }^{(1)}$ are the same, hence for the irreducibilty
theorem it is enough to deal with the special fiber.
\end{proof}

\begin{corollary}
Let $f_{i}\in M_{k_{i}}^{p}(N;\mathcal{O}_{p})$ ($i=1,2$) and assume that $%
f_{1}$ is not divisible by $p.$ Trivialize $\mathcal{L}|_{C}$ by fixing an $%
\mathcal{O}_{\mathcal{K}}$-isomorphism of the $p$-divisible group of the
toric part of the universal semi-abelian variety $\mathcal{A}|_{C}$ with $%
\delta _{\mathcal{K}}^{-1}\mathcal{O}_{\mathcal{K}}\otimes \mu _{p^{\infty
}},$ and using this isomorphism to identify $\mathcal{L}|_{C}=\omega _{%
\mathcal{A}/C}(\bar{\Sigma})$ with $\mathcal{O}_{C}.$ Suppose that with 
\emph{this} trivialization 
\begin{equation}
FJ(f_{1})\equiv FJ(f_{2})\mod p^{m}.
\end{equation}
Then $k_{1}\equiv k_{2}\mod (p^{2}-1)p^{m-1}.$
\end{corollary}

\subsubsection{Irreducibility of the Igusa tower and the big Igusa tower%
\label{Irreducibility again}}

It is possible to define an even larger Igusa tower $(BigIg(p^{n}))_{n\ge 1}$
over $\kappa =\Bbb{\bar{F}}_{p},$ of which $(Ig(p^{n}))_{n\ge 1}$ is a
quotient. If $R$ is a $\kappa $-algebra and $\underline{A}\in S_{\mu }(R)$,
then $A[p^{n}]$ admits a filtration as in \ref{p-div}. One can define $%
BigIg(p^{n})$ as the moduli space of $\mu $-ordinary tuples $\underline{A},$
equipped with $\mathcal{O}_{\mathcal{K}}$-isomorphisms 
\begin{eqnarray}
\varepsilon ^{2} &:&\delta _{\mathcal{K}}^{-1}\mathcal{O}_{K}\otimes \mu
_{p^{n}}\simeq gr^{2}A[p^{n}]  \notag \\
\varepsilon ^{1} &:&\frak{G}[p^{n}]\simeq gr^{1}A[p^{n}]  \notag \\
\varepsilon ^{0} &:&\mathcal{O}_{\mathcal{K}}\otimes \Bbb{Z}/p^{n}\Bbb{Z}%
\simeq gr^{0}A[p^{n}].
\end{eqnarray}
This would be, in the language of [Hi2], the $GU$-Igusa tower. If we insist
that the isomorphisms respect the pairings induced on these group schemes by
the polarization and Cartier duality ($gr^{0}$ and $gr^{2}$ are dual to each
other, $gr^{1}$ is self-dual), we would get the $U$-Igusa tower. Both these
towers are reducible, by the reasoning of [Hi1, 8.4.1] or [Hi2], and by the
description of the connected components of the characteristic 0 fiber of the
Shimura variety given in \ref{connected components}. The $SU$-Igusa tower,
which is \emph{irreducible}, turns out to be our tower $(Ig(p^{n})).$ It is
also the quotient of $(BigIg(p^{n}))$ under the map ``forget $\varepsilon
^{0}$ and $\varepsilon ^{1}".$ Thus there is no real advantage in studying
the tower $BigIg.$

\subsection{$p$-adic modular forms of $p$-adic weights}

\subsubsection{The space of $p$-adic weights}

Let 
\begin{equation}
\frak{X}_{p}=\underset{\leftarrow }{\,\lim }\,\Bbb{Z}/(p^{2}-1)p^{m-1}\Bbb{Z}%
,
\end{equation}
This is the space of $p$-adic weights. If $k\in \frak{X}_{p}$ then $\bar{%
\Sigma}^{k}$ is a well-defined locally $\Bbb{Q}_{p}$-analytic homomorphism
of $\mathcal{O}_{p}^{\times }$ to itself, but note that not every such
homomorphism is a $\bar{\Sigma}^{k}$ for some $k$ from $\frak{X}_{p}.$

\subsubsection{$p$-adic modular forms \`{a} la Serre}

We work with $\bar{S}$ (hence also the cuspidal divisor $C$) over the base $%
\mathcal{O}_{p},$ the $p$-adic completion of $R_{0}.$ Little is lost by
extending the base further to $\mathcal{O}_{N,\frak{P}}$, the completion of
the ring of integers of the ray class field $\mathcal{K}_{N}$ at a prime $%
\frak{P}$ above $p$. After such a base extension the irreducible components
of $C$ become absolutely irreducible. The reader may assume that this is the
case.

Consider the $p$-divisible group of the toric part of the universal
semi-abelian variety $\mathcal{A}|_{C}$. Once and for all fix an $\mathcal{O}%
_{\mathcal{K}}$-isomorphism of it with $\delta _{\mathcal{K}}^{-1}\mathcal{O}%
_{\mathcal{K}}\otimes \mu _{p^{\infty }},$ and use this isomorphism to
identify $\mathcal{L}|_{C}=\omega _{\mathcal{A}/C}(\bar{\Sigma})$ with $%
\mathcal{O}_{C}.$ This choice is unique up to multiplication by $\mathcal{O}%
_{p}^{\times }$ on each irreducible component of $C$. It determines a FJ
expansion $FJ(f)$ for every $f\in M_{k}^{p}(N;\mathcal{K}_{p})$ as in (\ref
{FJ}), and is equivalent to splitting the projection $\tau |_{\tilde{C}}:%
\tilde{C}\rightarrow C$ from the boundary of the Igusa tower $\left( 
\overline{Ig}_{\mu }(p^{n})\right) _{n=1}^{\infty }$ to the boundary of the
Picard modular surface.

Let $k\in \frak{X}_{p}.$ The space $M_{k}^{Serre}(N;\mathcal{K}_{p})$ will
be a subspace of the Banach algebra 
\begin{equation}
\mathcal{FJ}_{p}=\mathcal{K}_{p}\otimes _{\mathcal{O}_{p}}\prod_{j=0}^{%
\infty }H^{0}(C,\mathcal{N}^{j}).
\end{equation}
It will consist of all the $f\in \mathcal{FJ}_{p}$ for which there exists a
sequence $(f_{\nu }),$ $f_{\nu }\in M_{k_{\nu }}^{p}(N;\mathcal{K}_{p})$, $%
(k_{\nu }\in \Bbb{Z}),$ with $FJ(f_{\nu })$ converging to $f$, and $k_{\nu }$
converging in $\frak{X}_{p}$ to $k.$ As we have seen, if the sequence $%
(FJ(f_{\nu }))$ converges, the $k_{\nu }$ have to converge in $\frak{X}_{p}.$
We shall denote by $M_{k}^{Serre}(N;\mathcal{O}_{p})$ the intersection of $%
M_{k}^{Serre}(N;\mathcal{K}_{p})$ with $\prod_{j=0}^{\infty }H^{0}(C,%
\mathcal{N}^{j}).$

\begin{proposition}
(i) If $k\in \Bbb{Z}$ then $M_{k}^{Serre}(N;\mathcal{K}_{p})=M_{k}^{p}(N;%
\mathcal{K}_{p}).$ In other words, we do not get any new $p$-adic modular
forms by allowing limits of $p$-adic modular forms of varying weights, if
the weights converge to an integral $k.$

(ii) In the definition of $M_{k}^{Serre}(N;\mathcal{K}_{p})$ we can require $%
f_{\nu }\in M_{k_{\nu }}(N;\mathcal{K}_{p})$ (classical modular forms of
integral weight $k_{\nu }$) and still get the same space.

(iii) $M_{k}^{Serre}(N;\mathcal{K}_{p})$ is a closed subspace of $\mathcal{FJ%
}_{p}.$ The product of two $f_{i}\in M_{k_{i}}^{Serre}$ is in $%
M_{k_{1}+k_{2}}^{Serre}.$

(iv) If $f\in M_{k}^{Serre}(N;\mathcal{O}_{p})$ then its reduction modulo $p$
appears in $M_{k^{\prime }}(N;\kappa _{0})$ for some positive integer $%
k^{\prime }$ sufficiently close to $k$ in $\frak{X}_{p}.$
\end{proposition}

\begin{proof}
Let $H_{\bar{\Sigma}}\in M_{p^{2}-1}(N;\mathcal{O}_{p})$ be a lift of the
Hasse invariant $h_{\bar{\Sigma}}$ to characteristic 0. Such a lift exists
by general principles, whenever $p$ is large enough. For the few exceptional
primes $p$ we may replace $h_{\vec{\Sigma}}$ by a high enough power of it,
which is liftable, and use the same argument. This lift satisfies $FJ(H_{%
\bar{\Sigma}})\equiv 1\mod p,$ so $H_{\bar{\Sigma}}^{-1}\in
M_{1-p^{2}}^{p}(N;\mathcal{O}_{p})$ is a $p$-adic modular form defined over $%
\mathcal{O}_{p}.$ Indeed, $H_{\bar{\Sigma}}\mod p^{m}\in H^{0}(\bar{S}%
_{\mu }^{(m)},\mathcal{L}^{p^{2}-1})$ is nowhere vanishing over $\bar{S}%
_{\mu }^{(m)},$ and taking the limit of its inverse over $m$ we get $H_{\bar{%
\Sigma}}^{-1}.$ Suppose, as in (i), that $k,k_{\nu }\in \Bbb{Z}$, $k_{\nu
}\rightarrow k$ in $\frak{X}_{p},$ and $f_{\nu }\in M_{k_{\nu }}^{p}(N;%
\mathcal{K}_{p})$ are such that $FJ(f_{\nu })$ converge in $\mathcal{FJ}_{p}$
to $f.$ Replacing $f_{\nu }$ by $f_{\nu }H_{\bar{\Sigma}}^{p^{e_{\nu }}}$
for suitable $e_{\nu }$ we may assume that the $k_{\nu }$ are increasing and
are all in the same congruence class modulo $p^{2}-1.$ But then $f_{\nu }H_{%
\bar{\Sigma}}^{(k-k_{\nu })/(p^{2}-1)}$ are in $M_{k}^{p}(N;\mathcal{K}_{p})$
and their FJ expansions converge to $f$ in $\mathcal{FJ}_{p}.$ This proves
(i). For (ii) note that if $f\in H^{0}(\bar{S}_{\mu }^{(m)},\mathcal{L}^{k})$
then for all sufficiently large $e$, $fH_{\bar{\Sigma}}^{p^{e}}$ extends to
an element of $M_{k+(p^{2}-1)p^{e}}(N;R_{m})$ and has the same FJ expansion
as $f.$ Thus every $p$-adic modular form of integral weight is the $p$-adic
limit of classical forms of varying weights, and the same is therefore true
for Serre modular forms of $p$-adic weight. Points (iii) and (iv) are
obvious.
\end{proof}

\subsubsection{$p$-adic modular forms \`{a} la Katz}

We now explain Katz' point of view of the same objects. Let 
\begin{equation}
V_{n}^{(m)}=H^{0}(\overline{Ig}(p^{n})_{\mu }^{(m)},\mathcal{O})
\end{equation}
be the ring of regular functions on $\overline{Ig}(p^{n})_{\mu }^{(m)}.$ Let 
\begin{equation}
V^{(m)}=\text{\thinspace }\underset{\rightarrow }{\lim }\,V_{n}^{(m)},\,\,\,%
\,\,\,\,V=\,\underset{\leftarrow }{\lim }\,V^{(m)}.
\end{equation}
We call $V$ the space of \emph{Katz} $p$\emph{-adic modular forms} (of all
weights). Let 
\begin{equation}
\gamma \in \Delta =\mathcal{O}_{p}^{\times }=\,\underset{\leftarrow }{\lim }%
\,\Delta (p^{n})
\end{equation}
act on $V^{(m)}$ and on $V$ as usual, $\gamma (f)=f\circ \gamma ^{-1}$, and
recall that $\gamma ^{-1}(\underline{A},\varepsilon _{n}^{(m)})=(\underline{A%
},\varepsilon _{n}^{(m)}\circ \gamma ).$ Thus 
\begin{equation}
\gamma (f)(\underline{A},\varepsilon )=f(\underline{A},\varepsilon \circ
\gamma )
\end{equation}
(i.e. $\gamma $ acts by ``right translation''). Let $k\in \frak{X}_{p}$ and
define 
\begin{equation}
M_{k}^{Katz}(N;\mathcal{O}_{p})=V(\bar{\Sigma}^{k})=\left\{ f\in V|\,\gamma
(f)=\bar{\Sigma}^{k}(\gamma )\cdot f\,\,\,\,\,\forall \gamma \in \Delta
\right\} .
\end{equation}
We similarly define $M_{k}^{Katz}(N;R_{m})=V^{(m)}(\bar{\Sigma}^{k}).$

By the irreducibility of the Igusa tower and the $q$-expansion principle the
FJ expansion map 
\begin{equation}
V\rightarrow \mathcal{FJ}_{p}(\mathcal{O}_{p})
\end{equation}
is injective. It depends on our choice of the splitting of $\tilde{C}%
\rightarrow C$.

\begin{proposition}
\label{Katz is Serre}For $k\in \frak{X}_{p},$ there is a natural isomorphism 
\begin{equation}
M_{k}^{Serre}(N;\mathcal{O}_{p})\simeq M_{k}^{Katz}(N;\mathcal{O}_{p}).
\end{equation}
\end{proposition}

\begin{proof}
Given $k\in \Bbb{Z}$ and $f\in H^{0}(\bar{S}_{\mu }^{(m)},\mathcal{L}^{k}),$
the functions $(\tau _{n}^{(m)})^{*}f/(a_{n}^{(m)})^{k}\in V_{n}^{(m)}$ for
all $n\ge m,$ and these functions satisfy the obvious compatibility in $n,$
so they define 
\begin{equation}
f^{Katz}\in V^{(m)}(\bar{\Sigma}^{k}).
\end{equation}
If $k\in \Bbb{Z},$ this gives, by going to the inverse limit over $m,$ a map 
\begin{equation}
f\mapsto f^{Katz},\,\,\,\,\,\,\,M_{k}^{p}(N;\mathcal{O}_{p})\rightarrow
M_{k}^{Katz}(N;\mathcal{O}_{p}).
\end{equation}
This map is an isomorphism, which can be enhanced to include $p$-adic
weights $k\in \frak{X}_{p}$ as follows. If $k_{\nu }\in \Bbb{Z}$, $k_{\nu
}\rightarrow k\in \frak{X}_{p}$ and if $f_{\nu }\in M_{k_{\nu }}^{p}(N;%
\mathcal{O}_{p})$ are such that $FJ(f_{\nu })$ converge to $f\in
M_{k}^{Serre}(N;\mathcal{O}_{p}),$ then reducing modulo $p^{m}$ for a fixed $%
m,$ $(f_{\nu }^{(m)})^{Katz}\in V^{(m)}(\bar{\Sigma}^{k_{\nu }}).$ But for a
fixed $m,$ for all large enough $\nu ,$%
\begin{equation}
V^{(m)}(\bar{\Sigma}^{k_{\nu }})=V^{(m)}(\bar{\Sigma}^{k}),
\end{equation}
and the sequence $FJ(f_{\nu }^{(m)})$ stabilizes, so taking the limit over $%
\nu $ we get a well defined $(f^{(m)})^{Katz}\in V^{(m)}(\bar{\Sigma}^{k}).$
Finally, an inverse limit over $m$ gives $f^{Katz}\in M_{k}^{Katz}(N;%
\mathcal{O}_{p}).$ It is by now standard that this gives an isomorphism
between $M_{k}^{Serre}(N;\mathcal{O}_{p})$ and $M_{k}^{Katz}(N;\mathcal{O}%
_{p}).$ As we have seen earlier, when $k\in \Bbb{Z},$ this is also the same
as $M_{k}^{p}(N;\mathcal{O}_{p}).$
\end{proof}

From now on it is therefore legitimate to denote these spaces by the common
notation $M_{k}^{p}(N;\mathcal{O}_{p})$ and refer to them simply as $p$\emph{%
-adic modular forms} of $p$\emph{-adic weight} $k.$

\subsection{$p$-adic modular forms of $p$-adic bi-weights}

\subsubsection{The space of bi-weights}

A new feature of $p$-adic modular forms on Picard modular surfaces, that
does not show up in the classical theory of $GL_{2}(\Bbb{Q}),$ is that even
if we restrict attention to scalar-valued $p$-adic modular forms, we
sometimes need to consider classical \emph{vector-valued} forms to approach
them. This phenomenon, as we shall explain below, does not show up in the $%
\mod p$ theory, but is essential to the $p$-adic theory.

The space $\frak{X}_{p}$ of $p$-adic weights can be written as $\Bbb{Z}%
/(p^{2}-1)\Bbb{Z\times Z}_{p},$ and when we decompose it in such a way we
write 
\begin{equation}
k=(w,j)=(\omega (k),\left\langle k\right\rangle )
\end{equation}
for the two components. The space of \emph{bi-weights} $\frak{X}_{p}^{(2)}$
is, by definition, the quotient of $\frak{X}_{p}^{2}$ modulo the relation 
\begin{equation}
((w_{1},j_{1}),(w_{2},j_{2}))\equiv ((0,j_{1}),(pw_{1}+w_{2},j_{2}))\equiv
((pw_{2}+w_{1},j_{1}),(0,j_{2})).
\end{equation}

If $k_{1}$ and $k_{2}$ are in $\frak{X}_{p}$, then the character $\bar{\Sigma%
}^{k_{1}}\Sigma ^{k_{2}}:\Delta \rightarrow \mathcal{O}_{p}^{\times }$
depends only on the image of $(k_{1},k_{2})$ in $\frak{X}_{p}^{(2)}.$ Here $%
\Delta =\,\underset{\leftarrow }{\lim }\,\Delta (p^{n})$ is also $\mathcal{O}%
_{p}^{\times },$ but in the r\^{o}le of the Galois group of the Igusa tower.
The image of $\Bbb{Z}^{2}$ is dense in $\frak{X}_{p}^{2},$ hence also in $%
\frak{X}_{p}^{(2)}.$

\subsubsection{The line bundle $\mathcal{L}^{(k_{1},k_{2})}$ over $\bar{S}%
_{\mu }^{rig}$ and $p$-adic modular forms of integral bi-weights}

Let $m\ge 1.$ The plane bundle $\mathcal{P}$ admits a canonical filtration 
\begin{equation}
0\rightarrow \mathcal{P}_{0}\rightarrow \mathcal{P}\rightarrow \mathcal{P}%
_{\mu }\rightarrow 0
\end{equation}
over $\bar{S}_{\mu }^{(m)}$ defined by choosing any $n\ge m$ and setting 
\begin{equation*}
\mathcal{P}_{0}=\ker (\omega _{\mathcal{A}[p^{n}]^{0}}\rightarrow \omega _{%
\mathcal{A}[p^{n}]^{\mu }}),\,\,\mathcal{P}_{\mu }=\omega _{\mathcal{A}%
[p^{n}]^{\mu }}(\Sigma )
\end{equation*}
(recall $\omega _{\mathcal{A}}=\omega _{\mathcal{A}[p^{n}]^{0}}$). We also
recall that $\mathcal{L}=\omega _{\mathcal{A}[p^{n}]^{\mu }}(\bar{\Sigma}).$

If $m=1$ we showed that over $\bar{S}_{\mu }^{(1)},$ $\mathcal{L}\simeq 
\mathcal{P}_{\mu }^{p}$ and $\mathcal{P}_{\mu }\simeq \mathcal{L}^{p}.$ This
is no longer true for general $m$ and we let for $(k_{1},k_{2})\in \Bbb{Z}%
^{2}$%
\begin{equation}
\mathcal{L}^{(k_{1},k_{2})}=\mathcal{L}^{k_{1}}\otimes \mathcal{P}_{\mu
}^{k_{2}}.
\end{equation}
Going to the limit over $m,$ this defines a rigid analytic line bundle over $%
\bar{S}_{\mu }^{rig}.$

We define the space of $p$\emph{-adic modular forms of bi-weight} $%
(k_{1},k_{2})$ \emph{and level }$N$ over $\mathcal{O}_{p}$ as 
\begin{equation}
M_{k_{1},k_{2}}^{p}(N;\mathcal{O}_{p})=\,\underset{\leftarrow }{\lim }%
\,H^{0}(\bar{S}_{\mu }^{(m)},\,\mathcal{L}^{(k_{1},k_{2})}).
\end{equation}
This is the unit ball of the $p$-adic Banach space 
\begin{equation}
M_{k_{1},k_{2}}^{p}(N;\mathcal{K}_{p})=\mathcal{K}_{p}\otimes _{\mathcal{O}%
_{p}}M_{k_{1},k_{2}}^{p}(N;\mathcal{O}_{p})=H^{0}(\bar{S}_{\mu }^{rig},%
\mathcal{L}^{(k_{1},k_{2})}).
\end{equation}

\subsubsection{The trivialization of $\mathcal{L}^{(k_{1},k_{2})}$ over the
Igusa tower}

As before, fix $m,$ let $m\le n$ and consider the isomorphism 
\begin{equation}
\varepsilon ^{*}:\tau ^{*}\omega _{\mathcal{A}[p^{n}]^{\mu }}\simeq \mathcal{%
O}_{\mathcal{K}}\otimes \mathcal{O}_{\overline{Ig}(p^{n})_{\mu }^{(m)}}
\end{equation}
induced by the Igusa level structure $\varepsilon =\varepsilon _{n}^{(m)}$.
Taking $\bar{\Sigma}$ and $\Sigma $-types it induces trivializations 
\begin{equation}
\tau ^{*}\mathcal{L}\simeq \mathcal{O}_{\overline{Ig}(p^{n})_{\mu
}^{(m)}},\,\,\tau ^{*}\mathcal{P}_{\mu }\,\simeq \mathcal{O}_{\overline{Ig}%
(p^{n})_{\mu }^{(m)}}
\end{equation}
and we let $a=a_{n}^{(m)}$ and $\bar{a}=\bar{a}_{n}^{(m)}$ be the sections
corresponding to 1. Of course, the trivialization of $\mathcal{L}$ is the
one that we have met before.

Let $a^{k_{1},k_{2}}=a^{k_{1}}\bar{a}^{k_{2}}.$ Then we may trivialize $\tau
^{*}\mathcal{L}^{(k_{1},k_{2})}$ by $s\mapsto s/a^{k_{1},k_{2}}$ to get a
function on $\overline{Ig}(p^{n})_{\mu }^{(m)}$. This allows us to define,
as usual, canonical Fourier-Jacobi expansion $\widetilde{FJ}(f)$ (along $%
\tilde{C}$), and if we make a choice of a splitting of $\tau :\tilde{C}%
\rightarrow C,$ a Fourier-Jacobi expansion $FJ(f)$ (along $C$) for every $%
f\in M_{k_{1},k_{2}}^{p}(N;\mathcal{K}_{p}).$

\subsubsection{$p$-adic modular forms of $p$-adic bi-weights}

The yoga of $p$-adic weights, either \`{a} la Serre or \`{a} la Katz, allows
us now to define the space 
\begin{equation}
M_{k_{1},k_{2}}^{p}(N;\mathcal{K}_{p})
\end{equation}
of $p$-adic modular forms of any bi-weight $(k_{1},k_{2})\in \frak{X}%
_{p}^{(2)}.$ If we follow Serre, we define them as elements of the Banach
space $\mathcal{FJ}_{p}$ via limits of $p$-adic modular forms of integral
bi-weights. If we follow Katz, we have 
\begin{equation}
M_{k_{1},k_{2}}^{p}(N;\mathcal{O}_{p})=V(\bar{\Sigma}^{k_{1}}\Sigma
^{k_{2}}).
\end{equation}
We let the reader complete the details, which are identical to the case of a
single weight treated before.

\subsection{The theta operator for $p$-adic modular forms}

We are finally able to define the operator $\Theta $ on $p$-adic modular
forms. Compare [Ka3, V.5.8]. Let $f\in M_{k_{1},k_{2}}^{p}(N;\mathcal{O}%
_{p}).$ Assume first that $k_{1}$ and $k_{2}$ are from $\Bbb{Z},$ and reduce
modulo $p^{m},$ to get $f\in H^{0}(\bar{S}_{\mu }^{(m)},\mathcal{L}%
^{k_{1}}\otimes \mathcal{P}_{\mu }^{k_{2}}).$ Take any $n\ge m,$ pull back
to $\overline{Ig}(p^{n})_{\mu }^{(m)},$ divide by $a^{k_{1},k_{2}}$ and
consider 
\begin{equation}
\eta _{f}=d(\tau ^{*}f/a^{k_{1},k_{2}})\in H^{0}(\overline{Ig}(p^{n})_{\mu
}^{(m)},\Omega _{Ig}^{1}).
\end{equation}
Apply $KS^{-1}$ to $\eta _{f}.$ This results in a section of $\mathcal{L}%
\otimes \mathcal{P}.$ As explained before, when we project this section to $%
\mathcal{L}\otimes \mathcal{P}_{\mu }$ we get a section that is holomorphic
along $\tilde{C}$ and even vanishes there (recall $KS$ had a pole along the
cuspidal divisor). Multiply back by $a_{k_{1},k_{2}}$ and use Galois descent
to descend the resulting section to $S_{\mu }^{(m)}.$

We may now take the limit over $m$ to get our $\Theta ,$ if $%
(k_{1},k_{2})\in \Bbb{Z}^{2}.$ A further limit over weights, as in the proof
of Proposition \ref{Katz is Serre}, allows us to extend the definition to $%
(k_{1},k_{2})\in \frak{X}_{p}^{(2)}.$ Using Katz' approach, where the
process of dividing and multiplying back by $a^{k_{1},k_{2}}$ is already
built into the isomorphism with $V(\bar{\Sigma}^{k_{1}}\Sigma ^{k_{2}}),$ $%
\Theta $ is nothing but the map 
\begin{equation}
\Theta :f\mapsto (1\otimes pr_{\mu })\circ KS^{-1}\circ d(f)
\end{equation}
sending $V(\bar{\Sigma}^{k_{1}}\Sigma ^{k_{2}})$ to $V(\bar{\Sigma}%
^{k_{1}+1}\Sigma ^{k_{2}+1}).$ Here $pr_{\mu }:\mathcal{P}\rightarrow 
\mathcal{P}_{\mu }$ is the projection defined over $\bar{S}_{\mu }^{rig}.$

\begin{theorem}
Let $(k_{1},k_{2})\in \frak{X}_{p}^{(2)}.$ The operator 
\begin{equation}
\Theta :M_{k_{1},k_{2}}^{p}(N;\mathcal{O}_{p})\rightarrow
M_{k_{1}+1,k_{2}+1}^{p}(N;\mathcal{O}_{p})
\end{equation}
defined by the above formula, satisfies the following properties (and is
uniquely determined by its effect on $q$-expansions).

(i) When one reduces $M_{k_{1},k_{2}}^{p}(N;\mathcal{O}_{p})$ modulo $p,$
and uses the isomorphism $\mathcal{P}_{\mu }\simeq \mathcal{L}^{p},$ $\Theta 
$ reduces to the operator 
\begin{equation}
\Theta :M_{k}(N;\kappa )\rightarrow M_{k+p+1}(N;\kappa )
\end{equation}
on $\mod p$ modular forms.

(ii) The effect of $\Theta $ on the canonical FJ expansion $\widetilde{FJ}(f)
$ is given by ``$q\frac{d}{dq}",$ i.e. by the formula (\ref{qd/dq}).
\end{theorem}

We omit the proof of (ii), which goes along the same lines as in the $\
mod p$ theory.

\bigskip

\textbf{Bibliography\medskip }

[An-Go] F. Andreatta, E. Z. Goren: Hilbert modular forms: mod $p$ and $p$%
-adic aspects, Memoirs A.M.S. \textbf{819, }2005.\medskip

[Bel] J. Bella\"{i}che: Congruences endoscopiques et repr\'{e}sentations
Galoisiennes, \emph{Th\`{e}se}, Paris XI (Orsay), 2002.\medskip

[B-N] S. B\"{o}cherer, S. Nagaoka: On mod p properties of Siegel modular
forms, Math. Ann. \textbf{338 }(2007), 421-433.\medskip

[Bu-We] O. B\"{u}ltel, T. Wedhorn: Congruence relations for Shimura
varieties associated to some unitary groups, J. Instit. Math. Jussieu 
\textbf{5} (2006), 229-261.\medskip

[Cog] J. Cogdell: Arithmetic cycles on Picard modular surfaces and modular
forms of nebentypus, J. Reine Angew. Math. \textbf{357} (1985),
115-137.\medskip

[Col] R. Coleman: Classical and overconvergent modular forms, Inv. Math. 
\textbf{124} (1996), 215-241.\medskip

[De] P. Deligne: Travaux de Shimura, S\'{e}m. Bourbaki \textbf{389}
(1971).\medskip

[De-Ra] P. Deligne, M. Rapoport: Les sch\'{e}mas de modules de courbes
elliptiques, \emph{in:} \emph{Modular Functions of One Variable} II, LNM 
\textbf{349} (1973), 143-316.\medskip

[dS-G] E. de Shalit, E. Z. Goren: Supersingular curves on Picard modular
surfaces modulo an inert prime, \emph{submitted} (2015).\medskip

[Ei1] E. Eischen: $p$-adic differential operators on automorphic forms on
unitary groups, Annales de l'Institut Fourier, \textbf{62} (2012),
177-243.\medskip

[Ei2] E. Eischen: A $p$-adic Eisenstein measure for unitary groups, J. Reine
Angew. Math. \textbf{699} (2015), 111-142.\medskip

[Fa-Ch] G. Faltings, C.-L. Chai: \emph{Degeneration of abelian varieties, }%
Springer, 1990.\medskip

[Gor] E. Z. Goren: Hilbert modular forms modulo $p^{m}$ -- the unramified
case, J. Number Theory \textbf{90} (2001), 341--375.\medskip

[Gr] B. Gross: A tameness criterion for Galois representations associated to
modular forms (mod $p$), Duke Math. J. \textbf{61 }(1990), 445-517.\medskip

[Ha] M. Harris: Arithmetic vector bundles and automorphic forms on Shimura
varieties. I, Invent. Math. \textbf{82} (1985), 151-189.\medskip

[Hart] R. Hartshorne: \emph{Algebraic Geometry,} Graduate Texts in
Mathematics, No. \textbf{52}, Springer, 1977.\medskip

[Hi1] H. Hida: $p$\emph{-adic automorphic forms on Shimura varieties},
Springer, 2004.\medskip

[Hi2] H. Hida: Irreducibility of the Igusa tower, Acta Mathematica Sinica,
English Series, \textbf{25} (2009), 1-20.\medskip 

[Hs] M.-L. Hsieh: Eisenstein congruence on unitary groups and Iwasawa main
conjectures for CM fields, J. of the AMS, \textbf{27 }(2014),
753-862.\medskip 

[Joc] N. Jochnowitz: A study of the local components of the Hecke algebra
mod $l$, Trans. Amer. Math. Soc. \textbf{270} (1982), 253-267.\medskip

[Ka1] N. Katz: $p$-adic properties of modular schemes and modular forms, 
\emph{in:} \emph{Modular functions of one variable III}, 69-190. LNM \textbf{%
350}, Springer, 1973.\medskip

[Ka2] N. Katz: A result on modular forms in characteristic $p$, \emph{in:} 
\emph{Modular functions of one variable V}, 53-61, LNM \textbf{601},
Springer, 1977.$\medskip $

[Ka3] N. Katz: $p$-adic interpolation of real analytic Eisenstein series,
Ann. of Math. \textbf{104} (1976), 459-571.\medskip

[Ka-Ma] N. Katz, B. Mazur: \emph{Arithmetic moduli of elliptic curves}.
Annals of Mathematics Studies \textbf{108}, Princeton, 1985.\medskip

[Ka-O] N. Katz, T. Oda: On the differentiation of de Rham cohomology classes
with respect to parameters, J. Math. Kyoto Univ. \textbf{8} (1968),
199-213.\medskip

[Kob] N. Koblitz: $p$-adic variation of the zeta-function over families of
varieties defined over finite fields, Compositio Math. \textbf{31} (1975),
119-218.\medskip

[Lan] K.-W. Lan: Arithmetic compactifications of PEL-type Shimura varieties,
London Mathematical Society Monographs \textbf{36}, Princeton, 2013.\medskip

[L-R] R. Langlands, D. Ramakrishnan: \emph{The zeta functions of Picard
modular surfaces}, Univ. Montr\'{e}al, Montr\'{e}al, 1992.\medskip

[La1] M. Larsen: Unitary groups and $l$-adic representations, \emph{Ph.D.
Thesis}, Princeton University, 1988.\medskip

[La2] M. Larsen: Arithmetic compactification of some Shimura surfaces \emph{%
in}: [L-R], 31-45.\medskip

[Mo] B. Moonen: Group schemes with additional structures and Weyl group
cosets \emph{in: Moduli of abelian varieties}, Progr. Math. \textbf{195},
255-298, Birkh\"{a}user, 2001.\medskip

[Mu1] D. Mumford: \emph{Abelian varieties, }Oxford University Press, London
1970.\medskip

[Mu2] D. Mumford: \emph{The red book of varieties and schemes}, second
expanded edition, LNM \textbf{1358, }Springer (1999).\medskip

[Ri] K. Ribet: A. $p$-adic interpolation via Hilbert modular forms, \emph{in:%
} \emph{Algebraic geometry}, Proc. Sympos. Pure Math. \textbf{29}, 581--592,
AMS, 1975.\medskip

[Se] J.-P. Serre: Formes modulaires et fonctions z\^{e}ta $p$-adiques. \emph{%
in:} \emph{Modular functions of one variable III}, LNM \textbf{350},
191-268, Springer, 1973.\medskip

[Sh1] G. Shimura: Arithmetic of unitary groups, Ann. of Math. \textbf{79}
(1964), 369-409.\medskip

[Sh2] G. Shimura: The arithmetic of automorphic forms with respect to a
unitary group, Ann. of Math. \textbf{107} (1978), 569-605.\medskip

[Sh3] G. Shimura: \emph{Arithmeticity in the theory of automorphic forms, }%
Math. Surveys and Monographs \textbf{82, }AMS, 2000.\medskip

[SwD] H.P.F. Swinnerton-Dyer: On $l$-adic representations and congruences
for coefficients of modular forms, \emph{in:} \emph{Modular functions of one
variable III}, LNM \textbf{350}, 1-55, Springer, 1973.\medskip

[V] I. Vollaard: The supersingular locus of the Shimura variety for $GU(1,s)$%
, Canad. J. Math. \textbf{62} (2010), 668-720.\medskip

[We] T. Wedhorn: The dimension of Oort strata of Shimura varieties of
PEL-type, \emph{in: Moduli of abelian varieties,} Progr. Math. \textbf{195},
441-471, Birkh\"{a}user, 2001.\medskip

\end{document}